\documentclass[11pt,reqno]{amsart}
 \usepackage{geometry} 
\usepackage{pinlabel}             %
\usepackage{amssymb, amsthm}
\usepackage{graphicx}
\usepackage{float}
\makeatletter
\renewcommand*{\fps@figure}{htb}
\makeatother

\usepackage{epstopdf}
\allowdisplaybreaks

\usepackage{enumitem}
\usepackage{hyperref}

\usepackage[capitalize,noabbrev]{cleveref}
\hypersetup{pdfpagemode={UseOutlines},
	bookmarksopen=true,
	bookmarksopenlevel=0,
	hypertexnames=true,
	colorlinks=true, %
	citecolor=teal, %
	linkcolor=blue, %
	urlcolor=magenta, %
	pdfstartview={FitV},
	unicode,
	breaklinks=true,
}
\crefformat{enumi}{#2\textup{(#1)}#3}
\crefmultiformat{enumi}{#2\textup{(#1)}#3}{,~#2\textup{(#1)}#3}{,~#2\textup{(#1)}#3}{ and~#2\textup{(#1)}#3}

\usepackage{crossreftools}

\usepackage{mathtools}
\usepackage{tikz-cd}
\usepackage{thm-restate}

\usepackage{xspace}

\numberwithin{equation}{section}

\usepackage{tikz}
\usepackage{pgfplots}
\pgfplotsset{compat=1.18}
\usetikzlibrary{knots}
\usetikzlibrary{arrows}
\usetikzlibrary{calc}
\usetikzlibrary{intersections}
\usetikzlibrary{pgfplots.fillbetween}
\usepackage{pst-knot} %
\usepackage{graphicx, overpic}
\usepackage{subcaption}
\usepackage{microtype}
\usepackage{MnSymbol,wasysym}
\usepackage{soul} %

\usetikzlibrary{decorations.pathreplacing,decorations.markings}
\usetikzlibrary{patterns}
\usetikzlibrary{arrows}
\usepackage{extarrows}
\usetikzlibrary {arrows.meta}
\usetikzlibrary{decorations.markings}

\tikzset{middlearrow/.style={
			decoration={markings,
					mark= at position 0.55 with {\arrow{#1}} ,
				},
			postaction={decorate}
		}
}

\newtheorem{thmintro}{Theorem}
\newtheorem{corollaryintro}{Corollary}
\newtheorem{question}{Question}

\newtheorem{theorem}{Theorem}[section]
\newtheorem{corollary}[theorem]{Corollary}
\newtheorem{lemma}[theorem]{Lemma}

\newtheorem{proposition}[theorem]{Proposition}
\newtheorem*{proposition*}{Proposition}

\theoremstyle{definition}
\newtheorem{definition}[theorem]{Definition}
\newtheorem{remark}[theorem]{Remark}

\newtheorem{example}[theorem]{Example}

\newtheorem{notation}[theorem]{Notation}
\newtheorem{definitionintro}{Definition}

\NewDocumentCommand{\noproof}{}{%
  \pushQED{\qed}\AddToHookNext{env/\UseName{@currenvir}/end}{\popQED}%
}

\newcommand{\N}{\ensuremath{{\mathbb{N}}}}
\newcommand{\Z}{\ensuremath{{\mathbb{Z}}}}
\newcommand{\R}{\ensuremath{{\mathbb{R}}}}

\newcommand{\iso}{\cong}
\newcommand{\epi}{\twoheadrightarrow}

\NewDocumentCommand{\shortexactsequence}{m O{} m O{} m O{}}{
	\begin{tikzcd}[ampersand replacement=\&, column sep=small]
		1 \arrow[r] \& #1 \arrow[r, "#2"] \& #3 \arrow[r, "#4"] \& #5 \arrow[r] \& 1#6
	\end{tikzcd}}

\newcommand{\subgraph}{\subseteq}

\RequirePackage{xparse}
\RequirePackage{ifthen}
\DeclarePairedDelimiter\abs||
\DeclarePairedDelimiter\norm\Vert\Vert
\DeclarePairedDelimiter\ceil\lceil\rceil
\DeclarePairedDelimiter\set\{\}

\NewDocumentCommand{\range}{O{1} m}{%
	\ifthenelse{\equal{#1}{1} \AND \equal{#2}{3}}{%
		\set{1,2,3}%
	}{%
		\ifthenelse{\equal{#1}{1} \AND \equal{#2}{2}}{%
			\set{1,2}%
		}{%
			\set{#1, \dots, #2}%
		}%
	}
}
\DeclarePairedDelimiterX\presentation[2]\langle\rangle{
	#1 \ifblank{#2}{}{ \mid #2}
}

\DeclarePairedDelimiter\cardinality\lvert\rvert

\newcommand\restrict[2]{{%
\left.\kern-\nulldelimiterspace %
#1 %
\vphantom{|} %
\right|_{#2} %
}}

\newcommand{\ee}[1]{\mathrm{E}_{#1}} %
\newcommand{\vv}[1]{\mathrm{V}_{#1}} %
\newcommand{\flag}[1]{\Delta(#1)}
\newcommand{\Link}[1]{\operatorname{Lk}(#1)} 
\newcommand{\Star}[1]{\operatorname{St}(#1)} 
\newcommand{\Pal}[1]{\operatorname{Pal}(#1)}

\newcommand{\supp}[1]{\operatorname{supp}(#1)}

\newcommand{\free}[1]{F(#1)}
\newcommand{\monoid}[1]{(#1 \cup #1^{-1})^\ast}
\newcommand{\freeid}{=_F}

\NewDocumentCommand{\groupid}{m}{=_{#1}}
\newcommand{\identical}{\equiv}

\NewDocumentCommand{\length}{m}{\abs{#1}}

\NewDocumentCommand{\area}{O{} m}{\operatorname{Area}_{#1}(#2)}

\newcommand{\calp}{\mathcal{P}}

\newcommand{\cay}[1]{\operatorname{Cay}(#1)}
\newcommand{\compl}[1]{#1^{\mathrm{c}}} %

\NewDocumentCommand{\finitetype}{o}{\ensuremath{\IfValueTF{#1}{\mathcal F_{#1}}{\mathcal F}}\xspace}
\newcommand{\Dehn}[1]{\delta_{#1}}
\NewDocumentCommand{\D}{s m}{%
    \IfBooleanTF{#1}{%
        \ensuremath{(\mathcal D_{#2})}%
    }{%
    \ifnum0<0#2\relax%
        \ref{D#2}%
    \else%
        \D*{#2}%
    \fi%
    }%
}

\renewcommand{\implies}{\Rightarrow}

\newcommand{\raag}[1]{A_{#1}} %
\newcommand{\raaggens}[1][\Gamma]{S_{#1}}

\newcommand{\raagpres}[1][\Gamma]{\mathcal P_{#1}}

\newcommand{\bbg}[1]{\mathit{BB}_{#1}} %
\newcommand{\bbggens}{S_T}
\newcommand{\bbgpres}[1][\Gamma]{\mathcal Q_{#1}}

\newcommand{\SB}[1]{\mathit{SB}_{#1}}

\NewDocumentCommand{\dehnexp}{}{{d(\Gamma)}}

\NewDocumentCommand{\pushdown}{O{0} m}{\operatorname{push}_{#1}(#2)}
\NewDocumentCommand{\cycpushdown}{O{} m}{\operatorname{cycpush}_{#1}(#2)}

\newcommand{\colword}[2]{[{#1}]_{#2}}

\NewDocumentCommand{\colwordproduct}{O{w} O{1} m}{\colword{#1_{#2}}{a_{#2}} \cdots \colword{#1_{#3}}{a_{#3}}}
\newcommand{\transitionword}[2]{\tau_{#1,#2}}

\newcommand{\colraaggens}[1][\Gamma]{S_{#1}^{\text{col}}} %
\NewDocumentCommand{\balance}{m O{\Gamma} m}{#1(#2#3^{-1})}

\newcommand{\boldword}{\mathbf w}
\newcommand{\bolddiag}{\mathbf D}
\newcommand{\numedges}[1]{\cardinality{E_{#1}}}
\newcommand{\numvertices}[1]{\cardinality{V_{#1}}}
\newcommand{\density}[1]{\rho(#1)}

\NewDocumentCommand{\infrel}{o}{R^{\infty}\IfValueTF{#1}{(#1)}{}}
\NewDocumentCommand{\infcolrel}{o}{R_{\text{col}}^{\infty}\IfValueTF{#1}{(#1)}{}}
\NewDocumentCommand{\infpres}{o}{\calp^{\infty}\IfValueTF{#1}{(#1)}{}}
\NewDocumentCommand{\infcolpres}{o}{\calp_{\text{col}}^{\infty}\IfValueTF{#1}{(#1)}{}}
\newcommand{\height}{\varphi_\Gamma}
\newcommand{\cayheight}{\tilde\varphi_\Gamma}

\newcommand{\flatarea}[1]{\operatorname{Area\!\flat}(#1)}
\newcommand{\flatdehn}{\delta_\flat}
\newcommand{\flatnorm}[1]{\norm{#1}_{\flat}}
\newcommand{\altwords}{\operatorname{Alt}_\Gamma}

\usepackage{xcolor}

\keywords{Dehn functions, filling invariants, Bestvina--Brady groups, subgroups of right-angled Artin groups, asymptotic cones}

\subjclass[2020]{Primary: 20F69.  Secondary: 20F05, 20F38, 20F65, 51F30}

\begin{document}

\title{Complete classification of the Dehn functions of Bestvina--Brady groups}

\author{Yu-Chan Chang}
\address{Department of Mathematics, The Ohio State University, 231 W. 18th Ave., Columbus, OH 43210, USA}
\email{chang.2628@osu.edu}

\author{Jer\'{o}nimo Garc\'{i}a-Mej\'{i}a}
\address{Mathematical Institute, Andrew Wiles Building, Observatory Quarter, University of Oxford, Oxford OX2 6GG, United Kingdom}
\email{jeronimo.garcia-mejia@maths.ox.ac.uk}

\author{Matteo Migliorini}
\address{Faculty of Mathematics, Karlsruhe Institute of Technology, Englerstra\ss e 2, 76131 Karlsruhe, Germany}
\email{matteo.migliorini@kit.edu}

\begin{abstract}
	We prove that the Dehn function of every finitely presented Bestvina--Brady group grows as a linear, quadratic, cubic, or quartic polynomial. In fact, we provide explicit criteria on the defining graph to determine the degree of this polynomial. As a consequence, we identify an obstruction that prevents certain Bestvina--Brady groups from admitting a CAT(0) structure.
\end{abstract}

\maketitle

\section{Introduction}

Bestvina and Brady introduced a class of subgroups of right-angled Artin groups, now known as \emph{Bestvina--Brady groups}, that exhibit exotic finiteness properties \cite{BestvinaBrady}. For every finite simplicial graph $\Gamma$, they defined a group $\bbg \Gamma$ whose finiteness properties are determined by the homotopy type of the flag complex $\flag \Gamma$ associated to $\Gamma$: they proved that $\bbg \Gamma$ is of (finiteness) type $\finitetype[n]$ if and only if $\flag \Gamma$ is $(n-1)$-connected.
Recall that types $\finitetype[1]$ and $\finitetype[2]$ correspond to the more familiar notions of being finitely generated and finitely presented, respectively.

The family of Bestvina--Brady groups includes the Stallings–Bieri groups $\SB n$, which were the first known examples of groups of type $\finitetype[n-1]$ but not of type $\finitetype[n]$. Stallings first constructed $\SB 3$ \cite{Stallings}, and Bieri later generalised this to $\SB n$ for $n > 3$ \cite{Bieri}. Each $\SB n$ can be realised as $\bbg \Gamma$, where $\Gamma$ is the join of $n$ copies of the graph consisting of two non-adjacent vertices.

In this work, we study the Dehn functions of finitely presented Bestvina--Brady groups. The \emph{Dehn function} of a finitely presented group is an important quasi-isometry invariant, which bounds the number of relations that must be applied to reduce a word representing the identity to the trivial word; see \cref{sec:dehn-functions} for a precise definition. In this sense, the Dehn function can be interpreted as a quantitative version of finite presentability. Geometrically, the Dehn function of a finitely presented group $G$ provides an upper bound on the area required to fill a loop in  the universal cover of a finite $2$-complex with fundamental group $G$. Notably, a group is hyperbolic if and only if its Dehn function is linear \cite{gromovsub}, whereas any finitely presented non-hyperbolic group has at least quadratic Dehn function \cite{gromovsub,Bowditchsubquad}.

Dison proved that the Dehn function of a finitely presented Bestvina--Brady group is bounded above by a quartic polynomial \cite{Dison}; see \cite{PushingRAAGs} for a purely geometric proof. This naturally raises the question:

\begin{question}\label{question:polynomial-dehn}
	Is the Dehn function of every finitely presented Bestvina--Brady group equivalent to a polynomial?
\end{question}

Brady presented examples of quadratic, cubic, and quartic Dehn functions \cite{Geometryfg}. Subsequently, Dison, Elder, Riley, and Young showed that Stallings' group $\SB 3$ has quadratic Dehn function \cite{WMTYStallings2009}. Later, Carter and Forester, using geometric arguments, generalised this result by showing that $\SB n$ has quadratic Dehn function for every $n \geq 3$ \cite{CarterForester}. In fact, more generally, they showed that if $\Gamma$ is a join of three nonempty graphs, then the Dehn function of $\bbg\Gamma$ is quadratic. More recently, Chang proved that the Dehn functions of certain Bestvina--Brady groups can be determined directly from the defining graphs \cite{chang2021dehnbb}.

In the present work, we complete the picture giving a positive answer to \cref{question:polynomial-dehn}, by showing that the Dehn function of every finitely presented Bestvina--Brady group is equivalent to a polynomial of degree $d \in \{1,2,3,4\}$.
Furthermore, we provide conditions on the defining graph of a finitely presented Bestvina--Brady group that determine the precise degree $d$; see \cref{main-theorem}. In this sense, we obtain an effective classification of finitely presented Bestvina--Brady groups in terms of their Dehn functions.

Besides providing an answer to \cref{question:polynomial-dehn}, our main result contributes to the understanding of the Dehn functions of subgroups of right-angled Artin groups; see \cite{Bridson13,bradysoroko2019dehn} for more examples other than Bestvina--Brady groups.

\subsection{Main result}

It is known that $\bbg\Gamma$ is hyperbolic if and only if $\Gamma$ is a tree, as can be seen directly by looking at the presentation given in \cite{DickLearyPres}. This characterises all the graphs $\Gamma$ such that $\bbg\Gamma$ has linear Dehn function. However, providing a general criterion that distinguishes Bestvina--Brady groups with quadratic, cubic, and quartic Dehn functions is a much more challenging task.

Recall that a graph $\Gamma$ is called \emph{reducible} if it is a join $\Gamma_1 \ast \Gamma_2$, where $\Gamma_1$ and $\Gamma_2$ are nonempty subgraphs\footnote{All subgraphs we consider are assumed to be \emph{induced} subgraphs.}; otherwise, it is called \emph{irreducible}.
Our main result gives a criterion on the defining graph $\Gamma$ that determines the precise Dehn function of $\bbg \Gamma$ in terms of the \emph{maximal reducible subgraphs} of $\Gamma$, that is, reducible subgraphs that are maximal with respect to inclusion among all reducible subgraphs of $\Gamma$.
To state this full characterisation, we introduce the following definitions.

\begin{definitionintro}[Essentially $2$-reducible graph]\label{def:essential}
	A graph is called \emph{essentially $2$-reducible} if it is a join of two irreducible subgraphs, each with at least two vertices; equivalently, if it is neither a cone graph nor a join of three or more subgraphs.
\end{definitionintro}

\begin{definitionintro}
	Let $\Gamma$ be a finite simplicial graph whose associated flag complex $\flag\Gamma$ is simply connected. If this is the case, we say that $\Gamma$ has property \crtcrossreflabel{\D*1}[D1], and additionally, it has property
	\begin{enumerate}
		\item[{\crtcrossreflabel{\D*2}[D2]}] if $\Gamma$ is not a tree;
		\item[{\crtcrossreflabel{\D*3}[D3]}] if $\Gamma$ contains a maximal reducible subgraph that is essentially $2$-reducible;
		\item[{\crtcrossreflabel{\D*4}[D4]}] if $\Gamma$ contains a maximal reducible subgraph whose associated flag complex is not simply connected.
	\end{enumerate}
\end{definitionintro}

We are now ready to state our main result.

\begin{thmintro}\label{main-theorem}
	Let $\Gamma$ be a finite simplicial graph such that the associated flag complex $\flag\Gamma$ is simply connected. Let $\dehnexp$ be the maximal $\alpha \in \{1,2,3,4\} $ such that $\Gamma$ has property \D {\alpha}. Then the Dehn function of the Bestvina--Brady group $\bbg \Gamma$ satisfies
	\[
		\Dehn{\bbg \Gamma}(n) \asymp n^{\dehnexp}.
	\]
\end{thmintro}

Therefore, property $\D \alpha$ is the obstruction to $\bbg \Gamma$ having Dehn function strictly less than a polynomial of degree $\alpha$. This can be seen by combining \cref{main-theorem} with the fact that property
$\D \beta$ implies property $\D \alpha$ for $\beta>\alpha$ on every graph $\Gamma$ with $\flag\Gamma$ simply connected.
To see this, note that if $\Gamma$ contains a reducible subgraph $\Lambda = \Lambda' \ast \Lambda''$ such that $\flag \Lambda$ is not simply connected, then both $\Lambda$ and $\Lambda'$ are disconnected, and therefore irreducible with at least two vertices. This implies that $\Lambda$ is essentially $2$-reducible, so $\D4 \implies \D3$. It is also clear that if $\Gamma$ contains an essentially $2$-reducible subgraph, then it contains a square and cannot be a tree, so $\D3 \implies \D2$. Finally, \D1 is satisfied by definition.

If the defining graph $\Gamma$ is reducible, then it coincides with its unique maximal reducible subgraph, meaning that $\Gamma$ cannot have \D4. Thus, \cref{main-theorem} implies the following result.

\begin{corollaryintro}\label{main-theorem-reducible-case}
	If $\Gamma$ is an essentially $2$-reducible graph such that $\flag \Gamma$ is simply connected, then $\bbg \Gamma$ has cubic Dehn function.
\end{corollaryintro}

If $\Gamma$ is reducible but not essentially $2$-reducible, then the Dehn function of $\bbg\Gamma$ is at most quadratic. While this is a simple application of \cref{main-theorem}, it also follows from previously known results. Indeed, a reducible graph $\Gamma$ that is not essentially $2$-reducible is either a cone graph, namely, $\Gamma = \{c\} * \Lambda$ for some subgraph $\Lambda$, or it splits as a join of three or more subgraphs. In the first case, the group $\bbg \Gamma$ is isomorphic to the right-angled Artin group $\raag \Lambda$, which is a $\mathrm{CAT}(0)$ group \footnote{A \emph{$\mathrm{CAT}(0)$ group} is a group admitting a proper and cocompact action on a $\mathrm{CAT}(0)$ space by isometries.} \cite[Theorem 2.6]{charney-intro-raags}, and therefore has at most quadratic Dehn function \cite[Theorem 6.2.1]{bridsongeometry}. In the second case, the Dehn function of $\bbg \Gamma$ being quadratic is a consequence of \cite[Corollary 4.3]{CarterForester}.
Note that although we draw some inspiration from the work of Carter and Forester, our argument does not rely on their results.

\Cref{main-theorem} is also consistent with \cite[Theorem 1.1]{chang2021dehnbb}. Indeed, let $\flag\Gamma$ be a $2$-dimensional triangulated disc with square boundary. On the one hand, if $\Gamma$ is reducible, then it is the suspension of a path. In this case, according to both \cref{main-theorem} and \cite[Theorem 1.1]{chang2021dehnbb}, the Dehn function is cubic if the path has length greater than or equal to $3$ and quadratic otherwise. On the other hand, if $\Gamma$ is irreducible, then it has $\D4$, as every reducible subgraph $\Lambda$ containing the boundary of $\flag\Gamma$ must be strictly contained in $\Gamma$. Therefore, the boundary of $\flag\Gamma$ is not null-homotopic in $\flag\Lambda$. Furthermore, the flag complex $\flag\Gamma$ must contain an interior $2$-simplex (otherwise, one gets a contradiction by \cite[Lemma 4.2]{chang2021dehnbb}). Thus, both statements imply that $\bbg\Gamma$ has quartic Dehn function.
However, note that the lower bound established in \cite[Proposition 1.5]{chang2021dehnbb} is incorrect.\footnote{For example, let $\Gamma=\{c\}\ast\Lambda$, where $\Lambda$ is the suspension of the path of length three. Then the Dehn function of $\bbg\Lambda$ is cubic, whereas the Dehn function of $\bbg\Gamma$ is quadratic since $\Gamma$ is a cone graph. The gap in the proof of \cite[Proposition 1.5]{chang2021dehnbb} is due to \cite[Proposition 5.3]{chang2021dehnbb}, which is erroneous.}

\subsection{Applications of the main result}

Showing that the Dehn function of a group is strictly greater than quadratic yields geometric constraints on the group. In particular, since $\mathrm{CAT}(0)$ groups are finitely presented and have at most quadratic Dehn function, condition \D 3 provides an obstruction to a Bestvina--Brady group acting properly and cocompactly on a CAT(0) space.

\begin{corollaryintro}\label{cor:D3-not-cat0}
	If a finite simplicial graph $\Gamma$ has \D 3, then $\bbg\Gamma$ is not a $\mathrm{CAT}(0)$ group.
\end{corollaryintro}

In general, not every finitely presented Bestvina--Brady group with quadratic Dehn function is a $\mathrm{CAT}(0)$ group. Indeed, Stallings--Bieri groups have quadratic Dehn function, but they are not $\mathrm{CAT}(0)$ since they are not of type $\finitetype[\infty]$ \cite{BestvinaBrady}. This raises the following question:
\begin{question}\label{question-cat0}
	Does there exist a Bestvina--Brady group of type $\finitetype[\infty]$ with quadratic Dehn function that is not a $\mathrm{CAT}(0)$ group?
\end{question}
To the best of our knowledge, no such examples are known in the literature. One possible approach to finding such examples would be looking at higher filling functions. Informally, the $k$-dimensional Dehn function of a group measures the difficulty of filling $k$-spheres with $(k+1)$-balls in some suitable space on which the group acts; the $1$-dimensional Dehn function is the usual Dehn function. We refer the reader to \cite{PushingRAAGs} for a precise definition of $k$-dimensional Dehn functions. For a $\mathrm{CAT}(0)$ group, the $k$-dimensional Dehn function is at most $n^{(k+1)/k}$ \cite{Gromov-fillings,Wenger}. In \cite{PushingRAAGs}, the authors showed that for a Bestvina--Brady group of type $\finitetype[k+1]$, the $k$-dimensional Dehn function is bounded above by $n^{2(k+1)/k}$.

Finally, combining the result of \cite{Papasoglu} with \cref{main-theorem} yields the following result about the asymptotic cones of Bestvina--Brady groups.

\begin{corollaryintro}\label{cor: asymptotic cones}
	Let $\Gamma$ be a finite simplicial graph such that $\flag\Gamma$ is simply connected. If $\Gamma$ does not have \D 3, then the asymptotic cones of $\bbg \Gamma$ are simply connected.
\end{corollaryintro}

It would be interesting to know if there exist finitely presented Bestvina--Brady groups with non-quadratic Dehn functions that have simply connected asymptotic cones.

\subsection{Examples}\label{subsec: examples}

We present some examples to illustrate how to concretely apply \cref{main-theorem} to determine the Dehn function of a Bestvina--Brady group. Consider the graphs $\Gamma_1$, $\Gamma_2$, and $\Gamma_3 $ shown in \cref{fig:examples}. The graphs $\Gamma_1$ and $\Gamma_2$ were considered by Brady in \cite{Geometryfg}, where he constructed Bestvina--Brady groups with cubic and quartic Dehn functions, respectively. \Cref{main-theorem} recovers these cases. Although the graph $\Gamma_3$ is a slight variation of $\Gamma_1$, the Dehn function of the corresponding Bestvina--Brady group was, to the best of our knowledge, not previously known. In particular, it was not addressed in \cite{CarterForester} or \cite{chang2021dehnbb}.

\newboolean{quartic}
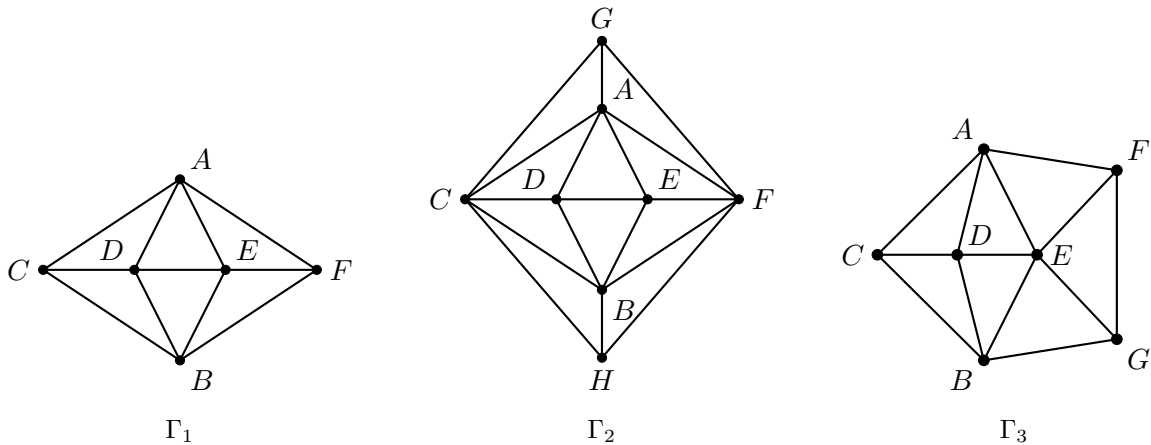
\begin{figure}
	\captionsetup[subfigure]{labelformat=empty, justification=centering}
	\centering
	\begin{subfigure}{0.3\textwidth}
		\begin{tikzpicture}[scale=.6]
	\coordinate (A) at (0,2);
	\coordinate (B) at (0,-2);
	\coordinate (C) at (-3,0);
	\coordinate (D) at (-1,0);
	\coordinate (E) at (1,0);
	\coordinate (F) at (3,0);
	\coordinate (G) at (0,3.5);
	\coordinate (H) at (0,-3.5);

	\foreach \p in {(A),(B)} {
			\foreach \q in {(C), (D), (E), (F)} {
					\draw[thick] \p -- \q;
				}
		}

	\draw[thick] (C) -- (D) -- (E) -- (F);
	\foreach \p in {(A),(B),(C),(D),(E),(F)} {
			\draw [fill] \p circle [radius=0.1];
		}

	\node[above right] at (A) {$A$};
	\node[below right] at (B) {$B$};
	\node[left=1] at (C) {$C$};
	\node[above left] at (D) {$D$};
	\node[above right] at (E) {$E$};
	\node[right=1] at (F) {$F$};

	\ifquartic
		\draw[thick]
		(G) -- (C)
		(G) -- (A)
		(G) -- (F)

		(H) -- (C)
		(H) -- (B)
		(H) -- (F)
		;

		\draw [fill] (G) circle [radius=0.1];
		\draw [fill] (H) circle [radius=0.1];
		\node[above=1] at (G) {$G$};
		\node[below=1] at (H) {$H$};
	\fi
\end{tikzpicture}
		\caption{$\Gamma_1$}
	\end{subfigure}
	\hfil \quad
	\setboolean{quartic}{true}
	\begin{subfigure}{0.3\textwidth}
		\begin{tikzpicture}[scale=.6]
	\coordinate (A) at (0,2);
	\coordinate (B) at (0,-2);
	\coordinate (C) at (-3,0);
	\coordinate (D) at (-1,0);
	\coordinate (E) at (1,0);
	\coordinate (F) at (3,0);
	\coordinate (G) at (0,3.5);
	\coordinate (H) at (0,-3.5);

	\foreach \p in {(A),(B)} {
			\foreach \q in {(C), (D), (E), (F)} {
					\draw[thick] \p -- \q;
				}
		}

	\draw[thick] (C) -- (D) -- (E) -- (F);
	\foreach \p in {(A),(B),(C),(D),(E),(F)} {
			\draw [fill] \p circle [radius=0.1];
		}

	\node[above right] at (A) {$A$};
	\node[below right] at (B) {$B$};
	\node[left=1] at (C) {$C$};
	\node[above left] at (D) {$D$};
	\node[above right] at (E) {$E$};
	\node[right=1] at (F) {$F$};

	\ifquartic
		\draw[thick]
		(G) -- (C)
		(G) -- (A)
		(G) -- (F)

		(H) -- (C)
		(H) -- (B)
		(H) -- (F)
		;

		\draw [fill] (G) circle [radius=0.1];
		\draw [fill] (H) circle [radius=0.1];
		\node[above=1] at (G) {$G$};
		\node[below=1] at (H) {$H$};
	\fi
\end{tikzpicture}
		\caption{$\Gamma_2$}
	\end{subfigure}
	\hfil \quad
	\begin{subfigure}{0.3\textwidth}
		\begin{tikzpicture}[scale=.7]
	\coordinate (A) at (-.5,0);
	\coordinate (B) at (1,0);
	\coordinate (C) at (-2,0);
	\coordinate (D) at (0,2);
	\coordinate (E) at (0,-2);
	\coordinate (F) at (2.5,1.6);
	\coordinate (G) at (2.5,-1.6);

	\foreach \p in {(B), (C), (D), (E)} {
			\draw[thick] (A) -- \p;
		}

	\foreach \p in {(D), (E), (F), (G)} {
			\draw[thick] (B) -- \p;
		}

	\foreach \p in {(A),(B),(C),(D),(E),(F), (G)} {
			\draw [fill] \p circle [radius=0.1];
		}

	\draw[thick] (C) -- (D) -- (F) -- (G) -- (E) -- (C);

	\node[above right] at (A) {$D$};
	\node[right=1] at (B) {$E$};
	\node[left=1] at (C) {$C$};
	\node[above left] at (D) {$A$};
	\node[below left] at (E) {$B$};
	\node[above right] at (F) {$F$};
	\node[below right] at (G) {$G$};

\end{tikzpicture}
		\caption{$\Gamma_3$}
	\end{subfigure}
	\caption{Brady's examples with cubic (left) and quartic (centre) Dehn functions. On the right, a variation of Brady's example whose Dehn function is quadratic.}
	\label{fig:examples}
\end{figure}

\emph{The graph $\Gamma_1$.}

The graph $\Gamma_1$ is the suspension of the path of length three induced by the vertices $C$, $D$, $E$, and $F$. Therefore, it is essentially $2$-reducible, so \cref{main-theorem-reducible-case} implies that $\bbg{\Gamma_1}$ has cubic Dehn function.

\emph{The graph $\Gamma_2$.} In this case, the graph $\Gamma_2$ has \D4. To see this, let $\Lambda$ be the join of $\set{C, F}$ and the subgraph induced by the vertices $G$, $A$, $B$, and $H$. Then $\Lambda$ is a maximal reducible subgraph of $\Gamma_2$, and $\flag\Lambda$ is not simply connected.
Therefore, the graph $\Gamma_2$ has \D4, and \cref{main-theorem} implies that the Dehn function of $\bbg{\Gamma_2}$ is quartic.

\emph{The graph $\Gamma_3$.} Clearly, the graph $\Gamma_3$ has \D2. Nevertheless, it does not have \D3, since none of the maximal reducible subgraphs is essentially $2$-reducible. In fact, all the maximal reducible subgraphs are cone graphs. The graph $\Gamma_3$ does contain a unique essentially $2$-reducible subgraph, namely, the square induced by the vertices $A$, $B$, $C$, and $E$, but it is not maximal. Therefore, by \cref{main-theorem}, the Dehn function of $\bbg {\Gamma_3}$ is quadratic.

\subsection{Strategy of the proof}\label{subsec:strategy-of-the-proof}

In \cref{main-theorem}, the linear case corresponds to $\bbg \Gamma$ being hyperbolic, while the quadratic lower bound and the quartic upper bound follow respectively from $\bbg\Gamma$ not being hyperbolic and from \cite[Theorem 1.1]{Dison}. Therefore, to obtain \cref{main-theorem}, it suffices to show the following statement, which involves only \D3 and \D4.

\begin{thmintro}\label{main-theorem-reduction}
	Let $\Gamma$ be a finite simplicial graph such that $\flag\Gamma$ is simply connected, and let $\alpha \in \set{3,4}$. If $\Gamma$ has property \D\alpha, then $\Dehn{\bbg\Gamma}(n) \succcurlyeq n^\alpha$. Otherwise, $\Dehn{\bbg\Gamma}(n) \preccurlyeq n^{\alpha-1}$.
\end{thmintro}

The proof of \cref{main-theorem-reduction} is divided into two parts: the quadratic and cubic upper bounds are given by \cref{upper-bounds}, and the cubic and quartic lower bounds are given in \cref{lower-bounds}.

From Bestvina and Brady's construction \cite{BestvinaBrady}, we can naturally view $\bbg \Gamma$ embedded in $\raag \Gamma$ as the $0$-level set of the height function $\height\colon\raag\Gamma\to\Z$, which sends every generator to $1$. This geometric perspective allows us, via a standard argument, to compute the Dehn function by considering null-homotopic words and van Kampen diagrams in the right-angled Artin group that have heights close to $0$. We formalise this by defining \emph{alternating words} and \emph{almost-flat van Kampen diagrams}.
This is particularly useful for establishing the lower bound: we show that the height of a vertex in a van Kampen diagram can be computed by looking at \emph{corridors} and \emph{annuli}, and the restriction on the height yields a lower bound on the number of intersections between annuli, and hence, on the area of the van Kampen diagram.

The proof of the upper bound employs a refined version of Dison's pushdown argument for the general quartic upper bound \cite{Dison}. Dison's argument does not yield an optimal upper bound: it is possible for a null-homotopic word to have support contained in a cone subgraph, in which case it has at most quadratic area, although the ``pushdown algorithm'' may not recognise this and produce an inefficient filling.

We address this issue by introducing the notion of $k$-\emph{coloured words}. A $k$-coloured word should be thought of as a word $w$ in the right-angled Artin group $\raag\Gamma$ with the usual generating set given by the vertices of $\Gamma$, together with a decomposition  $w = w_1 \cdots w_k$ and a sequence of generators $a_1, \dots, a_k$ of $\raag\Gamma$, such that every letter of $w_i$ commutes with $a_i$. We think of $a_i$ as colouring the subword $w_i$. We are particularly interested in the case where the coloured word has minimal $k$ among all coloured words representing a fixed group element in $\bbg \Gamma$.

The colouring information allows us to push coloured words down to the $0$-level set and to exploit the fact that some subwords have support contained in cone subgraphs. We show that this technique produces the optimal upper bounds for the Dehn function of $\bbg\Gamma$.

Throughout the proof of the upper bound, the case where $\Gamma$ is irreducible is the most challenging. Indeed, if $\Gamma$ is reducible, then any group element can be represented by a $2$-coloured word. This is a useful property that can be used to produce concrete triangular diagrams, as demonstrated in \cite{CarterForester}. However, for a general graph $\Gamma$, there is no uniform bound on the minimal number of colours required to represent an element.
We overcome this by using corridors in a van Kampen diagram over $\raag\Gamma$ to produce a suitable subdivision of triangular diagrams into regions that have a uniformly bounded number of colours, while keeping a good control over the area.

\subsection{Structure of the paper}

\Cref{sec:prelims} includes basic definitions and properties used throughout the paper.
In \cref{sec:alternating}, we define \emph{alternating words} and \emph{almost-flat van Kampen diagrams}, and show how they can be used to compute the Dehn function of Bestvina--Brady groups.
In \cref{sec:coloured-words}, we introduce \emph{coloured words} and show how the colouring allows us to ``push down'' coloured words to the $0$-level set in a canonical way to obtain alternating words. The entire \cref{sec:upper-bounds} is devoted to the proof of \cref{upper-bounds}, which establishes the upper bounds in \cref{main-theorem-reduction}. We establish the lower bounds in \cref{main-theorem-reduction} by proving \cref{lower-bounds} in \cref{sec:lower-bounds}.

\subsection*{Acknowledgements}
We are thankful for the environment provided by the 2023 thematic program \textit{Geometric Group Theory} at the Centre de Recherches Math\'{e}matiques, Universit\'{e} de Montr\'{e}al. We are grateful to Claudio Llosa Isenrich and Matt Zaremsky for helpful comments and suggestions on an earlier version of this work. We thank Claudio Llosa Isenrich for also raising interesting questions related to \cref
{cor:D3-not-cat0} and \cref{cor: asymptotic cones}. The second author is supported by the Royal Society of Great Britain. The second and third authors acknowledge funding from the DFG 281869850 (RTG 2229).
We thank the anonymous referee for their careful reading and for their valuable comments and suggestions, including the remark about higher Dehn functions in the discussion about \cref{question-cat0}.

\section{Preliminaries}\label{sec:prelims}

In this section, we introduce basic definitions, present well-known results, and establish general terminology and notation used throughout.

\subsection{Graphs}\label{sec:graphs}
A \emph{simplicial graph} is a graph without loops or multiple edges.
Given a finite simplicial graph $\Gamma$, its vertex set and edge set are denoted by $\vv\Gamma$ and $\ee\Gamma$, respectively. Two vertices $u$ and $v$ are \emph{adjacent} if they are connected by an edge, in which case we write $\{u,v\}\in\ee\Gamma$. An \emph{(induced) subgraph} $ \Lambda \subgraph \Gamma $ is a graph with $\vv\Lambda\subseteq \vv\Gamma$ and whose edge set $\ee\Lambda$ consists of all edges of $\Gamma$ that connect vertices in $\vv\Lambda$. Recall that all subgraphs considered in the present work are assumed to be induced subgraphs.

The flag complex associated to $\Gamma$, denoted by $\flag\Gamma$, is the simplicial complex whose $1$-skeleton is $ \Gamma $, and in which every complete subgraph of $\Gamma$ spans a simplex.

Let $u,v\in\vv\Gamma$. A \emph{(combinatorial) path} between $u$ and $v$ of length $k \in \N$ is a map $\gamma \colon \range[0]k \to \vv\Gamma$ such that $\gamma(i-1)$ is adjacent to $\gamma(i)$ for all $i \in \range k$.
The \emph{distance} between $u$ and $v$ in $\Gamma$ is the length of a shortest path between them.

The \emph{join} of two graphs $\Gamma_1$ and $\Gamma_2$, denoted by $\Gamma_1\ast\Gamma_2$, is the graph whose vertex set is $\vv{\Gamma_1}\cup\vv{\Gamma_2}$ and whose edge set is $\ee{\Gamma_1}\cup\ee{\Gamma_2}\cup\{\set{u,v} \ \vert \ u\in\vv{\Gamma_1}, v\in\vv{\Gamma_2}\}$. A graph is said to be \emph{reducible} if it decomposes as a join of two nonempty subgraphs, and \emph{irreducible} otherwise.

The \emph{link} of a vertex $v\in \vv\Gamma$, denoted by $\Link v$, is the subgraph induced by the vertices adjacent to $v$. The \emph{star} of $v$, denoted by $\Star v$, is the subgraph $\{v\}\ast\Link v$.

The \emph{complement} of $\Gamma$ is the graph $\compl{\Gamma}$ with the same vertex set as $\Gamma$, such that two vertices are adjacent in $\compl{\Gamma}$ if and only if they are not adjacent in $\Gamma$. Note that $\Gamma$ is irreducible if and only if its complement is connected. More generally, the graph $\Gamma$ admits a unique decomposition (up to the order) $\Gamma = \Gamma_1 \ast \dots \ast \Gamma_k$, where $\Gamma_1, \dots, \Gamma_k$ are irreducible subgraphs of $\Gamma$; these can be characterised as the connected components of the complement of $\Gamma$.

Let $ \Lambda$ be a reducible subgraph of $ \Gamma $. We say that $ \Lambda $ is \emph{maximal} if it is maximal with respect to inclusion among all reducible subgraphs of $ \Gamma $.
We say that $\Lambda$ is \emph{essentially $2$-reducible} if it is a join $\Lambda_1 \ast \Lambda_2$ of exactly two irreducible subgraphs, each of which has at least two vertices. Equivalently, an essentially $2$-reducible graph is a graph whose complement has exactly two connected components, neither of which is a single vertex.

\subsection{Words and free groups}\label{sec:words-and-free-groups}

For a set $S$, we denote by $S^{-1}$ the set of formal inverses of the elements in $S$. By a \emph{word $w$ in $S$} we mean an element of the free monoid $(S \cup S^{-1})^\ast$ generated by $S \cup S^{-1}$; that is, $w$ can be written as $ s_1 \cdots s_k$ for some $s_i \in S \cup S^{-1}$. In this case, the elements $s_1, \dots, s_k$ are called the \emph{letters of $w$}, and $k$ is called the \emph{length} of $w$, denoted by $\length w$. A \emph{subword of $w$} is a word of the form $s_{i} s_{i+1} \cdots s_j$ for some $1 \leq i \leq j \leq k$. A word in $S$ is \emph{freely reduced} if it does not contain subwords of the form $ss^{-1}$ or $s^{-1}s$. 

We denote the \emph{free group on a set $S$} by $\free S$. There is a natural epimorphism of monoids $\monoid S \epi \free S$ that sends a word $w$ to its equivalence class under \emph{free insertions and reductions}, that is, up to inserting or removing subwords of the form $ss^{-1}$ or $s^{-1}s$ in $w$. Two words $w$ and $w'$ in $S$ are said to be \emph{freely equivalent}, written $w \freeid w'$, if their images in $\free S$ coincide. We write $w \identical w'$ when words are actually identical ``letter by letter''.

\subsection{Groups and presentations}\label{sec:groups-and-presentations}

An (abstract) generating set for a group $G$ is a set $S$ equipped with a map $\iota\colon S \to G$ such that the induced map $\free S \to G$ is an epimorphism. Note that we do not require the map $\iota$ to be injective.
Let $R$ be a set of words in $S$. We denote by $\presentation SR$ the group defined as the quotient of $\free S$ by the normal closure of the image of $R$ in $\free S$. Given the canonical map $S \to \presentation SR$, the set $S$ is naturally a generating set for $\presentation SR$. We say that $\presentation SR$ is \emph{generated by $S$ with relations $R$}.

If $G$ is a group together with an isomorphism $G \iso \presentation SR$, we say that \emph{$\presentation SR$ is a presentation for $G$}.  When $S$ and $R$ are finite, we say that the \emph{presentation is finite} or that \emph{$G$ is finitely presented}.

Given a group $G$ with presentation $\presentation SR$, we say that two words $w$ and $w'$ in $S$ \emph{represent the same element in $G$}, written $w \groupid G w'$, if their images under the canonical surjective map $\monoid S \epi G$ coincide.

\subsection{Right-angled Artin groups and Bestvina--Brady groups}\label{sec:raags-bbgs}

Let $\Gamma$ be a finite simplicial graph, and let $\raaggens$ be a set in one-to-one correspondence with the vertex set $\vv\Gamma$; we denote the latter as $\{v_s : s \in \raaggens \}$.
The \emph{right-angled Artin group} $\raag \Gamma$ associated to $\Gamma$ is the group with finite presentation\footnote{We use the convention $[x,y]= xy x^{-1} y^{-1}$.}
\[
	\raagpres \coloneq \langle \raaggens \mid [s,t]=1 \ \text{for all} \ \{v_s,v_{t}\} \in \ee \Gamma \rangle.
\]
The set $\raaggens$ and the presentation $\raagpres$ are called the \emph{standard generating set} and the \emph{standard presentation} of $ \raag \Gamma $, respectively.

If $\Lambda$ is a subgraph of $\Gamma$, then $\raag\Lambda$ is naturally a subgroup of $\raag\Gamma$. We say that $ \raag\Gamma $ is \emph{irreducible} if there are no nontrivial subgraphs $ \Lambda_1, \Lambda_2 \subgraph \Gamma $ such that $ \raag\Gamma $ splits as $ \raag{\Lambda_1} \times \raag {\Lambda_2} $; this is equivalent to saying that the graph $\Gamma$ is irreducible.

Recall that associated to a right-angled Artin group $\raag\Gamma$, there is a natural epimorphism $\height\colon\raag\Gamma\to\Z$, called the \emph{height function}, defined by sending every generator $s\in\raaggens$ to $1$.
The \emph{Bestvina--Brady group} $\bbg \Gamma$ associated to $\Gamma$ is the kernel of the height function.

Bestvina and Brady \cite{BestvinaBrady} showed that $\bbg{\Gamma}$ is finitely presented if and only if the flag complex $ \flag\Gamma $ is simply connected. In this case, Dicks and Leary \cite{DickLearyPres} provided an algebraic proof of this result by exhibiting a finite presentation, which was further simplified by Papadima and Suciu \cite{papadimasuciu2007BBGs}. We now describe their presentation.

Fix a total ordering $\ll$ on the vertex set $\vv\Gamma$, orient the edges of $\Gamma$ increasingly, and choose a spanning tree $T$ of $\Gamma$. The generating set for $\bbg\Gamma$ is the set of edges of $T$, denoted by $\bbggens$.

Let $v_s$ and $v_t$ be two vertices of $\Gamma$. The unique simple path in $T$ from $v_s$ to $v_t$ defines a word $w_{s,t} \coloneq e^{\epsilon_1}_1\cdots e^{\epsilon_k}_k$ in $\bbggens$, where $e_i$ is the $i$th edge in the path, and $\epsilon_i = 1$ if the traversed direction agrees with the orientation of $e_i$, and $\epsilon_i = -1$ otherwise. If $e = (v_s,v_t)$ is an oriented edge, we also write $w_e \coloneq w_{s,t}$.

A \emph{directed triangle} is a triple of oriented edges $(e,f,g)$ in $\Gamma$ with $e = (v_{s_1}, v_{s_2})$, $ f = (v_{s_2}, v_{s_3}) $, and $g = (v_{s_1}, v_{s_3})$ for some vertices $v_{s_1} \ll v_{s_2} \ll v_{s_3}$. The presentation for $\bbg\Gamma$ provided by Papadima and Suciu can be expressed as
\[
	\bbgpres = \presentation{\bbggens}{[w_e,w_f] =1 \ \text{for all directed triangles $(e,f,g)$}}.
\]

\subsection{Dehn functions}\label{sec:dehn-functions}

Let $\calp :=  \presentation S R$ be a finite presentation for a group $G$.
We say that a word $w$ in $S$ is \emph{null-homotopic} if it represents the identity in $G$. The \emph{area with respect to $\calp$} of a null-homotopic word $w$ in $S$ is defined as
\[
	\area[\calp]{w}  := \min \left\{ \ell \in \mathbb{N} : w \freeid \prod_{i = 1}^{\ell} x_{i} r _{i} ^{\epsilon _{i}} x_{i} ^{-1}, \, \ \text{where} \  x_i \in \monoid{S}, \, r_i \in R, \, \ \text{and} \ \epsilon_i=\pm1 \right\}.
\]
The \emph{Dehn function of a finite presentation $\calp$} is the function $\delta_\calp \colon \mathbb{N} \rightarrow \mathbb{N}$ defined as
\[
	\Dehn\calp (n) := \max\{\area[\calp]{w} : w \ \text{null-homotopic word in} \ S \ \text{such that} \ \length w \leq n \}.
\]
Changing finite presentations of $G$ does not alter the asymptotic behaviour of the Dehn function of the presentations. More precisely, the \emph{Dehn function $\Dehn{G}$ of $G$} is well-defined up to \emph{$\asymp$-equivalence} of functions: two functions $f,g \colon \mathbb{N} \to \mathbb{N}$ are \emph{$\asymp$-equivalent}, and we write $f\asymp g$, if $f\preccurlyeq g$ and $g\preccurlyeq f$, where $f \preccurlyeq g$ means that there exist constants $A,B >0$ and $C,D,E \geq 0$ such that $f(n) \leq Ag(Bn +C) + Dn +E$ for all $n \geq 0$. If $\calp_1$ and $\calp_2$ are two finite presentations for $G$, then $\Dehn{\calp_1} \asymp \Dehn{\calp_2}$. Thus, the Dehn function of $G$ is defined to be the $\asymp$-equivalence class $\delta_\calp$ for some (hence any) finite presentation $\calp$ of $G$.

\subsection{Van Kampen diagrams}\label{sec:diagrams}

The study of van Kampen diagrams over a finite group presentation provides a powerful geometric tool for computing the areas of null-homotopic words, and thus for establishing bounds on the Dehn function.

Let  $\calp = \presentation S R$ be a finite presentation; we may always assume that the relation set $R$ is closed under taking cyclic conjugates and inverses.

Given a null-homotopic word $w$ in $\calp$, a \emph{van Kampen diagram} for $w$ is a connected, pointed, oriented, and labelled planar graph $D$ such that each edge is labelled by an element of $S$, and the boundary word of each bounded region in $\R^2 \setminus D$ is freely equivalent to a word in $R$.
The word is read either clockwise or counter-clockwise, from an arbitrary starting vertex in the boundary of the region. The label of each edge traversed this way has a $+1$ exponent (respectively, $-1$ exponent) if the reading direction coincides with (respectively, is opposite to) the orientation of the edge. Changing the starting point or direction alters the word by cyclic conjugation or inversion. The boundary word of $D$, when read from the base point, is labelled by $w$. The diagram $D$ can also be given the structure of a $2$-complex, where the graph is viewed as the $1$-skeleton, and each bounded region corresponds to a $2$-cell attached via the word labelling its boundary. The \emph{area of a van Kampen diagram $D$}, denoted by $\area{D}$, is the number of bounded regions in  $\R^2 \setminus D$, or equivalently, the number of $2$-cells.

By van Kampen's Lemma \cite[Theorem 4.2.2]{bridsongeometry}, the area of a null-homotopic word $w$ in a finite presentation $\calp$ satisfies
\[
	\area[\calp]{w} = \min\{\area{D} \mid \text{$D$ is a van Kampen diagram for $w$ over $\calp$}\}.
\]
A van Kampen diagram attaining this minimum is said to be a \emph{minimal-area} van Kampen diagram.

\subsection{Coarse diagrams}

In the definition of van Kampen diagrams, we may weaken the condition on the boundary words of bounded regions by allowing them to be arbitrary null-homotopic words instead of relations. This yields a broader class of diagrams that are helpful for our work.

\begin{definition}[Coarse diagram]\label{def:coarse diagram}
	Let $S$ a finite generating set for a group $G$. A \emph{coarse diagram over $S$} for a null-homotopic word $w$ in $S$ is a van Kampen diagram for $w$, except that the bounded regions may be labelled by arbitrary null-homotopic words.
\end{definition}

To be more precise, a coarse diagram can be defined as a van Kampen diagram over an \emph{infinite} presentation, where the set of relations includes all null-homotopic words in a set of generators. Although coarse diagrams cannot be used directly to compute the Dehn functions of finitely presented groups, they can still be used to construct van Kampen diagrams over finite presentations in a ``patchwork'' manner, as follows.

Let $\calp = \presentation S R$ be a finite presentation, and let $D$ be a coarse diagram for a null-homotopic word $w$ in $S$. By definition, each bounded region $R_i$ of $D$ is labelled by a null-homotopic word $w_i$ in $S$. In particular, for each $w_i$, there is a van Kampen diagram $D_i$ over $\calp$. Replacing each bounded region $R_i$ in $D$ with the corresponding van Kampen diagram $D_i$ yields a van Kampen diagram $D'$ for $w$ over the finite presentation $\calp$.
In \cref{sec:estimating-the-almost-flat-area}, we explain this ``two-step filling'' technique in greater detail for the case of Bestvina--Brady groups, which is the focus of this work.

\subsection{Corridors}

To estimate the area of van Kampen diagrams, it is often useful to study their geometric structure by means of  \emph{corridors} and \emph{annuli}. We are interested in the corridors and annuli in the van Kampen diagrams over the standard presentation $\raagpres$ for $\raag\Gamma$. For this reason, we give their definition only for this particular case.

Let $D$ be a van Kampen diagram over $\raagpres$ for a null-homotopic word $w$ in $\raaggens$,  and let $a\in\raaggens$. An edge in $D$ is called an \emph{$a$-edge} if it is labelled by $a$. Let $D^\star$ be the graph dual to the $1$-skeleton of $D$, together with a vertex $v_\infty$ dual to the unbounded region of $D \subset \R^2$. Let $\lambda$ be a loop in $D^\star$ whose edges are all dual to $a$-edges. The subdiagram $C$ of $D$, consisting of all the closed $1$-cells and $2$-cells of $D$ dual to $\lambda \setminus \{ v_\infty \}$, is called an \emph{$a$-corridor} if $\lambda$ includes $v_\infty$, and an \emph{$a$-annulus} if it does not. Moreover, if $C$ is an $a$-corridor, then the path obtained from $\lambda$ by removing its intersection with the interior of the unbounded region is called the \emph{core of $C$}. We say that two corridors \emph{cross} if they intersect in one or more $2$-cells; equivalently, if their cores intersect.

The standard orientation of $\R^2$ induces an orientation on the core of a corridor $C$, so that $a$-edges intersect the core transversally from left to right. The paths consisting of $1$-cells in $C$ that are parallel to the core are called the \emph{sides} of $C$. The \emph{length} of a corridor is the number of its $2$-cells, and the \emph{label} of a corridor is the word read along either of its sides, following the orientation.

For $a,b \in \raaggens$, if an $a$-corridor and a $b$-corridor cross each other, then $[a,b]=1$ must be a relation in $\raagpres$. This tells us that the corresponding vertices $v_a, v_b \in \Gamma$ are adjacent. In particular, we have $a \neq b$, since, while $a$ commutes with itself, the commutator $[a,a]$ is not a relation of $\raagpres$. Thus, a corridor cannot cross itself. Moreover, we have the following standard result regarding minimal-area van Kampen diagrams in right-angled Artin groups.

\begin{lemma}\label{no-annuli-and-bigons}
	Let $w$ be a null-homotopic word in the standard generating set $\raaggens$ of $\raag
		\Gamma$. If $D$ is a minimal-area van Kampen diagram for $w$ over $\raagpres$, then it contains no annuli, and if two corridors intersect, then they do so at most once.
\end{lemma}
\begin{proof}
	If $D$ contained an annulus, then removing it would yield a diagram of smaller area, a contradiction. Suppose instead that two corridors cross more than once, and call the region between them a bigon. By performing surgery on an innermost bigon, we can remove the double intersection and obtain a van Kampen diagram with two fewer $2$-cells, again, a contradiction.
\end{proof}

\section{Alternating words and diagrams in Bestvina--Brady groups}\label{sec:alternating}

Throughout the rest of the work, we denote by $\Gamma$ a finite simplicial graph whose associated flag complex $\flag \Gamma$ is simply connected.

As mentioned before, the group $\bbg \Gamma$ is naturally embedded in $\raag \Gamma$ as the $0$-level set of the height function $\height\colon\raag\Gamma\to\Z$. This perspective allows us to exploit our understanding of the ambient group $\raag\Gamma$ to study $\bbg \Gamma$.

To this end, instead of working with the abstract presentation of $\bbg\Gamma$, we use the presentation for $\raag\Gamma$, but consider only words in $\raaggens$ that are close to the $0$-level set; we call these \emph{alternating words} (see \cref{alternating-word}). We also consider the van Kampen diagrams that are close to the $0$-level set, leading to the definition of \emph{almost-flat van Kampen diagrams} (see \cref{almost-flat-diagram}).
Almost-flat van Kampen diagrams are central in our estimates for the Dehn function of $\bbg \Gamma$.

\subsection{Alternating words and diagrams} We start by defining the set of alternating words in $\raaggens$, which represent elements in $\bbg \Gamma$.

\begin{definition}[Alternating words]\label{alternating-word}
	A word $w$ in the standard generating set $\raaggens$ for $ \raag \Gamma $ is called an \emph{alternating word} if it has even length and its letters alternate between elements of $\raaggens$ and $\raaggens^{-1}$, starting with an element in $\raaggens$.
\end{definition}

Geometrically, the height function $\height \colon \raag\Gamma \to \Z$ extends naturally to a map $\cayheight \colon \cay{\raag\Gamma, \raaggens} \to \R$, where $\cay{\raag\Gamma, \raaggens}$ denotes the Cayley graph of $\raag\Gamma$ with respect to the standard generating set $\raaggens$. Thus, a word is alternating if and only if it represents a path in $ \cayheight^{-1}([0,1]) \subset \cay{\raag \Gamma, \raaggens}$ that starts and ends at $\cayheight^{-1}(0)$.

The following lemma shows that (up to free equivalence) there is a one-to-one correspondence between alternating words and words in the standard generating set $\bbggens$ of $\bbg\Gamma$.

\begin{lemma}\label{alternating-words-and-words-in-bbg}
	Let $\ll$ be a total order on the vertices of $\Gamma$. Let $T$ be a spanning tree for $\Gamma$, where the edges are oriented increasingly. Let $\bbggens$ be the generating set for $\bbg \Gamma$ with respect to $T$, and let $\altwords$ be the set of alternating words in $\raaggens$.
	There exist canonical maps $\Psi \colon \monoid{\bbggens} \to \altwords$ and $\Phi \colon \altwords \to \monoid{\bbggens}$ that satisfy the following properties.
	\begin{enumerate}
		\item The composition $\Phi \circ \Psi$ is the identity map on $\monoid\bbggens$. \label{item:xi-of-psi}
		\item For every $u \in \altwords$, we have $\Psi(\Phi(u)) \freeid u$. \label{item:psi-of-xi}
		\item For every $u \in \altwords$, the words $ \Phi(u) $ and $u$ represent the same element in $\bbg\Gamma$. Similarly, for every $v \in \monoid \bbggens$, the words $\Psi(v)$ and $v$ represent the same element in $\bbg \Gamma$. \label{item:commutative-diagram}
		\item If $r$ is a relation in the Papadima--Suciu presentation corresponding to a directed triangle with vertices $ v_a \ll v_b \ll v_c $, then $ \Psi(r) $ is freely equivalent to a conjugate of $ac^{-1}ba^{-1}cb^{-1}$.  \label{item:psi-of-triangle}
	\end{enumerate}
	In particular, there exists a canonical isomorphism between $\free \bbggens$ and the image of $\altwords$ in $\free \raaggens$ via the natural quotient map $\monoid \raaggens \to \free \raaggens$.
\end{lemma}

\begin{proof}
	The maps $\Psi$ and $\Phi$ are defined as follows. If $w=e_1^{\epsilon_1} \cdots e_k^{\epsilon_k}$ is a word in $\bbggens$, where $e_i=(v_{s_i},v_{t_i})$ is an oriented edge of $T$ for $1\leq i\leq k$, we define
	\[
		\Psi(w) \coloneq (s_1 t_1^{-1})^{\epsilon_1} \cdots (s_k t_k^{-1})^{\epsilon_k}.
	\]
	Conversely, let $u=s_1t^{-1}_1\cdots s_nt^{-1}_n$ be an alternating word, where $s_i,t_i \in \raaggens$ for $1\leq i\leq n$. Recall from the definition of the Papadima--Suciu presentation that $w_{s_i,t_i}$ is the word in the generating set $\bbggens$ obtained by taking the product of the edges along the path in $T$ from $v_{s_i}$ to $v_{t_{i}}$. We define
	\[
		\Phi(u) \coloneqq w_{s_1,t_1} \cdots w_{s_n,t_n}.
	\]

	The statements \eqref{item:xi-of-psi}, \eqref{item:psi-of-xi}, and \eqref{item:commutative-diagram} are clear from the definition of $\Psi$ and $\Phi$, and the fact that $w_{s,t}$ represents the element $st^{-1}$. For \eqref{item:psi-of-triangle}, let $(e,f,g)$ be a directed triangle with $e = (v_a, v_b)$ and $f = (v_b, v_c)$. Since $\flag \Gamma$ is simply connected, the words $w_e e^{-1}$ and $w_f f^{-1}$ are null-homotopic. It then follows from the definition of $\Phi$ and $\Psi$ that $\Phi(w_e) = \Phi(e) = ab^{-1}$; similarly $\Phi(w_f) = bc^{-1}$. Thus, for $r = [w_e, w_f]$, we have
	\[
		\Phi(r) =  ab^{-1} bc^{-1} ba^{-1} cb^{-1} \freeid ac^{-1}ba^{-1}cb^{-1}. \qedhere
	\]
\end{proof}

Using alternating words, we define the following notion of length for an element in $\bbg \Gamma$.

\begin{definition}\label{def: flat norm}
	For $g \in \bbg\Gamma$, we define the \emph{flat norm} of $g$ as
	\[
		\flatnorm g \coloneq \min \set{\abs w : w \text{ is an alternating word representing } g}.
	\]
\end{definition}

For an element $g \in \bbg \Gamma$, the flat norm $\flatnorm g$ and the usual word metric with respect to $\bbggens$, denoted by $\norm g_{\bbggens}$, are Lipschitz equivalent, as shown by the following result.

\begin{lemma}
	Let $T$ be a spanning tree for $\Gamma$. There exists $ C>0 $ such that for every $ g \in \bbg\Gamma $, we have
	\[
		\frac 1C \norm g_{\bbggens} \leq \flatnorm g \leq C \norm g_{\bbggens}.
	\]
\end{lemma}

\begin{proof}
	The result follows directly from the isomorphism between $\free \bbggens$ and the image of $\altwords$ in $\free \raaggens$; see \cref{alternating-words-and-words-in-bbg}.
\end{proof}

Recall that coarse diagrams are van Kampen diagrams in which the bounded regions are labelled by arbitrary null-homotopic words. We now define a particular class of coarse diagrams, where the bounded regions are labelled by null-homotopic alternating words.

\begin{definition}[Alternating diagram]\label{def:alternating diagram}
	Let $w$ be a null-homotopic alternating word in the standard generating set $\raaggens$ of $\raag \Gamma$, and let $D$ be a coarse diagram over $\raagpres$ for $w$. We say that $D$ is an \emph{alternating diagram} for $w$ if every loop based at a vertex of height $0$ is labelled by an alternating word.
\end{definition}

\Cref{def:alternating diagram} can also be interpreted geometrically as follows. The $1$-skeleton of a coarse diagram $D$ over $\raagpres$ admits a natural immersion into $\cay{\raag\Gamma, \raaggens}$. To construct the immersion, for each vertex $v$ of $D$, choose a combinatorial path connecting the base point of $D$ to $v$, and send $v$ to the group element $g_v$ that is represented by the word labelling the chosen path. Since bounded regions are labelled by null-homotopic words, the element $g_v$ is independent of the choice of path. In particular, every vertex $v$ has a well-defined height $\height(g_v)$.

Alternating diagrams are the coarse diagrams whose images in the Cayley graph under the natural immersion are contained in $\cayheight^{-1}([0,1])$. Equivalently, they are coarse diagrams in which all vertices have height $0$ or $1$.

\subsection{Almost-flat van Kampen diagrams}

Note that a van Kampen diagram over the standard presentation $\raagpres$ of $ \raag\Gamma$ cannot be an alternating diagram (unless the boundary word is freely trivial), since none of the defining relations are alternating words. To address this, we allow the van Kampen diagrams to be slightly farther away from the $0$-level set.

\begin{definition}[Almost-flat van Kampen diagram]\label{almost-flat-diagram}
	Let $w$ be a null-homotopic alternating word in the standard generating set $\raaggens$ of $\raag \Gamma$. An \emph{almost-flat van Kampen diagram for $w$} is a van Kampen diagram $D$ for $w$ over $\raagpres$, such that every vertex $p \in D$ has height $\height(p) \in \set{0,1,2}$.
\end{definition}

The reader should think of alternating diagrams as the coarse version of almost-flat van Kampen diagrams. All the regions of an alternating diagram $D$ for a null-homotopic alternating word $w$ are labelled by alternating words (assuming that we start reading each word from a vertex at height $0$). So, by filling each bounded region of $D$ with an appropriate almost-flat van Kampen diagram, we obtain an almost-flat van Kampen diagram for $w$.

We use almost-flat van Kampen diagrams to compute the asymptotic behaviour of the Dehn function of $\bbg\Gamma$.

\begin{definition}[Almost-flat area]\label{def:almost-flat-vK-diagram}
	Let $w$ be a null-homotopic alternating word in the standard generating set $\raaggens$ of $\raag \Gamma$. The \emph{almost-flat area} of $w$ is the quantity defined as
	\[
		\flatarea w = \min\{ \area{D} : D \ \text{is an almost-flat van Kampen diagram for} \ w \}.
	\]
\end{definition}

\begin{proposition} \label{prop:area-bbg-vs-almost-flat}
	Define $\delta_\flat \colon \N \to \N$ by
	\[
		\delta_\flat (n) \coloneq \max\left\{\flatarea w : \begin{array}{c}\text{$w$ is a null-homotopic alternating word in} \ \raag\Gamma \ \text{with} \ \length w \leq n \end{array}\right\}.
	\]
	Then, we have $\delta_\flat \asymp \Dehn{\bbg\Gamma}$.
\end{proposition}

\cref{prop:area-bbg-vs-almost-flat} follows from the following technical lemma.

\begin{restatable}{lemma}{techlemma}\label{pushdown-from-bounded-height}
	For every $H \in \N$, there exists $M_H \in \N$ such that the following holds. Let $w$ be a null-homotopic alternating word in $\raag\Gamma$. Suppose that $w$ admits a van Kampen diagram $D$ over $\raagpres$ such that the height of every vertex $v$ satisfies $\abs{\height(v)} \leq H$.
	Then $w$ admits an alternating diagram with the same number of bounded regions as $D$, such that each bounded region is labelled by an alternating word of length at most $M_H$.
\end{restatable}

\cref{pushdown-from-bounded-height} can be proved using a fairly standard ``pushing down'' argument; see \cite{Dison,PushingRAAGs}. We postpone the proof of \cref{pushdown-from-bounded-height} until \cref{sec:coloured-words}, where we introduce a variation of this technique to establish the upper bounds in \cref{main-theorem}. We now show how \cref{pushdown-from-bounded-height} implies \cref{prop:area-bbg-vs-almost-flat}.

\begin{proof}[Proof of \cref{prop:area-bbg-vs-almost-flat}]
	Fix a spanning tree $T$ for $\Gamma$, and let the maps $\Phi \colon \altwords \to \monoid \bbggens$ and $\Psi \colon \monoid \bbggens \to \altwords$ be as in \cref{alternating-words-and-words-in-bbg}.

	We start by proving $\delta_\flat \preccurlyeq \Dehn{\bbg\Gamma}$. Let $u$ be a null-homotopic alternating word of length $2n$, and let $D$ be a van Kampen diagram for the word $w = \Phi(u)$ over the Papadima--Suciu presentation $\bbgpres$ of $\bbg\Gamma$; its edges are labelled by elements of $\bbggens$, corresponding to oriented edges of $T$. Replacing each edge of $D$, labelled with an oriented edge of $T$ oriented from $ v_s $ to $ v_t $, with two consecutive edges labelled by $st^{-1}$ yields an alternating diagram for $\Psi(w) \freeid u$, where all bounded regions are labelled with words of the form $\Psi(r)$ for some relation $r$ of $\bbgpres$.
	By \cref{alternating-words-and-words-in-bbg}~\cref{item:psi-of-triangle}, the relations of the presentation for $\bbg\Gamma$ are mapped to words that are, up to free equivalence and conjugation, of the form $ac^{-1}ba^{-1}cb^{-1}$. Such an alternating word admits an almost-flat van Kampen diagram of area $3$, so $u$ admits an almost-flat van Kampen diagram of area $3 \cdot \area D $. This shows that $\flatdehn(2n) \leq 3 \Dehn\bbgpres (n \cdot \cardinality{\ee\Gamma})$.

	We now prove the inequality $\flatdehn \succcurlyeq \delta_{\bbg \Gamma}$.
	Let $H=2$ and $M_2$ be the constants in \cref{pushdown-from-bounded-height}. Let $R$ be the set of null-homotopic words in $ \bbggens $ of length at most $M_2 \cdot \cardinality{\vv\Gamma}$ and $\calp \coloneq \presentation \bbggens R$.
	Consider a null-homotopic word $w$ of length $n$ in $ \bbggens $, and let $D$ be an almost-flat van Kampen diagram for $u \coloneq \Psi(w)$ with area $ A \leq \flatdehn(2n) $.

	Since $D$ in an almost-flat van Kampen diagram, the height of every vertex $v \in D$ satisfies $\abs{\height(v)} \leq 2$. It follows from  \cref{pushdown-from-bounded-height} that there exists a constant $M_2$, independent of $u$, such that $u$ admits an alternating diagram with $A$ bounded regions, all labelled by alternating words of length at most $M_2$. By van Kampen's Lemma \cite[Theorem 4.2.2]{bridsongeometry}, there are alternating words $ \alpha_i$ and $u_i$ for $i \in \range A$, such that $\length{u_i} \leq M_2 $ and
	\[
		u \freeid \prod_{i=1}^A \alpha_i u_i \alpha_i^{-1}.
	\]

	By applying $\Phi$ on both sides of the identity we get
	\[
		w \freeid \prod_{i=1}^A \Phi(\alpha_i)\Phi(u_i)\Phi(\alpha_i)^{-1}.
	\]
	A simple estimate on the length of $\Phi(u_i)$ shows that $\Phi(u_i) \in R$. This implies that $\calp$ is a presentation for $\bbg\Gamma$, and $\Dehn \calp(n) \leq \flatdehn(2n)$.
\end{proof}

\subsection{Estimating the almost-flat area}\label{sec:estimating-the-almost-flat-area}

In this section, we develop some technical tools that are helpful to construct almost-flat van Kampen diagrams with good estimates on their areas.

As previously mentioned, one approach is to start with an alternating diagram $D$ for a null-homotopic alternating word $w$ and then fill each bounded region with an almost-flat van Kampen diagram.
The following quantity, which is well-defined for any coarse diagram, allows us to control the area of the resulting almost-flat van Kampen diagram.

\begin{definition}[Density]\label{def:density}
	Let $G$ be a group with a finite generating set $S$, and let $D$ be a coarse diagram over $S$ for a null-homotopic word $w$ in $G$. Denote the labels of the bounded regions of $D$ by $w_i$, $ i \in I $. The \emph{density} of $D$ is defined as
	\[
		\density D \coloneq \frac{ \sum_{i\in I}\length{w_i}}{\length{w}}.
	\]
\end{definition}

Equivalently, the density relates the total number of edges of $D$ with the length of $w$.

\begin{lemma}\label{density-vs-num-edges}
	Let $D$ be a coarse diagram for a null-homotopic word $w$, and denote by $E_D$ the edge set of $D$. The density of $D$ satisfies
	\[
		\density D = \frac{2 \numedges D}{\length w} - 1.
	\]
\end{lemma}

\begin{proof}
	Let $w_i$, $ i \in I $, denote the labels of the bounded regions of $D$. A double-counting argument on the lengths of $w$ and the $w_i$ yields
	\[
		2 \numedges D = \length w + \sum_{i \in I} \length {w_i}.
	\]
	The result now follows.
\end{proof}

The next result illustrates how density comes into play when estimating the almost-flat area of the boundary alternating word in terms of the areas of the bounded regions.

\begin{lemma}\label{lem:area-of-protodiagram}
	Let $ \alpha \geq 1$ and $C>0 $ be real numbers. Suppose that $w$ is a null-homotopic alternating word in $ \raaggens $ admitting an alternating diagram $D$, such that for every alternating word $w_i$ labelling a bounded region of $D$, we have $ \flatarea{w_i} \leq C \cdot \length{w_i}^\alpha $. Then
	\[
		\flatarea{w} \leq C \cdot \density D^{\alpha} \cdot \length{w}^\alpha.
	\]
\end{lemma}

\begin{proof}

	By assumption, we can fill each bounded region of $D$, labelled by $w_i$, with an almost-flat van Kampen diagram whose almost-flat area is at most $ C \cdot \length{w_i}^{\alpha} $. After filling each bounded region of $D$, we obtain an almost-flat van Kampen diagram $D'$ for $w$ with
	\[
		\area{D'}
		\leq\sum_i\flatarea{w_i}
		\leq\sum_i C \cdot \length{w_i}^\alpha \leq C \cdot \left(\sum_i \length{w_i}\right)^{\alpha}
		= C \cdot  \density D^{\alpha} \cdot  \length{w}^\alpha.
	\]
	This implies $\flatarea{w}\leq C \cdot \density D^{\alpha}\cdot\length{w}^\alpha$. \qedhere
\end{proof}

\cref{lem:area-of-protodiagram} tells us that, whenever we can control the density of the alternating diagram, the almost-flat area of the boundary word has the same asymptotic behaviour as the almost-flat area of the labels of the bounded regions. We formalise this strategy as the following corollary.

\begin{corollary}\label{two-step-filling}
	Let $W$ and $W' $ be two families of null-homotopic alternating words. Let $\alpha \geq 1$ and $C > 0$ be constants such that each $w' \in W'$ has almost-flat area bounded above by $C \cdot \length{w'}^\alpha$. Suppose that for every $w \in W$, there exists an alternating diagram $D_w$ such that
	\begin{itemize}
		\item all its bounded regions are labelled by words in $W'$, and
		\item its density is bounded above uniformly in $w$.
	\end{itemize}
	Then there exists $C'>0$ such that for all $w \in W$, we have $\flatarea w \leq C' \cdot \length w^\alpha$.
\end{corollary}
\begin{proof}
	Let $C' \coloneq C \cdot( \sup\set{\density{D_w} : w \in W})^\alpha$, which is finite by hypothesis. The result is now a direct consequence of \cref{lem:area-of-protodiagram}.
\end{proof}

The two-step filling procedure described above can also be turned into a multistep procedure, by means of the following observation.

\begin{lemma}\label{lem:fill-sparse-with-sparse}
	Let $D$ be an alternating diagram for a null-homotopic alternating word $w$ in $\raaggens$, and denote by $w_i$, $ i \in I $, the labels of the bounded regions of $D$. Let $D'$ be the diagram where each bounded region is replaced with an alternating diagram $D_i$ for $w_i$. Then
	\[
		\density{D'} \leq \density D \cdot \max\set{\density {D_i} : i \in I }.
	\]
\end{lemma}
\begin{proof}
	Denote by $w_{i,j}$, $j\in J_i$, the labels on the bounded regions of $D_i$. Let $ M \coloneq \max\set{\density {D_i} : i \in I } $, we get
	\[
		\sum_{i\in I}\sum_{j\in J_i}\length{w_{i,j}}= \sum_{i \in I} \density {D_i} \cdot \length{w_i} \leq M \cdot \sum_{i\in I} \length{w_i}= M \cdot \density D \cdot \length w. \qedhere
	\]
\end{proof}

We conclude this section with a standard technique for obtaining upper bounds on the area of null-homotopic words. The idea behind this method can be traced back to Gromov \cite[{$5.A_3''$}]{Gromov}. Gersten and Short later formalised this idea and used it to provide upper bounds on the Dehn functions of certain subgroups of hyperbolic groups \cite[Lemma 2.2]{GerstenShort}. In the next result, we present a generalised version due to Carter and Forester \cite{CarterForester}, applied to our setting of alternating words and almost-flat area. Intuitively, it tells us that to obtain an estimate of the Dehn function, it suffices to consider a certain class of triangular diagrams.

\begin{lemma}[Triangle Lemma]\label{triangle-lemma}
	Let $ C>0$ and $\alpha \geq 2 $ be real numbers. For each $ g \in \bbg\Gamma $, choose an alternating word $ w_g $ representing $g$ such that $ \length{w_g} \leq C \cdot \flatnorm{g}$. Suppose that for every $ g_1, g_2, g_3 \in \bbg\Gamma $ satisfying $ g_1 g_2 g_3 = 1 $, we have
	\[
		\flatarea{w_{g_1}w_{g_2}w_{g_3}} \leq C \cdot (\flatnorm{g_1} + \flatnorm{g_2} + \flatnorm{g_3})^\alpha.
	\]
	Then, $ \Dehn {\bbg\Gamma}(n) \preccurlyeq n^\alpha $.
\end{lemma}

\begin{proof}
	Since the proof is fairly standard and analogous to the proof of \cite[Theorem 4.2]{CarterForester}, we give only the main idea and omit the details. Starting with an empty diagram for an alternating word $w$, we subdivide it into triangular regions whose sides are labelled by alternating words of the form $w_g$. By hypothesis, we can fill each triangular region with an almost-flat van Kampen diagram and estimate the almost-flat area of the whole diagram.

	By exploiting the geometric series, one can see that the area is at most $C' \cdot \length w^\alpha$ for some $C' = C'(C,\alpha)$. Combining this with \cref{prop:area-bbg-vs-almost-flat} yields the result. 
\end{proof}

\section{Coloured words and coloured diagrams}\label{sec:coloured-words}

Establishing an upper bound for the area of a null-homotopic alternating word $w$ in $\bbg\Gamma$ is generally challenging. As mentioned in the introduction, there are cases where this is straightforward. For example, if we express $w$ as a word in $\raaggens$ and assume that all its letters commute with some fixed generator $a \in \raaggens$, then $w$ represents an element in $\bbg{\Star {v_a}}$, which is isomorphic to $\raag{\Link {v_a}}$ via the homomorphism induced by $sa^{-1} \mapsto s$ for $v_s\in\Link {v_a}$. Thus, the area of $w$ in $\bbg \Gamma$ is at most $\length{w}^2$. Concretely, to transform $w$ into the empty word, we commute its letters different from $a$ until they cancel out, using $a$ as a ``counterweight'' to keep the word alternating at all times.

To handle the general case, the main idea is to split a null-homotopic alternating word $w$ into subwords in such a way that each subword has its support contained in the star of some vertex of $\Gamma$. This provides a preferred way to manipulate each subword, as it belongs to the special case discussed above. We also want to keep track of the vertex used as a counterweight, as there might be multiple choices available.

This leads to the definition of \emph{coloured words}, which can be intuitively understood as follows. In the example above, the counterweight $a \in \raaggens$ is now thought as the colour of the word $w$. In general, the palette of colours consists of the standard generating set $\raaggens$ that corresponds to the generators used as counterweights. Given a word $w$ in $\raaggens$, we ``paint'' its letters so that each letter of $w$ commutes with its colour. 

This process splits $w$ into subwords by grouping together the maximal substrings of consecutive letters in $w$ that share the same colour. A subword with colour $a$ defines an element in the right-angled Artin group $\raag{\Star{v_a}}$. This gives rise to a canonical way to ``push down'' this coloured word to obtain an alternating word in $\bbg{\Star{v_a}}$; see \cref{def:push-down}. Moreover, we exploit this idea to define a class of diagrams for null-homotopic coloured words, called  \emph{coloured diagrams}, and a way to ``push down'' these diagrams to obtain alternating diagrams.

\subsection{Coloured words and pushdowns}
We now give the formal definitions of coloured words and their pushdowns, and present some basic properties.

\begin{definition}\label{def: coloured gen. set}
	The \emph{coloured generating set} is the subset $\colraaggens \subseteq \raaggens \times \raaggens$ consisting of all pairs $(s,a)$ such that $ [s,a]=1 $. It forms a generating set for $\raag\Gamma$ (with repeated generators) when equipped with the map $\colraaggens \to \raag\Gamma$ defined by $(s,a) \mapsto s$. A \emph{coloured word} is a word in $\colraaggens$, that is, an element in the free monoid generated by $\colraaggens \cup (\colraaggens)^{-1}$.
\end{definition}

The letter $ (s,a) \in \colraaggens $ should be thought of as the letter $s$ coloured with the colour $a$. To emphasise this, we introduce the following notation.

\begin{notation}\label{notation-colour-words}
	Let $a \in \raaggens$. If $s \in \raaggens$ commutes with $a$, then we denote by $\colword sa$ the element $(s,a) \in \colraaggens$ and refer to it as a \emph{coloured letter}. The formal inverse of $\colword sa$ is denoted by $\colword{s}{a}^{-1}$. Given a word $w= s_1 \cdots s_k $ in $\raaggens$ and $a \in \raaggens$ that commutes with $s_i$ for all $i \in \range k$, we write $\colword wa$ to denote the coloured word $\colword{s_1}a \cdots \colword{s_k}a$; we refer to $a$ as the \emph{colour} of $w$. Moreover, whenever we write $\colword wa$, we implicitly assume that all the letters of $w$ commute with $a$.
\end{notation}

\begin{definition}
	Let $\mathbf w = \colword{s_1}{a_1}  \cdots \colword{s_k}{a_k}$ be a coloured word. The \emph{underlying word} of $\mathbf w$ is the word $w =  s_1  \cdots s_k$ in $\raaggens$ obtained by forgetting all the colours.
\end{definition}

By definition, if $\boldword$ is a coloured word and $w$ its underlying word, then $\length \boldword = \length w$.

\begin{remark}
	In the rest of the paper, we use bold letters for coloured words in $\colraaggens$ and the same letters in normal font for the corresponding underlying word in $\raaggens$.
\end{remark}

A coloured word $\mathbf w$ and its underlying word $w$ represent the same element in $\raag \Gamma$. In particular, a coloured word $ \mathbf w $ is null-homotopic if and only if its underlying word $w$ is null-homotopic.

\begin{definition}[$k$-coloured word]
	A \emph{$k$-coloured word} is a coloured word of the form
	\[
		\colword{w_1}{a_1} \cdots \colword{w_k}{a_k},
	\]
	where $ w_i $ is a nonempty word for all $i$, and $a_i$ differs from $a_{i+1}$ for all $ i \in \range {k-1} $. We call $1$-coloured words \emph{monochromatic}.
\end{definition}

\begin{remark}
	As an example, the word $ \colword{w_1}{a} \colword{w_2}{b} \colword{w_3}{a} $ is $3$-coloured, even though a colour is repeated.
\end{remark}

Given a word $w = s_1^{\epsilon_1} \cdots s_k^{\epsilon_k}$ in $\raaggens$, where $ \epsilon_i \in \pm1 $, there is a natural way to ``colour it'', which we refer to as \emph{self-colouring}, namely, the coloured word $\colword{s_1}{s_1}^{\epsilon_1} \cdots \colword{s_k}{s_k}^{\epsilon_k}$.

Coloured words are useful because the colours provide a preferred way to ``push down'' a word so that it becomes an alternating word in $\raaggens$. We formalise this as follows.

We fix, once and for all, a generator $s_0 \in \raaggens$.

\begin{definition}[Transition word]\label{def: trans. word}
	Let $a \in \raaggens$ and $T$ a spanning tree of $\Gamma$. Let $\gamma_a \colon [0,\ell] \to T$ be the unique simple path in $T$ from $v_{s_0}$ to $v_a$. For $i \in \range[0]\ell$, let $t_i \in \raaggens$ be such that $v_{t_i} = \gamma_a(i) \in \vv \Gamma$; in particular, $t_0=s_0$, $t_\ell=a$, and $[t_i, t_{i+1}] = 1$. For $h \in \Z$, we define the \emph{transition word} $\transitionword ah$ as the alternating word
	\[
		\transitionword ah \coloneq (t_0 t_{1}^{-1})^h \cdots (t_{\ell-1} t_{\ell}^{-1})^h.
	\]
\end{definition}

Observe that $ \transitionword ah $ represents the element $ s_0^h a^{-h} \in \raag\Gamma $. For every pair of vertices $v_x,v_y \in T$, Dison defined the word $p_h(v_x,v_y)$ in $\bbggens$; see \cite[Section 4]{Dison}. The image of $p_h(v_x,v_y)$ under the map $\Psi \colon \monoid{\bbggens} \to \altwords$ from \cref{alternating-words-and-words-in-bbg} is the word $\transitionword{x}{h}^{-1}\transitionword{y}{h}$.

We record the following straightforward observation about the length of the transition words.

\begin{lemma}\label{length-transition-word}
	Let $a \in \raaggens$ and $h \in \Z$. If $a \neq s_0$, then $2 \abs h \leq \length {\transitionword ah} \leq 2 \abs h  \cdot\numvertices \Gamma$. Otherwise, $\transitionword ah$ is the trivial word.
\end{lemma}
\begin{proof}
	The result follows directly from the definition of $\transitionword ah$ and the observation that the length of the path $\gamma_a$ defining $\transitionword ah$ is at most $\numvertices \Gamma$.
\end{proof}

\begin{notation}
	For a word $w$ in $ \raaggens$ and $a \in \raaggens$, we denote by $\balance{w}{a}$ the word obtained from $w$ by replacing each letter $s \in \raaggens$ (respectively, $s^{-1} \in \raaggens^{-1}$) with $sa^{-1}$ (respectively, $as^{-1}$).
\end{notation}

\begin{definition}[Pushdown of coloured words]\label{def:push-down}
	Let $h \in \Z$ and let $\mathbf w = \colwordproduct k $ be a $k$-coloured word. Let $h_0=h$ and $h_i = h + \sum_{j=1}^i \height{(w_j)}$ for $i \in \range[1]k$. The \emph{$h$-pushdown of $\mathbf w$} is the alternating word
	\[
		\pushdown[h]{\mathbf w} \coloneq \prod_{i=1}^{k} \transitionword{a_i}{h_{i-1}} \balance{w_i}{a_i} {\transitionword {a_i}{h_i}^{-1}}.
	\]
\end{definition}

We give an example for the definition of the pushdown of a coloured word.

\begin{example}
	Let $\Gamma = \{v_a\} * \Lambda$ be a cone graph and set $a\in \raaggens$ to be the generator corresponding to the vertex $v_a$. Fix the spanning tree of $\Gamma$ consisting of all the edges that contain $v_a$ as a vertex. Let $w = s_1^2 s_2^{-1} s_3$ be a word in $\raaggens$. Consider the monochromatic word $\mathbf w = \colword wa = \colword{s_1}{a}^{2} \colword{s_2}{a}^{-1} \colword{s_3}{a}$. For $h \in \Z$, the $h$-pushdown of $\mathbf w$ is the alternating word
	\begin{align*}
		\pushdown[h]{\mathbf w} & = \transitionword{a}{h} \balance{w}{a} \transitionword{a}{h + \height{(w)}}^{-1}               \\
		                        & = \transitionword{a}{h} (s_1a^{-1})^{2} (as_2^{-1}) (s_3a^{-1}) \transitionword{a}{h + 2}^{-1} \\
		                        & = (s_0 a^{-1})^h (s_1a^{-1})^{2} (as_2^{-1}) (s_3a^{-1}) (a s_0^{-1})^{h+2},
	\end{align*}
	where $s_0$ is the chosen generator in $\raaggens$ corresponding to a chosen base point in the spanning tree.
\end{example}

We now prove a series of lemmas about the pushdown of coloured words.

\begin{lemma}\label{pushdown-properties}
	Let $h \in \Z$. The $h$-pushdown has the following properties.
	\begin{enumerate}
		\item\label{item: pushdown of coloured letter} Let $\colword sa ^{\epsilon}$ be a coloured letter and $ \epsilon=\pm1 $, then
		\[
			\pushdown[h]{\colword{s}a^\epsilon} = \transitionword{a}{h} (sa^{-1})^\epsilon {\transitionword a{h+\epsilon}^{-1}}.
		\]
		\item\label{item: pushdown of ww'} Let $ \mathbf w$ and $\mathbf w' $ be coloured words, then
		\[
			\pushdown[h]{\mathbf w \mathbf w'} \freeid \pushdown[h]{\mathbf w}\pushdown[h+\height(\mathbf w)]{ \mathbf w'}.
		\]
		\item\label{item: pushdown of w=w'} If $  \mathbf w \freeid \mathbf w' $, then $ \pushdown[h]{\mathbf w} \freeid \pushdown[h]{\mathbf w'} $.
		\item \label{item: pushdown of alternating} If the underlying word $w$ of $ \boldword $ is alternating, then $ \pushdown \boldword \freeid w $.
	\end{enumerate}
\end{lemma}

\begin{proof}
	Statement \eqref{item: pushdown of coloured letter} follows from~\cref{def:push-down}. Statement \eqref{item: pushdown of ww'} also follows from \cref{def:push-down}. In fact, if consecutive colours are different we get an identity between words, whereas if the last colour of $\boldword$ coincides with the first colour of $\boldword'$, then a free insertion is required; we illustrate this in the case where $\boldword$ and $\boldword'$ are monochromatic. In this case, we have $\boldword \boldword' = \colword{w}{a}\colword{w'}{a} = \colword{ww'}{a}$, therefore
	\begin{align*}
		\pushdown[h]{\colword{ww'}{a}} & = \transitionword{a}{h} \cdot \balance {(ww')}a \cdot \transitionword{a}{h + \height(ww')}^{-1}                                                                            \\
		                               & \freeid \transitionword{a}{h}\balance wa\transitionword{a}{h + \height(w)}^{-1} \transitionword{a}{h + \height(w)}\balance{w'}{a}\transitionword{a}{h + \height(ww')}^{-1} \\
		                               & = \pushdown[h]{\colword wa}\pushdown[h+\height(w)]{\colword{w'}{a}}.
	\end{align*}

	Combining \eqref{item: pushdown of coloured letter} and \eqref{item: pushdown of ww'} shows that $\pushdown[h]{\colword sa \colword sa^{-1}}$ is freely trivial, and \eqref{item: pushdown of w=w'} follows. Finally, to show \eqref{item: pushdown of alternating}, it suffices to check that the identity holds for a coloured word $ \boldword = \colword ab \colword cd^{-1} $ of length 2.
	Here, we have
	\[
		\pushdown \boldword = \transitionword b0 ab^{-1} \transitionword b1^{-1} \transitionword d1 dc^{-1} \transitionword d0^{-1} \freeid ab^{-1} b s_0^{-1} s_0 d^{-1} dc^{-1} \freeid ac^{-1},
	\]
	where we use the fact that $\transitionword{a'}0$ is the trivial word and $\transitionword{a'}1 \freeid a'{s_0}^{-1}$ for $a' \in \raaggens$.
\end{proof}

\begin{lemma}\label{length-of-pushdown}
	Let $\mathbf w = \colword{w_1}{a_1} \cdots \colword{w_k}{a_k}$ be a $k$-coloured word. Let $h_0=h\in \Z$ and $h_i = h + \sum_{j=1}^i \height{(w_j)}$ for $i \in \range[1]k$. Then
	\[
		\length{\pushdown[h]{\mathbf w}} \leq 2 \length \boldword + 4\left(\sum_{i=0}^k \abs{h_i}\right) \cdot \numvertices \Gamma.
	\]
	Additionally, if $\boldword$ is null-homotopic and $a_1 \neq a_k$, then $h_0 = h_k$ and
	\[
		2\length \boldword + 2\sum_{i=1}^{k} \abs{h_i} \leq \length{\pushdown[h]{\mathbf w}}.
	\]
\end{lemma}

\begin{proof}
	By \cref{def:push-down}, we have
	\begin{align*}
		\length{\pushdown[h]{\mathbf w}} & = \sum_{i=1}^{k} \length{\transitionword{a_i}{h_{i-1}}} +\sum_{i=1}^{k}\length{\balance{w_i}{a_i}} + \sum_{i=1}^{k}\length{\transitionword{a_i}{h_i}^{-1}}                                                             \\
		                                 & = \sum_{i=1}^{k}\length{\balance{w_i}{a_i}} + \sum_{i=1}^{k-1}(\length{\transitionword{a_{i+1}}{h_i}}+\length{\transitionword{a_{i}}{h_i}}) + \length{\transitionword{a_1}{h_0}} + \length{\transitionword{a_k}{h_k}}.
	\end{align*}
	Since $\length{\balance{w_i}{a_i}}=2\length{w_i}$, we have
	\[
		\sum_{i=1}^{k}\length{\balance{w_i}{a_i}}=2\sum_{i=1}^{k}\length{w_i}=2\length \boldword.
	\]
	Using \cref{length-transition-word} to bound the lengths of the transition words gives
	\[
		\sum_{i=1}^{k-1}(\length{\transitionword{a_{i+1}}{h_i}}+\length{\transitionword{a_{i}}{h_i}}) + \length{\transitionword{a_1}{h_0}} + \length{\transitionword{a_k}{h_k}} \leq \left( 4 \sum_{i=1}^{k-1} \abs{h_i} + 2 \length{h_0} + 2 \abs{h_k} \right) \cdot \numvertices \Gamma.
	\]
	Putting all the identities and inequalities together yields the desired upper bound.

	For the lower bound, note that $\boldword$ being null-homotopic implies that $h_0 = h_k$. Moreover, for $i \in \range{k-1}$, since $a_i$ and $a_{i+1}$ are different, at least one of them is different from $s_0$; the same holds for $a_1$ and $a_k$. Therefore, \cref{length-transition-word} yields
	\[
		\sum_{i=1}^{k-1}(\length{\transitionword{a_{i+1}}{h_i}}+\length{\transitionword{a_{i}}{h_i}}) + \length{\transitionword{a_1}{h_0}} + \length{\transitionword{a_k}{h_k}}  \geq 2 \sum_{i=1}^{k-1} \abs{h_i} + 2\length{h_k}. \qedhere
	\]
\end{proof}

\begin{lemma}\label{pushdown-as-group-element}
	Let $h \in \Z$. If $\mathbf w = \colword{w_1}{a_1} \cdots  \colword{w_k}{a_k}$ represents $g \in \raag \Gamma$, then $\pushdown[h]{\mathbf w}$ represents the element $s_0^h g s_0^{-h-\height(g)}$ in $\bbg \Gamma$.
\end{lemma}

\begin{proof}
	Let $h_0=h\in \Z$ and $h_i = h + \sum_{j=1}^i \height{(w_j)}$ for $i \in \range[1]k$.
	Since $a_i$ commutes with $w_i$ for $i \in \range k$, the word $\balance{w_i}
		{a_i}$ represents the element $w_i a_i^{-\height(w_i)}=a_i^{h_{i-1}}w_i a_i^{-h_i}$.
	Moreover, we have $\transitionword {a_i}h \groupid{\raag\Gamma} s_0^h {a_i}^{-h}$.
	The result follows from \cref{pushdown-properties}.
\end{proof}

\subsection{Efficient coloured words}
In the previous section, we showed that a coloured word can always be turned into an alternating word by means of the pushdown. In this section, we show that for each element of $\bbg\Gamma $, there is a particularly nice coloured word whose pushdown represents it.

\begin{definition}\label{def:proper-block}
	Let $ \mathbf{w} = \colword{w_1}{a_1} \cdots \colword{w_k}{a_k} $ be a $k$-coloured word. A \emph{proper block} of $ \mathbf w $ is a subword of the form $ \colword{w_i}{a_i}\colword{w_{i+1}}{a_{i+1}} \cdots \colword{w_j}{a_j} $ for some $ 2 \leq i \leq j \leq k-1 $.
\end{definition}

\begin{definition}
	We call a $k$-coloured word $ \mathbf w $ representing $ g \in \raag \Gamma $:
	\begin{itemize}
		\item \emph{geodesic}, if it is a minimal-length representative for $ g $;
		\item \emph{chromatically minimal}, if there are no $k'$-coloured words representing $g$ with $k'<k$;
		\item \emph{efficient}, if it is geodesic and every proper block is chromatically minimal.
	\end{itemize}
\end{definition}

\begin{remark}
	Note that a coloured word $\mathbf w$ is geodesic if and only if the underlying word $w$ is a minimal-length representative for $g$ with respect to the standard generating set $\raaggens$.
\end{remark}

If a coloured word is chromatically minimal and geodesic, then it is also efficient because the proper blocks of a chromatically minimal word are chromatically minimal. The reason for using proper blocks in the definition of efficient coloured words is that being chromatically minimal does not pass to subwords.
As an example, label the vertices of a path of length four, from left to right, by $a$, $b$, $c$, and $d$. Then $\colword {ab}a \colword {cd}c $ is chromatically minimal, whereas $\colword ba \colword {cd}c $ is not, since it represents the same element $bcd$ as $\colword {bcd}c$. However, it is not difficult to see that a subword of an efficient word is also efficient, as the proper blocks of a subword are, in particular, proper blocks of the original coloured word.

\begin{lemma}\label{lem:existence-efficient}
	For each element $g \in \bbg \Gamma$, there exists an efficient coloured word representing it.
\end{lemma}

\begin{proof}
	Let $ \mathbf w $ be a chromatically minimal representative for $g$. If it is not geodesic, then the underlying word $w$ has a subword of the form $ s  w' s^{-1} $, where $w'$ is a word commuting with $s \in \raaggens \cup \raaggens^{-1}$. Removing the coloured letters corresponding to $s$ and $ s^{-1} $ from $ \mathbf w $ yields another chromatically minimal word of shorter length. Repeating this procedure eventually yields a chromatically minimal and geodesic representative, and therefore, an efficient one.
\end{proof}

In general, the length of the pushdown of a coloured word $\mathbf w$ might be considerably greater than the length of $\mathbf w$ because transition words contribute to its length. Since Bestvina--Brady groups can be quadratically distorted in right-angled Artin groups \cite[Theorem 1.1]{TranDistortion}, the increase in the length of the pushdown cannot be avoided. Nevertheless, when $\mathbf w$ is efficient, the length of its pushdown is always optimal (up to a multiplicative constant).

\begin{proposition}\label{coloured-norm-vs-bbg-norm}
	Let $\mathbf w$ be an efficient $k$-coloured word representing an element $g \in \bbg \Gamma$. Then
	\[
		\flatnorm g \leq \length{\pushdown{\mathbf w}} \leq (24 \numvertices \Gamma +2) \flatnorm g.
	\]
\end{proposition}

To prove \cref{coloured-norm-vs-bbg-norm}, we need the following two lemmas.

\begin{lemma}\label{bound-colour-distance-via-corridors}
	Let $w$ be a null-homotopic word in $\raaggens$ and $D$ a van Kampen diagram for $w$. Let $p$ and $q$ be two vertices on the boundary of $D$, and let $g$ be the element represented by the word read from $p$ to $q$ along the boundary. Let $ C $ and $ C' $ be corridors containing $ p $ and $q$, respectively.
	\begin{enumerate}
		\item\label{item: 2-coloured represented} If $ C $ and $ C' $ cross, then $g$ can be represented by a $2$-coloured word.
		\item\label{item: 3-coloured represented} If there is a corridor that crosses both $ C$ and $C' $, then $g$ can be represented by a $3$-coloured word.
	\end{enumerate}
\end{lemma}

\begin{proof}
	To prove \eqref{item: 2-coloured represented}, let $ C $ be an $ a $-corridor and $C'$ a $b$-corridor for some $ a,b \in \raaggens $. If $ C $ and $C'$ cross, then the side of $C$ containing $p$ intersects the side of $ C' $ containing $q$ in a vertex $p'$. Let $w'$ be the word labelling the path along the side of $C$ from $ p $ to $p'$, and let $w''$ be the word labelling the path along the side of $C'$ from $p'$ to $q$. It follows that $ \colword {w'}a \colword{w''}b $ is a $2$-coloured word representing $g$.

	For the proof of \eqref{item: 3-coloured represented}, let $C''$ be a $c$-corridor for some $c\in\raaggens$ that crosses both $C$ and $C'$, and let $p'$ and $q'$ be the intersection vertices of one side of $C''$ with the sides of $C$ and $C'$ that contain $p$ and $q$, respectively. Let $w$, $w''$, and $w'$ be the words labelling the paths from $p$ to $p'$, $p'$ to $q'$, and $q'$ to $q$, respectively. Then, the element $ g $ is represented by the $3$-coloured word $ \colword wa \colword{w''}c \colword{w'}b $.
\end{proof}

\begin{lemma}\label{estimate-sum-of-corridors-length}
	Let $\mathbf w = \colword{w_1}{a_1} \cdots \colword{w_k}{a_k}$ be an efficient $k$-coloured word whose underlying word is $w=w_1 \cdots w_k$. Let $ w'$ be a word in $\raag \Gamma$ such that $w w'$ is null-homotopic, and let $D$ be a minimal-area van Kampen diagram for $w w'$. For $i \in \range{k-1}$, let $p_i$ be the vertex on the boundary of $D$ at the intersection of $w_i$ and $w_{i+1}$, and let $C_i$ be a corridor containing $p_i$. Then, for $i,j \in \range {k-1}$, the following statements hold.
	\begin{enumerate}
		\item\label{item:one-corridors} If $ \abs{i-j} \geq 3 $, then $ C_i $ and $ C_j $ do not cross.
		\item\label{item:two-corridors} If $ \abs{i-j} \geq 4 $, no corridor crosses both $ C_i $ and $ C_j $.
		\item\label{item:three-corridors} The total length of the corridors $C_i$ satisfies
		\[
			\sum_i \abs{C_i} \leq 2 \length{w w'}.
		\]
	\end{enumerate}
\end{lemma}

\begin{proof}
	Let $i,j \in \range {k-1}$ with $j \geq i+3$. If $C_i$ and $C_j$ cross, then by \cref{bound-colour-distance-via-corridors}, the group element $g$ represented by the proper block $\colwordproduct[w][{i+1}]{j}$ can be represented by a $2$-coloured word, contradicting the efficiency of $\mathbf w$. Similarly, let $j \geq i+4$, and suppose that $C_i$ and $C_j$ are crossed by the same corridor. Then the element $g$, defined as above, would be represented by a $3$-coloured word, again a contradiction. This proves \eqref{item:one-corridors} and \eqref{item:two-corridors}.
	In particular, the statement in \cref{item:two-corridors} implies that a corridor $C$ in $D$ can cross at most four of the corridors $C_1, \dots, C_{k-1}$.

	Since $D$ is a minimal-area diagram, it contains no annuli, and every two corridors cross at most once. Therefore, the length of a corridor equals the number of corridors crossing it. There are $\frac12 \length{w w'}$ corridors in total, and each corridor can cross at most four of the corridors in the family $\{C_1, \dots, C_{k-1}\}$. That is, each of the $\frac12 \length{w w'}$ corridors contributes  at most $4$ to the overall sum $\sum_i \abs{C_i}$. Statement \eqref{item:three-corridors} now follows from a double-counting argument.
\end{proof}

We now proceed with the proof of \cref{coloured-norm-vs-bbg-norm}.

\begin{proof}[Proof of \cref{coloured-norm-vs-bbg-norm}]
	The first inequality $\flatnorm{g} \leq \length{\pushdown{\mathbf w}} $ follows from the definition of $\flatnorm{\cdot}$; see \cref{def: flat norm}.
	For the second inequality, let $\mathbf w = \colword{w_1}{a_1} \cdots \colword{w_k}{a_k}$ be an efficient $k$-coloured word with underlying word $w=w_1\cdots w_k$, and let $u$ be a shortest alternating word representing $g$, meaning that $\length u = \flatnorm g$. Since $\mathbf w$ is geodesic, we have $\length w \leq \length u$.

	Let $D$ be a minimal-area van Kampen diagram for $ w u^{-1} $, and let $ p_1, \dots, p_{k-1}$ be the points on the boundary path labelled by $w$ as in \cref{estimate-sum-of-corridors-length}. We also denote by $p_0$ and $p_k$ the points respectively at the start and end of $w$. Then the height $h_i$ of $p_i$ is $ \height(w_1 \dots w_i)$.

	Consider corridors $C_i$ containing $p_i$ for $i \in \range{k-1}$.
	Since $\mathbf w$ is geodesic, the corridor $C_i$ cannot have both ends on $w$, so it must begin at an edge of $w$ and end at an edge of $u$. Since $u$ is alternating, all the vertices of $ u $ have height $0$ or $1$, so $p_i$ has distance at most $\abs{C_i} + 1$ from a vertex of height $0$. Therefore, we may estimate the height of $p_i$ as $\abs{h_i} \leq \abs{C_i} + 1$ for every $i \in \range[0]{k-1}$.

	Finally, applying \cref{estimate-sum-of-corridors-length,length-of-pushdown} gives
	\[
		\length{\pushdown {\mathbf w}} \leq 2\length w + 4 \left(\sum_{i=0}^k \abs{h_i}\right) \numvertices \Gamma
		\leq 2 \length w + 4 (2\length{wu^{-1}} + k-1 + \abs{h_k}) \numvertices \Gamma
		\leq (24 \numvertices \Gamma + 2) \length u,
	\]
	where we use the fact that $ \max\set{k, \abs{h_k}} \leq \length w \leq \length u $.
\end{proof}

\subsection{Coloured diagrams}

To compute the Dehn function of $\bbg\Gamma$, we need to produce an algorithm that, given an arbitrary null-homotopic alternating word $u$ in $\bbg\Gamma$, constructs an almost-flat van Kampen diagram with small area. However, if $u$ is the pushdown of a coloured word $\mathbf w$, it is more convenient to work directly with the coloured word. To this end, we define the notion of \emph{coloured diagram}, which can be pushed down to obtain an alternating diagram for $u$.

\begin{definition}[Coloured diagram]
	A \emph{coloured diagram} for a null-homotopic coloured word $ \mathbf w $ is a coarse diagram for $\mathbf w $ over the generating set $ \colraaggens $ .
\end{definition}

Each edge of a coloured diagram is labelled with a coloured letter $ \colword sa $, and we sometimes refer to $a$ as the \emph{colour} of the edge. In this sense, a coloured diagram is a coarse diagram over the standard generating set $\raaggens$, where every edge carries additional information given by the colour.

Just as the pushdown of a coloured word produces an alternating word, the pushdown of a coloured diagram produces an alternating diagram.

\begin{definition}[Pushdown of a coloured diagram]\label{def:pushdown-of-coloured-diagram}
	Let $\mathbf D$ be a coloured diagram for a null-homotopic coloured word $\mathbf w$, and let $h \in \Z$. The \emph{$h$-pushdown of $\mathbf D$} is the alternating diagram $\pushdown[h]{\mathbf D}$ obtained as follows.
	\begin{itemize}
		\item For each vertex $p$ in $\bolddiag$, there is a vertex $p'$ in $\pushdown[h]{\mathbf D}$.
		\item For each edge in $\mathbf D$, labelled by the coloured letter $\colword sa$ and oriented from $p$ to $q$, we add a path from the corresponding vertices $p'$ to $q'$ in $\pushdown[h]{\mathbf D}$. This path is labelled with
		      \[
			      \pushdown[h+\height(p)]{\colword sa} = \transitionword a{h+\height(p)} sa^{-1} \transitionword a {h+\height(q)}^{-1}.
		      \]
		\item For every vertex $p$ of $\mathbf D$, and for every pair of consecutive edges incident to $p$ labelled with a letter of the same colour $a$, the corresponding paths in $\pushdown[h]{\mathbf D}$ have an initial segment labelled by the same word $\transitionword a{\height(p)}$. We \emph{fold} these initial segments by identifying them together.
		\item If all the edges incident to $p$ have the same colour $a$, the vertex $p'$ corresponding to $p$ in $\pushdown[h]{\mathbf D}$ has, after collapsing, a single path starting at it, which is labelled with $\transitionword {a}{\height(p)}$: we remove $p'$ and the path altogether.
	\end{itemize}
\end{definition}

\begin{example}
	\cref{fig:pushdown} illustrates a coloured diagram for $\boldword = \colword{a^2}b \colword{b}a \colword{a^{-1}}b \colword{a^{-1}b^{-1}}a$
	together with its $h$-pushdown, which is an alternating diagram for $\pushdown[h] \boldword$ which is the alternating word
	\[
		\transitionword bh ab^{-1}ab^{-1} \transitionword b{h+2}^{-1}  \transitionword a{h+2} ba^{-1} \transitionword a{h+3}^{-1} \transitionword b{h+3} ba^{-1} \transitionword b{h+2}^{-1} \transitionword a{h+2} aa^{-1} ab^{-1} \transitionword ah^{-1}.
	\]
\end{example}

\newboolean{pushdown}
\newboolean{collapse}
\newboolean{removeedges}
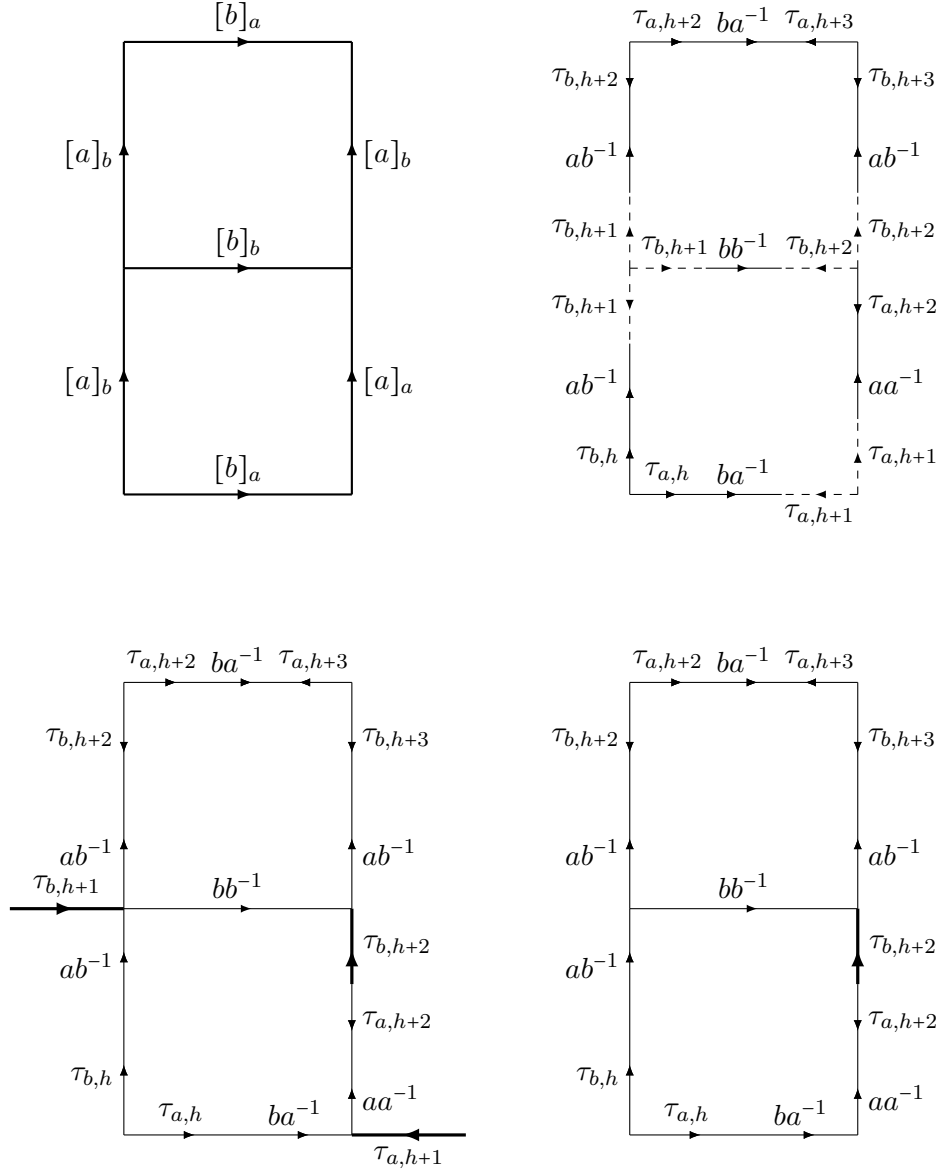
\begin{figure}
	\centering
	\begin{subfigure}{0.4\textwidth}
		\begin{tikzpicture}
	\useasboundingbox (-2,-1) (5,7);
	\tikzset{->-/.style={decoration={markings,mark=at position #1+.01 with {\arrow{latex}}},postaction={decorate}}}
	\tikzset{-<-/.style={decoration={markings,mark=at position #1-.01 with {\arrowreversed{latex}}},postaction={decorate}}}
	\definecolor{happyblue}{RGB}{0,0,0}
	\definecolor{happypink}{RGB}{0,0,0}
	\tikzset{-|>-/.style={decoration={markings,mark=at position #1+.01 with {\arrow{latex}}},postaction={decorate}}}
	\tikzset{-<|-/.style={decoration={markings,mark=at position #1-.01 with {\arrowreversed{latex}}},postaction={decorate}}}

	\coordinate (A) at (0,0);
	\coordinate (B) at (0,3);
	\coordinate (C) at (0,6);
	\coordinate (D) at (3,6);
	\coordinate (E) at (3,3);
	\coordinate (F) at (3,0);

	\foreach\X in {A,B,C,D} {
			\foreach\Y in {A,B,C,D} {
					\coordinate (\X\Y) ($(\X)!1/3!(\Y)$);
				}
		}

	\unless\ifpushdown
		\draw[thick,-|>-=.55,happyblue] (A) -- (B) node[left,midway] {$\colword ab$};
		\draw[thick,-|>-=.55,happyblue] (B) -- (C) node[left,midway] {$\colword ab$};
		\draw[thick,->-=.55,happypink] (C) -- (D) node[above,midway] {$\colword ba$};
		\draw[thick,-|>-=.55,happyblue] (E) -- (D) node[right,midway] {$\colword ab$};
		\draw[thick,->-=.55,happypink] (F) -- (E) node[right,midway] {$\colword aa$};
		\draw[thick,->-=.55,happypink] (A) -- (F) node[above,midway] {$\colword ba$};
		\draw[thick,-|>-=.55,happyblue] (B) -- (E) node[above,midway] {$\colword bb$};
	\else
		\unless\ifcollapse
			\draw[-|>-=.3,-|>-=.7,happyblue] (A) -- +(0,2)
			node[left,pos=1/4] {$\transitionword bh$}
			node[left,pos=3/4] {$ab^{-1}$ }
			;
			\draw[-|>-=.3, -<|-=.7,happyblue] (B) ++(0,1) -- (C)
			node[left,pos=1/4] {$ab^{-1}$ }
			node[left,pos=3/4] {$\transitionword b{h+2}$}
			;

			\draw[-|>-=.55,happyblue] (B) ++(1,0) -- ++(1,0)
			node[above,pos=1/2] {$bb^{-1}$ }
			;
			\draw[dashed, -|>-=.55] (B) -- ++(0,1)
			node[left,midway] { $\transitionword b{h+1}$};
			\draw[dashed, -|>-=.55] (B) -- ++(1,0)
			node[above,pos=.6] { $\transitionword b{h+1}$};
			\draw[dashed, -|>-=.55] (B) -- ++(0,-1)
			node[left,midway] { $\transitionword b{h+1}$};
			\draw[dashed, -|>-=.55] (E) -- ++(-1,0)
			node[above,midway] { $\transitionword b{h+2}$};
			\draw[dashed, -|>-=.55] (E) -- ++(0,1)
			node[right,midway] { $\transitionword b{h+2}$};
		\else
			\draw[-|>-=.3,-|>-=.8,happyblue] (A) -- (B)
			node[left,pos=1/4] {$\transitionword bh$}
			node[left,pos=3/4] {$ab^{-1}$ }
			;
			\draw[-|>-=.3,-<|-=.7,happyblue] (B) -- (C)
			node[left,pos=1/4] {$ab^{-1}$ }
			node[left,pos=3/4] {$\transitionword b{h+2}$}
			;
			\draw[-|>-=.55,happyblue] (B) -- (E)
			node[above,midway] {$bb^{-1}$ }
			;
			\unless\ifremoveedges
				\draw[very thick,-<-=.45,happyblue] (B) -- ++(-1.5,0) node[above,midway] {$\transitionword b{h+1}$};
			\fi
		\fi
		\draw[->-=.22,->-=.55, -<-=.78,happypink] (C) -- (D)
		node[above,pos=1/6] {$\transitionword a{h+2}$}
		node[above,midway] {$ba^{-1}$ }
		node[above,pos=5/6] {$\transitionword a{h+3}$}
		;
		\unless\ifcollapse
			\draw[-|>-=.3, -<|-=.7,happyblue] (E) ++(0,1) -- (D)
			node[right,pos=1/4] {$ab^{-1}$ }
			node[right,pos=3/4] {$\transitionword b{h+3}$}
			;
			\draw[->-=.3, -<-=.7,happypink] (F) ++(0,1) -- (E)
			node[right,pos=3/4] {$\transitionword a{h+2}$}
			node[right,pos=1/4] {$aa^{-1}$ }
			;
			\draw[dashed, -|>-=.55] (F) -- ++(0,1)
			node[right,midway] { $\transitionword a{h+1}$};
			\draw[dashed, -|>-=.55] (F) -- ++(-1,0)
			node[below,midway] { $\transitionword a{h+1}$};

			\draw[->-=.3,->-=.7,happypink] (A) -- ++(2,0)
			node[above,pos=3/4] {$ba^{-1}$ }
			node[above,pos=1/4] {$\transitionword a{h}$}
			;
		\else
			\draw[-|>-=.3, -<|-=.7,happyblue] (E) -- (D)
			node[right,pos=1/4] {$ab^{-1}$ }
			node[right,pos=3/4] {$\transitionword b{h+3}$}
			;
			\draw[->-=.3,->-=.8,happypink] (A) -- (F)
			node[above,pos=1/4] {$\transitionword a{h}$}
			node[above,pos=3/4] {$ba^{-1}$ }
			;
			\draw[very thick,-<|-=.55,happyblue] (E) -- ++(0,-1)
			node[right,midway] {$\transitionword b{h+2}$}
			;
			\draw[->-=.3,-<-=.7,happypink] (F) -- ++(0,2)
			node[right,pos=1/4] {$aa^{-1}$}
			node[right,pos=3/4] {$\transitionword a{h+2}$}
			;

			\unless\ifremoveedges
				\draw[very thick,-<-=.45,happypink] (F) -- ++(+1.5,0) node[below,midway] {$\transitionword a{h+1}$};
			\fi
		\fi
	\fi
\end{tikzpicture}
	\end{subfigure}
	\setboolean{pushdown}{true}
	\begin{subfigure}{0.4\textwidth}
		\begin{tikzpicture}
	\useasboundingbox (-2,-1) (5,7);
	\tikzset{->-/.style={decoration={markings,mark=at position #1+.01 with {\arrow{latex}}},postaction={decorate}}}
	\tikzset{-<-/.style={decoration={markings,mark=at position #1-.01 with {\arrowreversed{latex}}},postaction={decorate}}}
	\definecolor{happyblue}{RGB}{0,0,0}
	\definecolor{happypink}{RGB}{0,0,0}
	\tikzset{-|>-/.style={decoration={markings,mark=at position #1+.01 with {\arrow{latex}}},postaction={decorate}}}
	\tikzset{-<|-/.style={decoration={markings,mark=at position #1-.01 with {\arrowreversed{latex}}},postaction={decorate}}}

	\coordinate (A) at (0,0);
	\coordinate (B) at (0,3);
	\coordinate (C) at (0,6);
	\coordinate (D) at (3,6);
	\coordinate (E) at (3,3);
	\coordinate (F) at (3,0);

	\foreach\X in {A,B,C,D} {
			\foreach\Y in {A,B,C,D} {
					\coordinate (\X\Y) ($(\X)!1/3!(\Y)$);
				}
		}

	\unless\ifpushdown
		\draw[thick,-|>-=.55,happyblue] (A) -- (B) node[left,midway] {$\colword ab$};
		\draw[thick,-|>-=.55,happyblue] (B) -- (C) node[left,midway] {$\colword ab$};
		\draw[thick,->-=.55,happypink] (C) -- (D) node[above,midway] {$\colword ba$};
		\draw[thick,-|>-=.55,happyblue] (E) -- (D) node[right,midway] {$\colword ab$};
		\draw[thick,->-=.55,happypink] (F) -- (E) node[right,midway] {$\colword aa$};
		\draw[thick,->-=.55,happypink] (A) -- (F) node[above,midway] {$\colword ba$};
		\draw[thick,-|>-=.55,happyblue] (B) -- (E) node[above,midway] {$\colword bb$};
	\else
		\unless\ifcollapse
			\draw[-|>-=.3,-|>-=.7,happyblue] (A) -- +(0,2)
			node[left,pos=1/4] {$\transitionword bh$}
			node[left,pos=3/4] {$ab^{-1}$ }
			;
			\draw[-|>-=.3, -<|-=.7,happyblue] (B) ++(0,1) -- (C)
			node[left,pos=1/4] {$ab^{-1}$ }
			node[left,pos=3/4] {$\transitionword b{h+2}$}
			;

			\draw[-|>-=.55,happyblue] (B) ++(1,0) -- ++(1,0)
			node[above,pos=1/2] {$bb^{-1}$ }
			;
			\draw[dashed, -|>-=.55] (B) -- ++(0,1)
			node[left,midway] { $\transitionword b{h+1}$};
			\draw[dashed, -|>-=.55] (B) -- ++(1,0)
			node[above,pos=.6] { $\transitionword b{h+1}$};
			\draw[dashed, -|>-=.55] (B) -- ++(0,-1)
			node[left,midway] { $\transitionword b{h+1}$};
			\draw[dashed, -|>-=.55] (E) -- ++(-1,0)
			node[above,midway] { $\transitionword b{h+2}$};
			\draw[dashed, -|>-=.55] (E) -- ++(0,1)
			node[right,midway] { $\transitionword b{h+2}$};
		\else
			\draw[-|>-=.3,-|>-=.8,happyblue] (A) -- (B)
			node[left,pos=1/4] {$\transitionword bh$}
			node[left,pos=3/4] {$ab^{-1}$ }
			;
			\draw[-|>-=.3,-<|-=.7,happyblue] (B) -- (C)
			node[left,pos=1/4] {$ab^{-1}$ }
			node[left,pos=3/4] {$\transitionword b{h+2}$}
			;
			\draw[-|>-=.55,happyblue] (B) -- (E)
			node[above,midway] {$bb^{-1}$ }
			;
			\unless\ifremoveedges
				\draw[very thick,-<-=.45,happyblue] (B) -- ++(-1.5,0) node[above,midway] {$\transitionword b{h+1}$};
			\fi
		\fi
		\draw[->-=.22,->-=.55, -<-=.78,happypink] (C) -- (D)
		node[above,pos=1/6] {$\transitionword a{h+2}$}
		node[above,midway] {$ba^{-1}$ }
		node[above,pos=5/6] {$\transitionword a{h+3}$}
		;
		\unless\ifcollapse
			\draw[-|>-=.3, -<|-=.7,happyblue] (E) ++(0,1) -- (D)
			node[right,pos=1/4] {$ab^{-1}$ }
			node[right,pos=3/4] {$\transitionword b{h+3}$}
			;
			\draw[->-=.3, -<-=.7,happypink] (F) ++(0,1) -- (E)
			node[right,pos=3/4] {$\transitionword a{h+2}$}
			node[right,pos=1/4] {$aa^{-1}$ }
			;
			\draw[dashed, -|>-=.55] (F) -- ++(0,1)
			node[right,midway] { $\transitionword a{h+1}$};
			\draw[dashed, -|>-=.55] (F) -- ++(-1,0)
			node[below,midway] { $\transitionword a{h+1}$};

			\draw[->-=.3,->-=.7,happypink] (A) -- ++(2,0)
			node[above,pos=3/4] {$ba^{-1}$ }
			node[above,pos=1/4] {$\transitionword a{h}$}
			;
		\else
			\draw[-|>-=.3, -<|-=.7,happyblue] (E) -- (D)
			node[right,pos=1/4] {$ab^{-1}$ }
			node[right,pos=3/4] {$\transitionword b{h+3}$}
			;
			\draw[->-=.3,->-=.8,happypink] (A) -- (F)
			node[above,pos=1/4] {$\transitionword a{h}$}
			node[above,pos=3/4] {$ba^{-1}$ }
			;
			\draw[very thick,-<|-=.55,happyblue] (E) -- ++(0,-1)
			node[right,midway] {$\transitionword b{h+2}$}
			;
			\draw[->-=.3,-<-=.7,happypink] (F) -- ++(0,2)
			node[right,pos=1/4] {$aa^{-1}$}
			node[right,pos=3/4] {$\transitionword a{h+2}$}
			;

			\unless\ifremoveedges
				\draw[very thick,-<-=.45,happypink] (F) -- ++(+1.5,0) node[below,midway] {$\transitionword a{h+1}$};
			\fi
		\fi
	\fi
\end{tikzpicture}
	\end{subfigure}

	\setboolean{collapse}{true}
	\begin{subfigure}{0.4\textwidth}
		\begin{tikzpicture}
	\useasboundingbox (-2,-1) (5,7);
	\tikzset{->-/.style={decoration={markings,mark=at position #1+.01 with {\arrow{latex}}},postaction={decorate}}}
	\tikzset{-<-/.style={decoration={markings,mark=at position #1-.01 with {\arrowreversed{latex}}},postaction={decorate}}}
	\definecolor{happyblue}{RGB}{0,0,0}
	\definecolor{happypink}{RGB}{0,0,0}
	\tikzset{-|>-/.style={decoration={markings,mark=at position #1+.01 with {\arrow{latex}}},postaction={decorate}}}
	\tikzset{-<|-/.style={decoration={markings,mark=at position #1-.01 with {\arrowreversed{latex}}},postaction={decorate}}}

	\coordinate (A) at (0,0);
	\coordinate (B) at (0,3);
	\coordinate (C) at (0,6);
	\coordinate (D) at (3,6);
	\coordinate (E) at (3,3);
	\coordinate (F) at (3,0);

	\foreach\X in {A,B,C,D} {
			\foreach\Y in {A,B,C,D} {
					\coordinate (\X\Y) ($(\X)!1/3!(\Y)$);
				}
		}

	\unless\ifpushdown
		\draw[thick,-|>-=.55,happyblue] (A) -- (B) node[left,midway] {$\colword ab$};
		\draw[thick,-|>-=.55,happyblue] (B) -- (C) node[left,midway] {$\colword ab$};
		\draw[thick,->-=.55,happypink] (C) -- (D) node[above,midway] {$\colword ba$};
		\draw[thick,-|>-=.55,happyblue] (E) -- (D) node[right,midway] {$\colword ab$};
		\draw[thick,->-=.55,happypink] (F) -- (E) node[right,midway] {$\colword aa$};
		\draw[thick,->-=.55,happypink] (A) -- (F) node[above,midway] {$\colword ba$};
		\draw[thick,-|>-=.55,happyblue] (B) -- (E) node[above,midway] {$\colword bb$};
	\else
		\unless\ifcollapse
			\draw[-|>-=.3,-|>-=.7,happyblue] (A) -- +(0,2)
			node[left,pos=1/4] {$\transitionword bh$}
			node[left,pos=3/4] {$ab^{-1}$ }
			;
			\draw[-|>-=.3, -<|-=.7,happyblue] (B) ++(0,1) -- (C)
			node[left,pos=1/4] {$ab^{-1}$ }
			node[left,pos=3/4] {$\transitionword b{h+2}$}
			;

			\draw[-|>-=.55,happyblue] (B) ++(1,0) -- ++(1,0)
			node[above,pos=1/2] {$bb^{-1}$ }
			;
			\draw[dashed, -|>-=.55] (B) -- ++(0,1)
			node[left,midway] { $\transitionword b{h+1}$};
			\draw[dashed, -|>-=.55] (B) -- ++(1,0)
			node[above,pos=.6] { $\transitionword b{h+1}$};
			\draw[dashed, -|>-=.55] (B) -- ++(0,-1)
			node[left,midway] { $\transitionword b{h+1}$};
			\draw[dashed, -|>-=.55] (E) -- ++(-1,0)
			node[above,midway] { $\transitionword b{h+2}$};
			\draw[dashed, -|>-=.55] (E) -- ++(0,1)
			node[right,midway] { $\transitionword b{h+2}$};
		\else
			\draw[-|>-=.3,-|>-=.8,happyblue] (A) -- (B)
			node[left,pos=1/4] {$\transitionword bh$}
			node[left,pos=3/4] {$ab^{-1}$ }
			;
			\draw[-|>-=.3,-<|-=.7,happyblue] (B) -- (C)
			node[left,pos=1/4] {$ab^{-1}$ }
			node[left,pos=3/4] {$\transitionword b{h+2}$}
			;
			\draw[-|>-=.55,happyblue] (B) -- (E)
			node[above,midway] {$bb^{-1}$ }
			;
			\unless\ifremoveedges
				\draw[very thick,-<-=.45,happyblue] (B) -- ++(-1.5,0) node[above,midway] {$\transitionword b{h+1}$};
			\fi
		\fi
		\draw[->-=.22,->-=.55, -<-=.78,happypink] (C) -- (D)
		node[above,pos=1/6] {$\transitionword a{h+2}$}
		node[above,midway] {$ba^{-1}$ }
		node[above,pos=5/6] {$\transitionword a{h+3}$}
		;
		\unless\ifcollapse
			\draw[-|>-=.3, -<|-=.7,happyblue] (E) ++(0,1) -- (D)
			node[right,pos=1/4] {$ab^{-1}$ }
			node[right,pos=3/4] {$\transitionword b{h+3}$}
			;
			\draw[->-=.3, -<-=.7,happypink] (F) ++(0,1) -- (E)
			node[right,pos=3/4] {$\transitionword a{h+2}$}
			node[right,pos=1/4] {$aa^{-1}$ }
			;
			\draw[dashed, -|>-=.55] (F) -- ++(0,1)
			node[right,midway] { $\transitionword a{h+1}$};
			\draw[dashed, -|>-=.55] (F) -- ++(-1,0)
			node[below,midway] { $\transitionword a{h+1}$};

			\draw[->-=.3,->-=.7,happypink] (A) -- ++(2,0)
			node[above,pos=3/4] {$ba^{-1}$ }
			node[above,pos=1/4] {$\transitionword a{h}$}
			;
		\else
			\draw[-|>-=.3, -<|-=.7,happyblue] (E) -- (D)
			node[right,pos=1/4] {$ab^{-1}$ }
			node[right,pos=3/4] {$\transitionword b{h+3}$}
			;
			\draw[->-=.3,->-=.8,happypink] (A) -- (F)
			node[above,pos=1/4] {$\transitionword a{h}$}
			node[above,pos=3/4] {$ba^{-1}$ }
			;
			\draw[very thick,-<|-=.55,happyblue] (E) -- ++(0,-1)
			node[right,midway] {$\transitionword b{h+2}$}
			;
			\draw[->-=.3,-<-=.7,happypink] (F) -- ++(0,2)
			node[right,pos=1/4] {$aa^{-1}$}
			node[right,pos=3/4] {$\transitionword a{h+2}$}
			;

			\unless\ifremoveedges
				\draw[very thick,-<-=.45,happypink] (F) -- ++(+1.5,0) node[below,midway] {$\transitionword a{h+1}$};
			\fi
		\fi
	\fi
\end{tikzpicture}
	\end{subfigure}
	\setboolean{removeedges}{true}
	\begin{subfigure}{0.4\textwidth}
		\begin{tikzpicture}
	\useasboundingbox (-2,-1) (5,7);
	\tikzset{->-/.style={decoration={markings,mark=at position #1+.01 with {\arrow{latex}}},postaction={decorate}}}
	\tikzset{-<-/.style={decoration={markings,mark=at position #1-.01 with {\arrowreversed{latex}}},postaction={decorate}}}
	\definecolor{happyblue}{RGB}{0,0,0}
	\definecolor{happypink}{RGB}{0,0,0}
	\tikzset{-|>-/.style={decoration={markings,mark=at position #1+.01 with {\arrow{latex}}},postaction={decorate}}}
	\tikzset{-<|-/.style={decoration={markings,mark=at position #1-.01 with {\arrowreversed{latex}}},postaction={decorate}}}

	\coordinate (A) at (0,0);
	\coordinate (B) at (0,3);
	\coordinate (C) at (0,6);
	\coordinate (D) at (3,6);
	\coordinate (E) at (3,3);
	\coordinate (F) at (3,0);

	\foreach\X in {A,B,C,D} {
			\foreach\Y in {A,B,C,D} {
					\coordinate (\X\Y) ($(\X)!1/3!(\Y)$);
				}
		}

	\unless\ifpushdown
		\draw[thick,-|>-=.55,happyblue] (A) -- (B) node[left,midway] {$\colword ab$};
		\draw[thick,-|>-=.55,happyblue] (B) -- (C) node[left,midway] {$\colword ab$};
		\draw[thick,->-=.55,happypink] (C) -- (D) node[above,midway] {$\colword ba$};
		\draw[thick,-|>-=.55,happyblue] (E) -- (D) node[right,midway] {$\colword ab$};
		\draw[thick,->-=.55,happypink] (F) -- (E) node[right,midway] {$\colword aa$};
		\draw[thick,->-=.55,happypink] (A) -- (F) node[above,midway] {$\colword ba$};
		\draw[thick,-|>-=.55,happyblue] (B) -- (E) node[above,midway] {$\colword bb$};
	\else
		\unless\ifcollapse
			\draw[-|>-=.3,-|>-=.7,happyblue] (A) -- +(0,2)
			node[left,pos=1/4] {$\transitionword bh$}
			node[left,pos=3/4] {$ab^{-1}$ }
			;
			\draw[-|>-=.3, -<|-=.7,happyblue] (B) ++(0,1) -- (C)
			node[left,pos=1/4] {$ab^{-1}$ }
			node[left,pos=3/4] {$\transitionword b{h+2}$}
			;

			\draw[-|>-=.55,happyblue] (B) ++(1,0) -- ++(1,0)
			node[above,pos=1/2] {$bb^{-1}$ }
			;
			\draw[dashed, -|>-=.55] (B) -- ++(0,1)
			node[left,midway] { $\transitionword b{h+1}$};
			\draw[dashed, -|>-=.55] (B) -- ++(1,0)
			node[above,pos=.6] { $\transitionword b{h+1}$};
			\draw[dashed, -|>-=.55] (B) -- ++(0,-1)
			node[left,midway] { $\transitionword b{h+1}$};
			\draw[dashed, -|>-=.55] (E) -- ++(-1,0)
			node[above,midway] { $\transitionword b{h+2}$};
			\draw[dashed, -|>-=.55] (E) -- ++(0,1)
			node[right,midway] { $\transitionword b{h+2}$};
		\else
			\draw[-|>-=.3,-|>-=.8,happyblue] (A) -- (B)
			node[left,pos=1/4] {$\transitionword bh$}
			node[left,pos=3/4] {$ab^{-1}$ }
			;
			\draw[-|>-=.3,-<|-=.7,happyblue] (B) -- (C)
			node[left,pos=1/4] {$ab^{-1}$ }
			node[left,pos=3/4] {$\transitionword b{h+2}$}
			;
			\draw[-|>-=.55,happyblue] (B) -- (E)
			node[above,midway] {$bb^{-1}$ }
			;
			\unless\ifremoveedges
				\draw[very thick,-<-=.45,happyblue] (B) -- ++(-1.5,0) node[above,midway] {$\transitionword b{h+1}$};
			\fi
		\fi
		\draw[->-=.22,->-=.55, -<-=.78,happypink] (C) -- (D)
		node[above,pos=1/6] {$\transitionword a{h+2}$}
		node[above,midway] {$ba^{-1}$ }
		node[above,pos=5/6] {$\transitionword a{h+3}$}
		;
		\unless\ifcollapse
			\draw[-|>-=.3, -<|-=.7,happyblue] (E) ++(0,1) -- (D)
			node[right,pos=1/4] {$ab^{-1}$ }
			node[right,pos=3/4] {$\transitionword b{h+3}$}
			;
			\draw[->-=.3, -<-=.7,happypink] (F) ++(0,1) -- (E)
			node[right,pos=3/4] {$\transitionword a{h+2}$}
			node[right,pos=1/4] {$aa^{-1}$ }
			;
			\draw[dashed, -|>-=.55] (F) -- ++(0,1)
			node[right,midway] { $\transitionword a{h+1}$};
			\draw[dashed, -|>-=.55] (F) -- ++(-1,0)
			node[below,midway] { $\transitionword a{h+1}$};

			\draw[->-=.3,->-=.7,happypink] (A) -- ++(2,0)
			node[above,pos=3/4] {$ba^{-1}$ }
			node[above,pos=1/4] {$\transitionword a{h}$}
			;
		\else
			\draw[-|>-=.3, -<|-=.7,happyblue] (E) -- (D)
			node[right,pos=1/4] {$ab^{-1}$ }
			node[right,pos=3/4] {$\transitionword b{h+3}$}
			;
			\draw[->-=.3,->-=.8,happypink] (A) -- (F)
			node[above,pos=1/4] {$\transitionword a{h}$}
			node[above,pos=3/4] {$ba^{-1}$ }
			;
			\draw[very thick,-<|-=.55,happyblue] (E) -- ++(0,-1)
			node[right,midway] {$\transitionword b{h+2}$}
			;
			\draw[->-=.3,-<-=.7,happypink] (F) -- ++(0,2)
			node[right,pos=1/4] {$aa^{-1}$}
			node[right,pos=3/4] {$\transitionword a{h+2}$}
			;

			\unless\ifremoveedges
				\draw[very thick,-<-=.45,happypink] (F) -- ++(+1.5,0) node[below,midway] {$\transitionword a{h+1}$};
			\fi
		\fi
	\fi
\end{tikzpicture}
	\end{subfigure}
	\captionsetup{singlelinecheck=off}
	\caption{Example of the pushdown of a coloured diagram. The top left picture shows a coloured diagram $\bolddiag$. To obtain the $h$-pushdown, we first construct a diagram by replacing each edge of $\bolddiag$ with a path labelled by the pushdown of the label on that edge, as shown in the top right picture. Next, we fold the paths originating from the same vertex that are labelled with the same transition word, as illustrated in the bottom left picture. Finally, wherever we have collapsed all the paths arising from a vertex, we can remove the resulting path; see the bottom right picture.}
	\label{fig:pushdown}
\end{figure}

To ensure that the $h$-pushdown of a coloured diagram is well-defined, we must check that the labels of the $h$-pushdown of the bounded regions are null-homotopic words in $\raaggens$. This follows directly from the next result.

\begin{lemma}\label{labels-of-regions-of-pushdown}
	Let $h \in \Z$. Let $\mathbf D$ be a coloured diagram for a null-homotopic coloured word. Let $\mathbf w = \colwordproduct k$ be the word read on a (possibly unbounded) region of $ \mathbf D $, starting from a point $p$, and let $u$ be the alternating word read on the corresponding region of $\pushdown[h]{\mathbf D} $. If $a_k \neq a_1$, then $u = \pushdown[h+\height(p)]{\mathbf w}$. Otherwise, if $a_k = a_1 = a $, then $\pushdown[h+\height(p)]{\mathbf w}=\transitionword a{h+\height(p)} u \transitionword a{h+\height(p)}^{-1}$. In either case, the word $u$ is null-homotopic.
\end{lemma}
\begin{proof}
	For $ i \in \range[0]{k} $, let $ h_i = h + \height(p) + \height(w_1 \cdots w_i) $. Since $\mathbf w$ is the label of a (possibly unbounded) region in $\bolddiag$, it is null-homotopic, so $h_0 = h_k = h+\height(p)$. From \cref{def:pushdown-of-coloured-diagram}, we get that, if $a_1 \neq a_k$, then
	\[
		u = \prod_{i=1}^k \transitionword {a_i}{h_{i-1}} \balance{w_i}{a_i} \transitionword{a_i}{h_i}^{-1} = \pushdown[h+\height(p)]{\mathbf w}.
	\]
	The key point is that all the transition words appearing between letters of the same colour disappear after folding. The first and the last transition words are not folded because of the assumption $ a_1 \neq a_k $. Otherwise, if they are folded, that is, if $a_k=a_1=a$, we obtain
	\[
		\pushdown[h+\height(p)]{\mathbf w}=\transitionword a{h+\height(p)} u \transitionword a{h+\height(p)}^{-1}
	\]

	In both cases, the fact that $u$ is null-homotopic follows directly from \cref{pushdown-as-group-element}, since $\boldword$ is null-homotopic.
\end{proof}

As highlighted in \cref{labels-of-regions-of-pushdown}, it is often convenient to assume that the first and last colour of a null-homotopic coloured word $ \mathbf w $ are different.
This is not a critical assumption, since monochromatic words are simple enough that we do not need to examine their coloured diagrams.
If $ \mathbf w $ is not monochromatic, we can always achieve this assumption by considering a cyclic conjugate of $ \mathbf w $. Geometrically, if we imagine $ \mathbf w $ drawn on a circle, this corresponds to reading the word starting from a vertex adjacent to edges of different colours. Therefore, we get the following consequence.

\begin{corollary}\noproof\label{labels-of-monochromatic-regions-of-pushdown}
	Let $h \in \Z$. If $\mathbf D$ is a coloured diagram for a null-homotopic $k$-coloured word $\mathbf w = \colwordproduct k$ with $a_k \neq a_1$, then $\pushdown[h]{\mathbf D}$ is an alternating diagram for $\pushdown[h]{\mathbf w}$. If the boundary word $\mathbf w = \colword wa$ is monochromatic, then $\pushdown[h]{\mathbf D}$ is an alternating diagram for $\balance{w}{a}$. \qedhere
\end{corollary}

We are now able to prove \cref{pushdown-from-bounded-height}, which we recall here for the reader's convenience.

\techlemma*

\begin{proof}%
	Let $\mathbf D$ be the coloured diagram obtained from $D$ by replacing each label $s$ with $ \colword ss $. The pushdown of $\mathbf D$ is an alternating diagram $D'$ for a word $w'$ that is freely equivalent to $w$ by \cref{pushdown-properties}~\eqref{item: pushdown of alternating}. After performing appropriate foldings in $D'$, we obtain an alternating diagram $D_w$ for $w$. Since the bounded regions of $\mathbf D$ are labelled with words of length $4$, by \cref{labels-of-regions-of-pushdown,length-of-pushdown}, the bounded regions of $D_w$ are labelled with the words whose lengths are bounded by a constant $M_H$ depending only on $H$.
\end{proof}

Recall that our strategy for obtaining upper bounds on the Dehn function of $\bbg \Gamma$ is to produce an algorithm that takes an alternating word $u$ as input and outputs an almost-flat van Kampen diagram whose almost-flat area satisfies the desired upper bound.

We introduced coloured diagrams to construct this algorithm in the case where $u = \pushdown[h]\boldword$ for some null-homotopic coloured word $\mathbf w$ and some $h \in \Z$. In this case, if $\bolddiag$ is a coloured diagram for $\boldword$, then $\pushdown[h]\bolddiag$ is an alternating diagram for $u$. By \cref{labels-of-regions-of-pushdown}, the bounded regions of $\pushdown[h]\bolddiag$ are labelled by the pushdowns of the labels of the bounded regions of $\bolddiag$. Therefore, to estimate $\flatarea u$, it suffices to estimate the almost-flat area of the label of each bounded region of $\pushdown[h]\bolddiag$.

By \cref{lem:area-of-protodiagram}, the estimates are particularly nice if we have a uniform bound on the density of
$\pushdown[h]\bolddiag$. Note that the coloured diagram $\bolddiag$, which is by definition a coarse diagram over the generating set $ \colraaggens $, also has a well-defined density. However, estimating the density of $\bolddiag$ is not enough to obtain an estimate of the density of its pushdown, as one needs to take into account for the lengths of the transition words, which appear at every vertex incident to edges of different colours.

\begin{definition}
	Let $\mathbf D$ be a coloured diagram. A vertex of $\mathbf D$ is called \emph{polychromatic} if it is incident to at least two edges of different colours.
	It is called \emph{$ \partial $-polychromatic} if it is on the boundary of $ \mathbf D $ and incident to two consecutive boundary edges of different colours.
\end{definition}

\begin{proposition}\label{sparse-pushdown}
	For every $M > 0$, there exists $C_M > 0$ such that the following holds.
	Let $\mathbf D$ be a coloured diagram, and for every polychromatic vertex $p$, let $\gamma_p$ be a path in $\bolddiag$ from $p$ to an arbitrary $\partial$-polychromatic vertex (if $p$ is $\partial$-polychromatic, then we can choose $\gamma_p$ to be the constant path). Suppose that
	\begin{enumerate}
		\item for every polychromatic vertex $p$ in $\bolddiag$, we have $\deg p \leq M$, and \label{item:bound-degree}
		\item for every edge $e$ in $\bolddiag$, we have $\cardinality{ \set{ \gamma_p \mid e \in \gamma_p} } \leq M$, and \label{item:bound-paths-through-edge}
		\item $\density \bolddiag \leq M$.\label{item:bound-density-coloured-diagram}
	\end{enumerate}
	Then $\density{\pushdown[h]{\mathbf D}} \leq C_M$ for every $h \in \Z$.
\end{proposition}

\begin{proof}
	Recall that by \cref{density-vs-num-edges}, it suffices to estimate the number of edges $\numedges{\pushdown[h]\bolddiag}$ for the pushdown diagram. We assume that there is at least one $\partial$-polychromatic vertex; otherwise, there would not be any polychromatic vertices, so the labels of all bounded regions of $\bolddiag$ would be monochromatic, and it would follow from \cref{labels-of-monochromatic-regions-of-pushdown} that $\density\bolddiag = \density{\pushdown[h]\bolddiag}$.

	Let $\boldword = \colwordproduct k$ be the boundary word of $\bolddiag$. Since the statement is independent on the choice of base point of $\bolddiag$, we may assume that the base point is $\partial$-polychromatic, so $a_k \neq a_1$. In particular, the diagram $\bolddiag$ has $\partial$-polychromatic vertices $p_1, \dots, p_k$, where $k\geq2$, and $p_i$ is at the end of $\colword{w_i}{a_i}$. 

	Recall that an edge in $\bolddiag$ from $p$ to $q$ labelled by $\colword sa$ is replaced by a path in $\pushdown[h]\bolddiag$ labelled by $\pushdown[h+\height(p)]{\colword sa} = \transitionword a{h+\height(p)} sa^{-1} \transitionword a {h+\height(q)}^{-1}$. Call the two edges in the middle labelled with $sa^{-1}$ \emph{central edges}, and the ones in the paths labelled with $\transitionword a{\height(p)}$ and $\transitionword a{\height(q)}$ \emph{transition edges}. Denote by $Z$ and $T$ the number of central and transition edges, respectively.

	For every edge of $ \mathbf D $, we get two central edges in $\pushdown[h]{\mathbf D}$, so
	\[
		Z = 2\numedges{\bolddiag}.
	\]
	If a vertex $p$ is not polychromatic, then all the transition edges arising from it get removed. Otherwise, the vertex $p$ gives rise to at most $2M \cdot \numvertices \Gamma \cdot \length{h+\height(p)}$ transition edges by \cref{length-transition-word}. Therefore, we obtain
	\[
		T \leq  2M \cdot \numvertices \Gamma \cdot \sum_{p \text{ polychr.}} \abs{h+\height(p)}.
	\]
	For $i \in \range k$, let $P_i$ be the set of polychromatic vertices $p$ connected to $p_i$ by $\gamma_p$. It follows from the assumptions \eqref{item:bound-degree} and \eqref{item:bound-paths-through-edge} that
	\[
		\cardinality{P_i} \leq M^2 + 1,
	\]
	since there are at most $M$ paths coming through each of the edges incident to $p_i$, plus the constant path from $p_i$ to itself. For $p \in P_i$, we can estimate
	\[
		\abs{\height(p)-\height(p_i)} \leq \length{\gamma_p}.
	\]

	Putting the estimates together yields
	\begin{align*}
		\numedges{\pushdown[h]\bolddiag} & = Z + T                                                                                                                                                              \\
		                                 & \leq 2\numedges\bolddiag  + 2M \cdot \numvertices \Gamma \cdot \sum_{p \text{ polychr.}} \abs{h+\height(p)}                                                          \\
		                                 & = 2\numedges\bolddiag  + 2M \cdot \numvertices \Gamma \cdot \sum_{i=1}^k \sum_{p \in P_i} \abs{h+\height(p)}                                                         \\
		                                 & \leq 2\numedges\bolddiag  + 2M \cdot \numvertices \Gamma \cdot \sum_{i=1}^k \sum_{p \in P_i} \abs{\gamma_p} + \abs{h+\height(p_i)}                                   \\
		                                 & \leq 2\numedges\bolddiag  + 2M \cdot \numvertices \Gamma \cdot M \numedges\bolddiag + 2M \cdot \numvertices \Gamma (M^2 + 1)\sum_{i=1}^k \abs{h+\height(p_i)}        \\
		                                 & = (1 + M^2 \numvertices \Gamma) (\length \boldword \density\bolddiag + \length \boldword) + 2M \cdot \numvertices \Gamma (M^2 + 1)\sum_{i=1}^k \abs{h+\height(p_i)},
	\end{align*}
	where the last equality follows from \cref{density-vs-num-edges}, which tells us $2\numedges\bolddiag = \length \boldword \cdot \density\bolddiag + \length \boldword$.

	Finally, by \cref{length-of-pushdown}, we have $\length{\pushdown[h]\boldword} \geq 2 \length{\boldword}+2\sum_{i=1}^k \abs{h+\height(p_i)}$. Therefore, we can deduce
	\begin{align*}
		\density{\pushdown[h]\bolddiag} & = \frac{2\numedges{\pushdown[h]\bolddiag}}{\length{\pushdown[h]\boldword}} - 1                             \\
		                                & \leq (1+M^2 \numvertices \Gamma) \cdot (\density\bolddiag+1)+ 2M(M^2 + 1) \numvertices \Gamma - 1          \\
		                                & \leq 3 \numvertices \Gamma \cdot M^3 + \numvertices \Gamma \cdot M^2 + (1 + 2\numvertices \Gamma) \cdot M,
	\end{align*}
	where the last estimate follows from the assumption \eqref{item:bound-density-coloured-diagram}.
\end{proof}

\subsection{Cutting along corridors}\label{sec:cut-along-corridors}

In this section, we develop a method to ``cut'' a coloured word along corridors to produce a coloured diagram.

Let $ \mathbf{w}$ and $ \mathbf{w'} $ be coloured words representing the same element in $\raag\Gamma$, and let $w$ and $w'$ denote their underlying words, respectively. Let $D$ be a minimal-area van Kampen diagram for $w w'^{-1}$, and let $p_0$ be the starting vertex of $w$, which we set as the base point of $D$.
Let $k \in \N$, and for $i \in \range {k-1}$, let $C_i$ be an $a_i$-corridor of $D$ connecting an edge of $w$ to an edge of $ w' $. Assume that the corridors $C_i$ are pairwise disjoint and are ordered by increasing distance from $p_0$; see the left diagram in \cref{fig:cut-corridors}.

\newboolean{coloured}
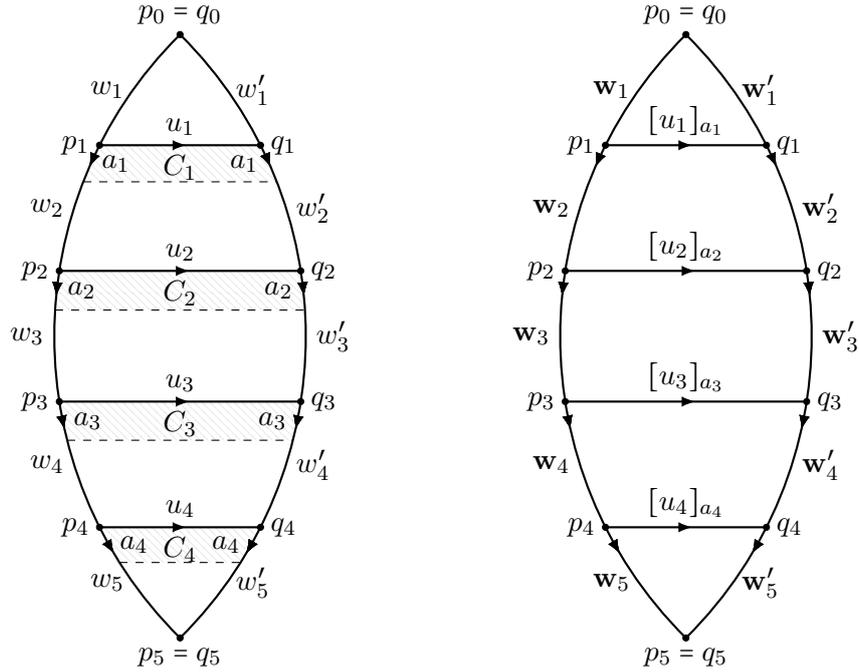
\begin{figure}
	\setboolean{coloured}{false}
	\begin{subfigure}{0.4\textwidth}
		\centering
		\begin{tikzpicture}
	\tikzset{|->-|/.style={decoration={markings,
			mark=at position 0 with {\arrow[ultra thin]{|}},
			mark=at position #1 with {\arrow{latex}},
			mark=at position 1 with {\arrow[ultra thin]{|}},
		},postaction={decorate}}}
	\tikzset{->-/.style={decoration={markings,mark=at position #1+.01 with {\arrow{latex}}},postaction={decorate}}}
	\tikzset{-<-/.style={decoration={markings,mark=at position #1-.01 with {\arrowreversed{latex}}},postaction={decorate}}}

	\coordinate (p0) at (0,0);
	\coordinate (p5) at (0,-8);
	\coordinate (q0) at (p0);
	\coordinate (q5) at (p5);

	\draw[thick, ->-=.23,->-=.43,->-=.63,->-=.83](p0) to[out=-135, in=135]
	coordinate[pos=0.2](p1)
	coordinate[pos=0.4](p2)
	coordinate[pos=0.6](p3)
	coordinate[pos=0.8](p4)
	coordinate[pos=0.26](r1)
	coordinate[pos=0.46](r2)
	coordinate[pos=0.66](r3)
	coordinate[pos=0.86](r4)
	(p5);
	\draw[thick, ->-=.23,->-=.43,->-=.63,->-=.83](p0) to[out=-45, in=+45]
	coordinate[pos=0.2](q1)
	coordinate[pos=0.4](q2)
	coordinate[pos=0.6](q3)
	coordinate[pos=0.8](q4)
	coordinate[pos=0.26](s1)
	coordinate[pos=0.46](s2)
	coordinate[pos=0.66](s3)
	coordinate[pos=0.86](s4)
	(p5);

	\node[above] at (p0) {$p_0=q_0$};
	\node[below] at (p5) {$p_5=q_5$};
	\draw[fill=black] (p0) circle (.04);
	\draw[fill=black] (p5) circle (.04);

	\foreach \ii in {1,...,4} {
			\draw[thick,->-=.53] (p\ii) -- (q\ii) coordinate[midway] (m\ii);
			\unless\ifcoloured
				\draw[dashed] (r\ii) -- (s\ii);
				\node[right] at ($(p\ii)!.5!(r\ii)$) {$a_\ii$};
				\node[left] at ($(q\ii)!.5!(s\ii)$) {$a_\ii$};

			\fi

			\node[left] at (p\ii) {$p_\ii$};
			\node[right] at (q\ii) {$q_\ii$};
			\draw[fill=black] (p\ii) circle (.04);
			\draw[fill=black] (q\ii) circle (.04);

			\ifcoloured
				\node[above] at (m\ii) {$\colword{u_\ii}{a_\ii}$};

			\else
				\path[pattern=north west lines, opacity=0.3] (p\ii) -- (q\ii) -- (s\ii) -- (r\ii);
				\node[below] at (m\ii) {$C_\ii$};
				\node[above] at (m\ii) {${u_\ii}$};
			\fi

		}

	\foreach \ii in {0,...,4} {
			\def\jj{\the\numexpr\ii+1\relax}
			\ifcoloured
				\node[left=2] at ($(p\ii)!.5!(p\jj)$) {$\mathbf w_\jj$};
				\node[right=2] at ($(q\ii)!.5!(q\jj)$) {$\mathbf w'_\jj$};
			\else

				\node[left=2] at ($(p\ii)!.5!(p\jj)$) {$ w_\jj$};
				\node[right=2] at ($(q\ii)!.5!(q\jj)$) {$ w'_\jj$};
			\fi
		}

\end{tikzpicture}
	\end{subfigure}
	\setboolean{coloured}{true}
	\begin{subfigure}{0.4\textwidth}
		\centering
		\begin{tikzpicture}
	\tikzset{|->-|/.style={decoration={markings,
			mark=at position 0 with {\arrow[ultra thin]{|}},
			mark=at position #1 with {\arrow{latex}},
			mark=at position 1 with {\arrow[ultra thin]{|}},
		},postaction={decorate}}}
	\tikzset{->-/.style={decoration={markings,mark=at position #1+.01 with {\arrow{latex}}},postaction={decorate}}}
	\tikzset{-<-/.style={decoration={markings,mark=at position #1-.01 with {\arrowreversed{latex}}},postaction={decorate}}}

	\coordinate (p0) at (0,0);
	\coordinate (p5) at (0,-8);
	\coordinate (q0) at (p0);
	\coordinate (q5) at (p5);

	\draw[thick, ->-=.23,->-=.43,->-=.63,->-=.83](p0) to[out=-135, in=135]
	coordinate[pos=0.2](p1)
	coordinate[pos=0.4](p2)
	coordinate[pos=0.6](p3)
	coordinate[pos=0.8](p4)
	coordinate[pos=0.26](r1)
	coordinate[pos=0.46](r2)
	coordinate[pos=0.66](r3)
	coordinate[pos=0.86](r4)
	(p5);
	\draw[thick, ->-=.23,->-=.43,->-=.63,->-=.83](p0) to[out=-45, in=+45]
	coordinate[pos=0.2](q1)
	coordinate[pos=0.4](q2)
	coordinate[pos=0.6](q3)
	coordinate[pos=0.8](q4)
	coordinate[pos=0.26](s1)
	coordinate[pos=0.46](s2)
	coordinate[pos=0.66](s3)
	coordinate[pos=0.86](s4)
	(p5);

	\node[above] at (p0) {$p_0=q_0$};
	\node[below] at (p5) {$p_5=q_5$};
	\draw[fill=black] (p0) circle (.04);
	\draw[fill=black] (p5) circle (.04);

	\foreach \ii in {1,...,4} {
			\draw[thick,->-=.53] (p\ii) -- (q\ii) coordinate[midway] (m\ii);
			\unless\ifcoloured
				\draw[dashed] (r\ii) -- (s\ii);
				\node[right] at ($(p\ii)!.5!(r\ii)$) {$a_\ii$};
				\node[left] at ($(q\ii)!.5!(s\ii)$) {$a_\ii$};

			\fi

			\node[left] at (p\ii) {$p_\ii$};
			\node[right] at (q\ii) {$q_\ii$};
			\draw[fill=black] (p\ii) circle (.04);
			\draw[fill=black] (q\ii) circle (.04);

			\ifcoloured
				\node[above] at (m\ii) {$\colword{u_\ii}{a_\ii}$};

			\else
				\path[pattern=north west lines, opacity=0.3] (p\ii) -- (q\ii) -- (s\ii) -- (r\ii);
				\node[below] at (m\ii) {$C_\ii$};
				\node[above] at (m\ii) {${u_\ii}$};
			\fi

		}

	\foreach \ii in {0,...,4} {
			\def\jj{\the\numexpr\ii+1\relax}
			\ifcoloured
				\node[left=2] at ($(p\ii)!.5!(p\jj)$) {$\mathbf w_\jj$};
				\node[right=2] at ($(q\ii)!.5!(q\jj)$) {$\mathbf w'_\jj$};
			\else

				\node[left=2] at ($(p\ii)!.5!(p\jj)$) {$ w_\jj$};
				\node[right=2] at ($(q\ii)!.5!(q\jj)$) {$ w'_\jj$};
			\fi
		}

\end{tikzpicture}
	\end{subfigure}
	\caption{On the left is a minimal-area van Kampen diagram for $ ww'^{-1} $, cut by a family of non-crossing corridors. On the right is the resulting coloured diagram for $ \mathbf w \mathbf w'^{-1} $.}
	\label{fig:cut-corridors}
\end{figure}

Let $\gamma_i$ be a side of $C_i$, and let $p_i$ and $q_i$ be its endpoints on $w$ and $w'$, respectively. For consistency of notation, we also denote the base point $p_0$ as $q_0$, and let $p_k=q_k$ be the other point where $w$ and $w'$ meet. Let $w_i$ be the word read along the path from $p_{i-1}$ to $p_i$, $w'_i$ the word read along the path from $q_{i-1}$ to $q_i$, and $u_i$ the word read along $ \gamma_i $ from $p_i$ to $q_i$.

The decompositions $ w=w_1 \cdots w_k $ and $ w' = w'_1 \cdots w'_k $ naturally induce decompositions $ \mathbf w = \mathbf w_1  \cdots \mathbf w_k  $ and $ \mathbf w'= \mathbf w' _1 \cdots \mathbf w' _k $ into coloured subwords, where $w_i$ and $w'_i$ are the underlying words of the coloured words $\mathbf w_i$ and $\mathbf w'_i$, respectively.
By the definition of corridor, all letters of $u_i$ commute with $a_i$, so $u_i$ is the underlying word of the monochromatic word $ \colword{u_i}{a_i} $.

We can now construct a coloured diagram for $\mathbf{w}\mathbf{w'}^{-1}$ by cutting along the $\gamma_i$: this coloured diagram has $k$ bounded regions, $R_1, \dots, R_k$, where the boundary word of $R_i$, when read starting from $p_i$, is $\mathbf w_i \colword{u_i}{a_i} \mathbf{w}_i'^{-1} \colword{u_{i-1}^{-1}}{a_{i-1}}$, as shown in the right diagram in \cref{fig:cut-corridors}. We apply this construction to prove the following result.

\begin{lemma}\label{cut-efficient-diagram}
	There exists a constant $C > 0$ such that the following holds.
	Let $ \mathbf w $ be an efficient coloured word, and let $ \mathbf w' $ be another coloured word representing the same element as $ \mathbf w $.
	There exists a coloured diagram $\mathbf D$ for $\mathbf w\mathbf w'^{-1}$ with $ \density{\pushdown[h]{\mathbf D}} \leq C $ for every $h \in \Z$, such that the label on every region of $ \mathbf D $ decomposes as a product $ \mathbf u \mathbf u'^{-1} $, where $\mathbf u$ is at most $5$-coloured, and $\mathbf u'$ is a subword of $ \mathbf w' $.
\end{lemma}

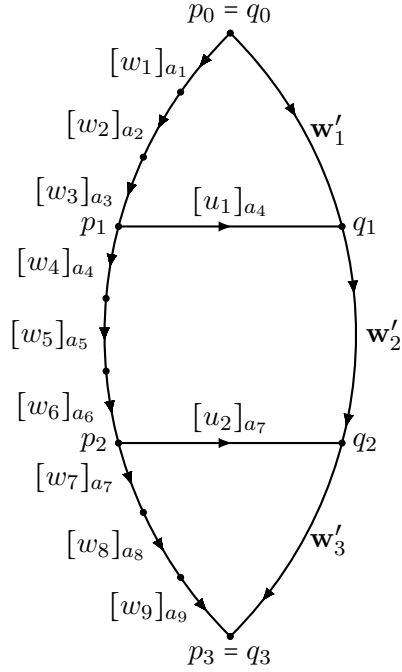
\begin{figure}
	\centering
	\begin{tikzpicture}
	\tikzset{|->-|/.style={decoration={markings,
			mark=at position 0 with {\arrow[ultra thin]{|}},
			mark=at position #1 with {\arrow{latex}},
			mark=at position 1 with {\arrow[ultra thin]{|}},
		},postaction={decorate}}}
	\tikzset{->-/.style={decoration={markings,mark=at position #1+.01 with {\arrow{latex}}},postaction={decorate}}}
	\tikzset{-<-/.style={decoration={markings,mark=at position #1-.01 with {\arrowreversed{latex}}},postaction={decorate}}}

	\coordinate (p0) at (0,0);
	\coordinate (p9) at (0,-8);
	\coordinate (q0) at (p0);
	\coordinate (q9) at (p9);

	\draw[thick, ->-=.06,->-=.17,->-=.28,->-=.39,->-=.50,->-=.61,->-=.72,->-=.83,->-=.94](p0) to[out=-135, in=135]
	coordinate[pos=1/9](p1)
	coordinate[pos=2/9](p2)
	coordinate[pos=3/9](p3)
	coordinate[pos=4/9](p4)
	coordinate[pos=5/9](p5)
	coordinate[pos=6/9](p6)
	coordinate[pos=7/9](p7)
	coordinate[pos=8/9](p8)
	(p9);
	\draw[thick, ->-=.15,->-=.43,->-=.63,->-=.92](p0) to[out=-45, in=+45]
	coordinate[pos=3/9](q3)
	coordinate[pos=6/9](q6)
	(p9);

	\node[above] at (p0) {$p_0=q_0$};
	\node[below] at (p9) {$p_3=q_3$};
	\draw[fill=black] (p0) circle (.04);
	\draw[fill=black] (p9) circle (.04);

	\foreach \ii in {1,2} {
			\def\jj{\the\numexpr\ii*3\relax}
			\draw[thick][->-=.5] (p\jj) -- (q\jj) coordinate[midway] (m\jj);

			\node[left] at (p\jj) {$p_\ii$};
			\node[right] at (q\jj) {$q_\ii$};

			\node[above] at (m\jj) {$\colword{u_\ii}{a_\the\numexpr\ii*3+1\relax}$};
			\draw[fill=black] (q\jj) circle (.04);
		}
	\foreach \jj in {1, ..., 8} {
			\draw[fill=black] (p\jj) circle (.04);
		}

	\foreach \ii in {0,...,8} {
			\def\jj{\the\numexpr\ii+1\relax}
			\node[left=2] at ($(p\ii)!.5!(p\jj)$) {$\colword{w_{\jj}}{a_{\jj}}$};
		}
	\foreach \ii in {1,2,3} {
	\def\jj{\the\numexpr3*\ii-3\relax}
	\def\kk{\the\numexpr3*\ii\relax}
	\node[right=5] at ($(q\jj)!.5!(q\kk)$) {$\mathbf {w}'_{\ii}$};
	}
\end{tikzpicture}
	\caption{Cutting a van Kampen diagram along corridors determined by an efficient coloured word on its boundary.}
	\label{fig:cut-corridors-bis}
\end{figure}

\begin{proof}
	Let $w$ and $w'$ be the underlying words of $\mathbf{w}$ and $\mathbf{w}'$, respectively. Assume that $ \mathbf w $ decomposes as $ \colwordproduct k $, and let $D$ be a van Kampen diagram for $w w'^{-1}$. Let $m = \ceil{\frac k3}$, and for each $i \in \range {m-1}$, let $C_i$ be the corridor starting at the first edge of $w_{3i+1}$. Let $\gamma_i$ be the side of $C_i$ whose endpoint $p_i$ on $w$ separates $w_{3i}$ and $w_{3i+1}$, and let $q_i$ be the other endpoint of $\gamma_i$, and let $ u_i $ be the word read along $ \gamma_i $ from $p_i$ to $q_i$. Since $ \mathbf w $ is efficient, the point $ q_i $ must belong to $ w' $.

	By \cref{estimate-sum-of-corridors-length}, the corridors $C_i$ do not cross each other, so we can construct a coloured diagram $\mathbf D$ by cutting along the $\gamma_i$, where the corridor sides are labelled with $\colword{u_i}{a_{3i+1}}$; see \cref{fig:cut-corridors-bis}. For each $i \in \range m$, the $i$th region of $\mathbf D$ is labelled by the word
	\[
		\colword{u_{i-1}}{a_{3i-2}}^{-1} \colword{w_{3i-2}}{a_{3i-2}} \colword{w_{3i-1}}{a_{3i-1}}\colword{w_{3i}}{a_{3i}} \colword{u_i}{a_{3i+1}} \mathbf w_i'^{-1},
	\]
	where we conventionally set $u_0 = u_m$ to be the empty word, and similarly for $w_i$ when $k < i \leq 3m$. Then this word is of the form $\mathbf{u}\mathbf{u'}^{-1}$, where $\mathbf u'=\mathbf w_i'$, so it satisfies the desired properties.

	It remains to establish the upper bound on the density of $ \pushdown[h] {\mathbf D} $; we do so by checking the hypotheses for \cref{sparse-pushdown}.
	By \cref{estimate-sum-of-corridors-length}, part \eqref{item:three-corridors}, we have $\sum_{i=1}^{m-1} \length{u_i} \leq 2 \length {ww'}$, so the density of $ \bolddiag$ is at most $5$.
	The only vertices of $ \bolddiag $ that may be polychromatic but not $ \partial $-polychromatic are the $ q_i $; this occurs when the two boundary edges incident to $ q_i $ have the same colour, which is different from the colour $a_{3i+1}$ of the interior edge. In this case, we connect $q_i$ to the $ \partial $-polychromatic vertex $ p_i $ via $ \gamma_i $. Since the paths $ \gamma_i $ are all disjoint and all vertices of $ \bolddiag $ have degree at most $3$, we can apply \cref{sparse-pushdown} and conclude.
\end{proof}

\section{Upper bounds}\label{sec:upper-bounds}

We now have all the necessary tools to establish upper bounds on the Dehn functions of Bestvina--Brady groups.

\begin{theorem}\label{upper-bounds}
	Let $\Gamma$ be a finite simplicial graph such that the associated flag complex $\flag\Gamma$ is simply connected, and let $\alpha \in \set{3,4}$. If $\Gamma$ does not have $\D \alpha$, then $\Dehn{\bbg\Gamma}(n) \preccurlyeq n^{\alpha-1}$.
\end{theorem}

To prove \cref{upper-bounds}, we need to show that, under the hypothesis above, every null-homotopic alternating word $w$ in $\bbg\Gamma$ has an almost-flat area bounded above by $C \cdot \length w^{\alpha-1}$ for some constant $C>0$. This is done by constructing an alternating diagram for $w$ and applying our multi-step strategy to reduce the task to estimating the almost-flat area of simpler null-homotopic alternating words.

We begin by establishing upper bounds for these relatively simple words. Then, we explain how these pieces can be viewed as the building blocks of a general coloured diagram by means of cutting along corridors as discussed in \cref{sec:cut-along-corridors}.

\subsection{Fundamental pieces}

In the previous section, we showed that given a coloured diagram $ \bolddiag $ for a null-homotopic coloured word $\boldword$, its $h$-pushdown $\pushdown[h]\bolddiag$ is an alternating diagram for $\pushdown[h]\boldword$. This allows us to estimate the almost-flat area of $\pushdown[h]\boldword$ in terms of the almost-flat area of the bounded regions of $\pushdown[h]\bolddiag$.

We start with the case where $\boldword$ belongs to one of the following three families of null-homotopic words, which we call \emph{fundamental pieces}: these are \emph{monochromatic words}, \emph{coloured bigons} (\cref{bigons}), and \emph{coloured commutators} (\cref{coloured-commutators}). These fundamental pieces are combinatorially easier to handle and can be used to estimate the almost-flat area of the $h$-pushdown of more complicated null-homotopic coloured words.

When dealing with the almost-flat area estimates of the fundamental pieces, we illustrate the power of working with diagrams: the filling of a coloured diagram, and thus of its pushdown, can be turned into a purely algebraic algorithm that tells us how to reduce the pushdown of the corresponding coloured word by applying relations at each step; see, for instance, \cref{bigon-quadratic-join}, where estimating areas is equivalent to bounding the number of relations involved in proving an identity between words.

\begin{definition}[Support of a word]
	Let $w$ be a word in $\raaggens$. The \emph{support of $w$}, denoted by $\supp w$, is the subgraph of $\Gamma$ induced by the vertices $v_s$ such that $s^{\pm 1}$ appears in $w$.
\end{definition}

\begin{definition}[Palette of a word]
	The \emph{palette of $w$}, denoted by $\Pal w$, is the subgraph of $\Gamma$ consisting of the intersection of the stars of all vertices in $\supp w$. That is, the vertex set of $\Pal w$ corresponds to the set of colours $a$ such that $\colword w a$ is a well-defined coloured word.
\end{definition}

\begin{definition}[Coloured bigon]\label{bigons}
	A \emph{coloured bigon} is a coloured word of the form $\colword w a \colword{w}{b}^{-1}$, where $w$ is a word in $\raaggens$, and $v_a$ and $v_b$ are vertices of $\Pal w$.
\end{definition}

\begin{definition}[Coloured commutator]\label{coloured-commutators}
	A \emph{coloured commutator} is the commutator $[\colword{w_1}{a}, \colword{w_2}{b}]$ of two monochromatic words $\colword{w_1}{a}$ and $\colword{w_2}{b}$ such that $\set{v_a} \cup \supp{w_2} \subset \Pal{w_1}$ and $\set{v_b} \cup \supp{w_1} \subset \Pal{w_2}$. 
\end{definition}

By definition, the $h$-pushdown of coloured bigons and coloured commutators are null-homotopic alternating words. Observe that the $h$-pushdown of a coloured commutator is not in general the commutator of the $h$-pushdowns of the corresponding monochromatic coloured words.

To establish upper bounds for the areas of null-homotopic monochromatic words, coloured bigons, and coloured commutators, we first need some preliminary results. We begin by recalling a result from \cite{Dison}, stated in our notation.

\begin{lemma}\textnormal{(\cite[Lemma 4.7]{Dison})}\label{fellow-quadratic}
	There exists a constant $K >0$ such that the following holds. Let $h \in \Z$ and $\{v_a,v_b\} \in \ee \Gamma$. The alternating  word $\transitionword a{h}^{-1} \transitionword b{h} (ab^{-1})^{-h}$ is null-homotopic, and its almost-flat area satisfies
	\[
		\flatarea{\transitionword a{h}^{-1} \transitionword b{h} (ab^{-1})^{-h}} \leq K \abs h ^2.
	\]
\end{lemma}

Throughout the rest of the section, we use the following notation.
\begin{notation}
	Given two alternating words $u$ and $v$ in $\raaggens$ that represent the same element in $\bbg \Gamma$, we define the \emph{almost-flat area of the identity $u \groupid{\bbg \Gamma} v$} to be the almost-flat area of the null-homotopic word $uv^{-1}$. If $A$ is the almost-flat area of the identity $u \groupid{\bbg \Gamma} v$, we say that \emph{$u$ can be rewritten as $v$ with almost-flat area $A$}.
\end{notation}

\begin{proposition}\label{area-monochromatic}
	Let $h \in \Z$, $a \in \raaggens$, and $\colword wa$ be a null-homotopic monochromatic word. Then
	\[
		\flatarea{\pushdown[h]{\colword wa}} \leq 3 \length w^2.
	\]
\end{proposition}
\begin{proof}
	Since the monochromatic word $\colword wa$ is null-homotopic, we have $\height(w) = 0$. It follows from the definition of the pushdown (\cref{def:push-down}) that
	\[
		\pushdown[h]{\colword wa} \freeid \transitionword ah \balance wa \transitionword ah ^{-1}.
	\]
	The word $\balance wa$ is null-homotopic in $\bbg {\Star {v_a}}$. The latter is isomorphic to the right-angled Artin group $\raag {\Link {v_a}}$ via the homomorphism defined by $sa^{-1} \mapsto s$ for all $v_s \in \Link{v_a}$, which sends $\balance wa$ to $w$.
	Since the area of $w$ in $\raag {\Link {v_a}}$ is at most $\abs w^2$, it follows from \cref{prop:area-bbg-vs-almost-flat} that the word $\balance wa$ admits an almost-flat van Kampen diagram of quadratic area.

	To get the precise bound on the almost-flat area, note that a relation $[b,c]=1$ in $\raag {\Link {v_a}}$ is mapped to the null-homotopic alternating word $b a^{-1} cb^{-1} ac^{-1}$ in $\bbg {\Star {v_a}}$, which admits an almost-flat diagram of area $3$. Therefore, a van Kampen diagram for $w$ can be transformed into an almost-flat van Kampen diagram for $\balance wa$ of area at most $3 \length w^2$, as done in the proof of \cref{prop:area-bbg-vs-almost-flat}.
\end{proof}

\subsubsection{Coloured bigons} We now proceed to establish upper bounds for the almost-flat area of the $h$-pushdown of coloured bigons.

\begin{lemma}\label{bigon-quadratic-combinatorial-path}
	There exists a constant $C > 0$ with the following property. Let $\mathbf w = \colword wa \colword wb ^{-1}$ be a coloured bigon.
	Assume that $v_a$ and $v_b$ lie in the same connected component of $\Pal w$.
	Then, for all $h \in \Z$, we have
	\[
		\flatarea{\pushdown[h]{\mathbf{w}}} \leq C \cdot ( \length{w} + \abs{h} )^2.
	\]
\end{lemma}

\begin{proof}
	For any pair of adjacent vertices $v_{p}$ and $v_{q}$ in $\Pal w$, we have that the alternating word
	\begin{align*}
		\transitionword ph^{-1} \pushdown[h]{\colword w{p} \colword w{q}^{-1}}  \transitionword ph \freeid \balance wp \transitionword p{h + \height(w)} ^{-1} \transitionword q{h + \height(w)} \balance wq ^{-1} \transitionword qh ^{-1} \transitionword ph
	\end{align*}
	can be rewritten using \cref{fellow-quadratic} as
	\[
		\balance wp (pq^{-1})^{h + \height(w)} \balance wq ^{-1} (qp^{-1})^h.
	\]
	This rewriting has almost-flat area at most $2 K(\length w + \abs{h})^2$. This last word is a null-homotopic alternating word in the generating set of $\bbg {\Star{v_p}}$, which is isomorphic to the right-angled Artin group $\raag {\Link{v_p}}$. In particular, there is a constant $K'>0$ such that it has almost-flat area bounded above by $K' \cdot (\length{w} + \abs h)^2$. Putting everything together yields
	\[
		\flatarea{\pushdown[h]{\colword wp} \pushdown[h]{\colword wq}^{-1}} \leq (2K+K) \cdot (\length w + \abs h)^2.
	\]

	Since a shortest path in $\Pal w$ connecting $v_a$ to $v_b$ has length at most $\cardinality{\vv\Gamma}$, by applying the above argument at most $\cardinality{\vv\Gamma}$ times, we obtain
	\begin{align*}
		\flatarea{\pushdown[h]{\mathbf w}} & = \flatarea{\pushdown[h]{\colword wa} \pushdown[h]{\colword wb}^{-1}}             \\
		                                   & \leq \cardinality{\vv\Gamma} \cdot (2K+K') \cdot (\length w + \abs h)^2. \qedhere
	\end{align*}
\end{proof}

In particular, if $\Pal w$ is connected, then by \cref{bigon-quadratic-combinatorial-path}, the almost-flat area of the pushdown of the coloured bigon is at most quadratic in its length.

\begin{lemma}\label{bigon-quadratic-join}
	There exists a constant $C > 0$ with the following property. Let $\mathbf w = \colword wa \colword wb ^{-1}$ be a coloured bigon. Assume that $\supp w$ is reducible. Then, for all $h \in \Z$, we have
	\[
		\flatarea{\pushdown[h]{\mathbf{w}}} \leq C \cdot ( \length{w} + \abs{h} )^2.
	\]
\end{lemma}

\begin{proof}
	Let $\Lambda \coloneq \supp w$, and let $\Lambda = \Lambda' * \Lambda''$ be a join decomposition. Due to the decomposition $\raag\Lambda=\raag{\Lambda'} \times \raag{\Lambda''}$, the element $ g \in \raag\Gamma $ represented by $w$ can also be represented by a product of two words $w'$ and $w''$, whose supports lie respectively in $\Lambda'$ and $\Lambda''$, such that $\length{w'} + \length{w''} \leq \length{w}$.

	This allows us to build a coloured diagram $\bolddiag$ for $\colword wa \colword wb^{-1}$ with four bounded regions. Two of these regions are labelled by the null-homotopic monochromatic words $\boldword_1 = \colword wa \colword{w''}a^{-1} \colword{w'}a^{-1}$ and $\boldword_2 = \colword{w'}b  \colword{w''}b \colword wb^{-1} $; the other two are labelled by the coloured bigons $\boldword_3 = \colword{w'}a \colword{w'}b^{-1}$ and $\boldword_4 = \colword{w''}a \colword{w''}b^{-1}$; see \cref{fig:peanut}.

	\begin{figure}
		\centering
		\begin{tikzpicture}
    \definecolor{happya}{RGB}{230,96,0}
	\definecolor{happyb}{RGB}{93,58,155}
	
	\tikzset{|->-|/.style={decoration={markings,
			mark=at position 0 with {\arrow[ultra thin]{|}},
			mark=at position #1 with {\arrow{latex}},
			mark=at position 1 with {\arrow[ultra thin]{|}},
		},postaction={decorate}}}
	\tikzset{->-/.style={decoration={markings,mark=at position #1+.01 with {\arrow{latex}}},postaction={decorate}}}
	\tikzset{-<-/.style={decoration={markings,mark=at position #1-.01 with {\arrowreversed{latex}}},postaction={decorate}}}

	\coordinate (p0) at (0,0);
	\coordinate (p2) at (0,-6);
	\coordinate (q0) at (p0);
	\coordinate (q2) at (p2);
	\coordinate (m) at ($(p0)!0.5!(q2)$);

	\draw[thick,->-=.5, happya](p0) to[out=-180, in=+180]
	coordinate[pos=1/2](p1)
	(p2);
	\draw[thick,->-=.5, happyb](p0) to[out=0, in=0]
	coordinate[pos=1/2](q1)
	(p2);
	\draw[thick,draw opacity=0,->-=.5](p0) to[out=-135, in=135]
	coordinate[pos=1/2](m11)
	(m);
	\draw[thick,draw opacity=0,->-=.5](p0) to[out=-45, in=45]
	coordinate[pos=1/2](m12)
	(m);
	\draw[thick,draw opacity=0,->-=.5](m) to[out=-135, in=135]
	coordinate[pos=1/2](m21)
	(p2);
	\draw[thick,draw opacity=0,->-=.5](m) to[out=-45, in=45]
	coordinate[pos=1/2](m22)
	(p2); 
	
	\node[left, happya] at (p1) {$\colword wa$};
	\node[right, happyb] at (q1) {$\colword wb $};
	\node[left, fill=white, fill opacity=.6, text opacity=1, text=happya] at (m11) {$\colword {w'}a$};
	\draw[thick,->-=.5, happya](p0) to[out=-135, in=135](m);
	\node[right, fill=white, fill opacity=.6, text opacity=1, text=happyb] at (m12) {$\colword {w'}b$};
	\draw[thick,->-=.5, happyb](p0) to[out=-45, in=45](m);
	\node[left, fill=white, fill opacity=.6, text opacity=1, text=happya] at (m21) {$\colword {w''}a$};
	\draw[thick,->-=.5, happya](m) to[out=-135, in=135](p2);
	\node[right, fill=white, fill opacity=.6, text opacity=1, text=happyb] at (m22) {$\colword {w''}b$};
	\draw[thick,->-=.5, happyb](m) to[out=-45, in=45](p2);
	
	\draw[fill=black] (p0) circle (.04);
	\draw[fill=black] (p2) circle (.04);
	\draw[fill=black] (m) circle (.04);
\end{tikzpicture}
		\caption{Filling of a coloured diagram for the coloured bigon when $\Lambda$ decomposes as a join $\Lambda' * \Lambda''$.}
		\label{fig:peanut}
	\end{figure}
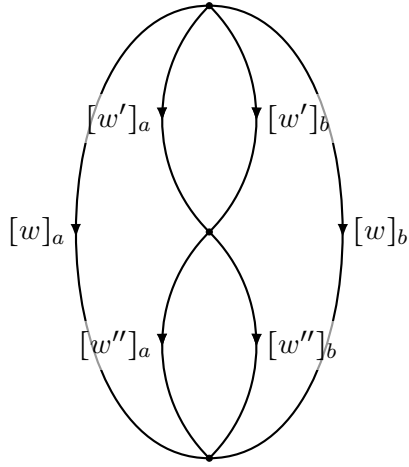

	Note that $ \pushdown[h]\bolddiag $ is an alternating diagram for $ \pushdown[h]\boldword $.
	We now proceed to estimate the almost-flat area of the four bounded regions in $\pushdown[h] \bolddiag$. These regions are labelled by the alternating words $u_i = \pushdown[h_i]{\boldword_i}$ with $i \in \set{1,2,3,4}$, where $h_1 = h_2 = h_3 = h$ and $h_4 = h + \height(w')$.
	\begin{enumerate}[label=(\alph*)]
		\item For the pushdown of the null-homotopic monochromatic words $\boldword_1 = \colword wa \colword{w''}a^{-1} \colword{w'}a^{-1}$ and $\boldword_2= \colword wb \colword{w''}b^{-1} \colword{w'}b^{-1}$, we apply \cref{area-monochromatic} to obtain $\flatarea{u_i} \leq 12 \cdot \length w ^2$ for $i \in \{1,2\}$. \label{monochromatic-pieces}
		\item For the pushdown of the coloured bigon $\boldword_3 = \colword{w'}a \colword{w'}b^{-1}$, we can apply \cref{bigon-quadratic-combinatorial-path}, since for every $v_c \in \Lambda'' \subgraph \Pal{w'}$, we have $v_a,v_b \in \Star{v_c}$, and thus $v_a$ and $v_b$ lie in the same connected component of $\Pal{w'}$. The same argument holds for $\boldword_4 = \colword{w''}a \colword{w''}b^{-1}$. Therefore, we obtain $\flatarea{u_i} \leq K \cdot (\length{w} + \abs {h_i})^2$ for $i \in \{3,4\}$ and a sufficiently large constant $K>0$.  \label{bigon-pieces}
	\end{enumerate}
	Summing up, we find a constant $K'>0 $ such that $\flatarea{u_i} \leq K' \cdot \length{u_i}^2$ for all $i \in \range 4$. For this, we are using the lower bound on $\length{u_i}$ given by \cref{length-of-pushdown}.

	There are now two different ways to conclude. The first is to apply \cref{sparse-pushdown} with $M = 4$, obtaining a uniform bound $C_M$ on the density of $\pushdown[h]\bolddiag$, and then apply \cref{lem:area-of-protodiagram} to obtain
	\[
		\flatarea{\boldword} \leq K' \cdot C_M^2 \cdot \length{\pushdown[h] \boldword }^2  \leq C \cdot(\length{w} + \abs h)^2
	\]
	for large enough $C>0$.

	Alternatively, we can translate the coloured diagram $\bolddiag$ into an algorithm to reduce the $h$-pushdown of $\colword{w}a \colword{w}b^{-1}$. Reducing this alternating word is equivalent to estimating the almost-flat area of the identity $\pushdown[h]{\colword wa} \groupid{\bbg \Gamma} \pushdown[h]{\colword wb}$.

	From \ref{monochromatic-pieces}, we get that the identity
	\begin{equation}\label{bigon-mono-color-a}
		\pushdown[h_1]{ \colword{w}a } \groupid{\bbg \Gamma} \pushdown[h_1]{ \colword{w'}a \colword{w''}a }
	\end{equation}
	has almost-flat area at most $12 \cdot \length{w}^2$. Similarly, the almost-flat area of the identity
	\begin{equation}\label{bigon-mono-color-b}
		\pushdown[h_2]{ \colword{w}b } \groupid{\bbg \Gamma} \pushdown[h_2]{ \colword{w'}b \colword{w''}b }
	\end{equation}
	is at most $12 \cdot \length{w}^2$. On the other hand, it follows from \ref{bigon-pieces} that the identity
	\begin{equation}\label{bigon-bigon-first-factor}
		\pushdown[h_3]{\colword {w'}a} \groupid{\bbg \Gamma} \pushdown[h_3]{\colword {w'}b}
	\end{equation}
	holds with almost-flat area at most $K \cdot (\length{w} + \abs{h_3})^2$. Similarly, the almost-flat area of the identity 
	\begin{equation}\label{bigon-bigon-second-factor}
		\pushdown[h_4]{ \colword {w''}a } \groupid{\bbg \Gamma} \pushdown[h_4]{ \colword{w''}b }
	\end{equation}
	is at most $K \cdot (\length{w} + \abs{h_4})^2$.

	To assemble these identities together, we use the following property of the pushdown
	\begin{equation}\label{pushdown-split}
		\pushdown[\ell]{\mathbf{u} \mathbf{u}'} \freeid \pushdown[\ell]{\mathbf{u}} \pushdown[h + \height(u)]{\mathbf{u}'},
	\end{equation}
	for any coloured words $\mathbf{u}$ and $\mathbf{u}'$; see \cref{pushdown-properties}, part \eqref{item: pushdown of ww'}. Recall $h_1 = h_2 =h_3 = h$ and $h_4 = h + \height(w')$. We obtain that the following identities
	\begin{align*}
		\pushdown[h]{ \colword wa } & \groupid{\bbg \Gamma} \pushdown[h_1]{\colword{w'}a \colword{w''}a}                 &  & \text{\small{by \eqref{bigon-mono-color-a}}}                                             \\
		                            & \freeid \pushdown[h_1]{\colword{w'}a} \pushdown[h_4]{\colword{w''}a}               &  & \text{\small{by \eqref{pushdown-split}}}                                                 \\
		                            & \groupid{\bbg \Gamma} \pushdown[h_2]{\colword{w'}b} \pushdown[h_4]{\colword{w''}b} &  & \text{\small{by \eqref{bigon-bigon-first-factor} and \eqref{bigon-bigon-second-factor}}} \\
		                            & \freeid \pushdown[h_2]{ \colword{w'}b \colword{w''}b }                             &  & \text{\small{by \eqref{pushdown-split}}}                                                 \\
		                            & \groupid{\bbg \Gamma} \pushdown[h]{ \colword wb}                                   &  & \text{\small{by \eqref{bigon-mono-color-b}}}
	\end{align*}
	hold with almost-flat area at most
	\begin{align*}
		\,      & 24 \cdot \length w ^2 + K \cdot \left( (\length{w} + \abs{h})^2  + (\length{w} + \abs{h + \height(w')})^2 \right) \\
		\leq \, & 24 \cdot \length w ^2 + K \cdot \left( (\length{w} + \abs{h})^2  + (2\length{w} + \abs{h})^2 \right)              \\
		\leq \, & 29 \cdot (K+1) \cdot (\length{w} + \abs{h})^2.
	\end{align*}

	This second procedure yields in general a better constant, since we are estimating the areas of the regions of $ \pushdown\bolddiag $ more precisely without relying on the loose estimate of the density given by \cref{sparse-pushdown}.
\end{proof}

We saw that under certain conditions, the pushdown of a coloured bigon has quadratic area. We now show that in general we get a cubic upper bound.

\begin{lemma}\label{bigon-cubic}
	There exists a constant $C > 0$ with the following property. Let $\boldword = \colword wa \colword wb ^{-1}$ be a coloured bigon. Then, for all $h \in \Z$, we have
	\[
		\flatarea{ \pushdown[h]{\mathbf{w}}} \leq C \cdot \left( \length{w} + \abs{h} \right)^3.
	\]
\end{lemma}

\begin{proof}
	Let $k = \length w$, and write $w = s_1^{\epsilon_1} \cdots s_k^{\epsilon_k}$, where ${s_i} \in \raaggens$ and $\epsilon_i \in \{\pm 1\}$. By definition, we have
	\[
		\colword wa = \colword{s_1}a^{\epsilon_1} \cdots \colword{s_k}a^{\epsilon_k}.
	\]
	A coloured diagram for $\boldword = \colword wa \colword wb ^{-1}$ can be constructed using $k$ coloured bigons of the form $\colword{s_i}a^{\epsilon_i} \colword{s_i}b^{-\epsilon_i}$; see \cref{fig:necklace}.

	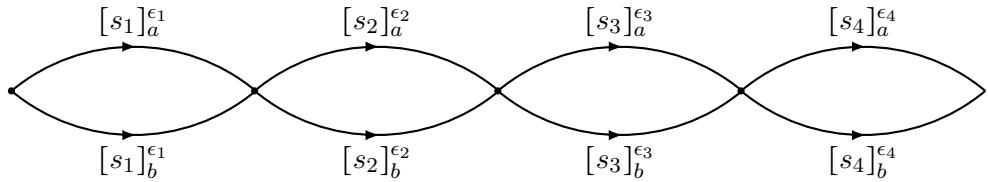
\begin{figure}
		\centering
		\begin{tikzpicture}
	\definecolor{happya}{RGB}{230,96,0}
	\definecolor{happyb}{RGB}{93,58,155}
	\tikzset{|->-|/.style={decoration={markings,
			mark=at position 0 with {\arrow[ultra thin]{|}},
			mark=at position #1 with {\arrow{latex}},
			mark=at position 1 with {\arrow[ultra thin]{|}},
		},postaction={decorate}}}
	\tikzset{nodes={fill=white,fill opacity=0,text opacity=1}}
	\tikzset{->-/.style={decoration={markings,mark=at position #1+.01 with {\arrow{latex}}},postaction={decorate}}}
	\tikzset{-<-/.style={decoration={markings,mark=at position #1-.01 with {\arrowreversed{latex}}},postaction={decorate}}}
	
	\def\p{4};
	
	\foreach \i in {1,...,\p} { 
	\draw[thick,->-=.5, happya] (3.2*\i,0) arc (130:50:2.5) node[midway,above]{$\colword{s_{\i}}{a}^{\epsilon_\i}$};
	\draw[thick,->-=.5, happyb] (3.2*\i,0) arc (-130:-50:2.5) node[midway,below]{$\colword{s_\i}{b}^{\epsilon_\i}$};
	\draw[fill=black] (3.2*\i,0) circle (.04);
    \draw[fill=black] (3.2*5,0) circle (.04);
	}
 
\end{tikzpicture}
		\caption{Filling of a coloured diagram for the coloured bigon $\colword wa \colword wb ^{-1}$ with $w =s_1^{\epsilon_1} \cdots s_k^{\epsilon_k}$. The filling consists of $k$ coloured bigons.}
		\label{fig:necklace}
	\end{figure}

	For every $i \in \range {k}$ and $\ell \in \Z$, the alternating word $\pushdown[\ell]{\colword{s_i}a^{\epsilon_i} \colword{s_i}b^{-\epsilon_i}}$ is null-homotopic. Moreover, since $\Pal {s_i^{\epsilon_i} } = \Star{s_i}$ is connected, \cref{bigon-quadratic-combinatorial-path} implies that
	\[
		\flatarea{\pushdown[\ell]{\colword{s_i}a^{\epsilon_i} \colword{s_i}b^{-\epsilon_i}}} \leq K \cdot \left( 1 + \abs{\ell} \right)^2.
	\]
	Therefore, for $h_0 = h$ and $h_i = h + \height(s_1 \cdots s_i)$, putting these $k$ pieces together using the properties of the pushdown given in \cref{pushdown-properties}, we obtain the identities
	\[
		\pushdown[h]{\colword wa} \freeid \prod_{i=1}^k \pushdown[h_{i-1}]{\colword{s_i}a^{\epsilon_i}} \groupid{\bbg \Gamma} \prod_{i=1}^k \pushdown[h_{i-1}]{\colword{s_i}b^{\epsilon_i}} \freeid \pushdown[h]{\colword wb}.
	\]
	Thus, the almost-flat area of $\pushdown[h]{\boldword}$ satisfies 
	\begin{align*}
		\flatarea{\pushdown[h]{\boldword}} & \leq K \cdot \sum_{i = 1}^k \left( 1 + \abs{h_{i-1}} \right)^2                             \\
		                                   & = K \cdot \sum_{i = 1}^k 1 + 2\abs{h_{i-1}} + \abs{h_{i-1}}^2                              \\
		                                   & \leq K \cdot \length{w} \left(1 + 2(\length{w} + \abs h ) + (\length{w} + \abs h)^2\right) \\
		                                   & \leq 4K \cdot (\length{w} + \abs h)^3. \qedhere
	\end{align*}
\end{proof}

We now have all the tools to prove that coloured bigons satisfy \cref{upper-bounds}.

\begin{proposition}\label{area-coloured-bigon}
	There exists a constant $C > 0$ with the following property. Assume that $\alpha \in \set{3,4}$ and that $\Gamma$ does not have $\D\alpha$. Let $\boldword = \colword wa \colword wb ^{-1}$ be a coloured bigon. Then, for all $ h \in \Z $, we have
	\[
		\flatarea{ \pushdown[h]\boldword} \leq C \cdot \length{\pushdown[h]\boldword}^{\alpha - 1}.
	\]
\end{proposition}

\begin{proof}
	Observe that the case $\alpha = 4$ follows directly from \cref{bigon-cubic}, so we assume $\alpha = 3$. Since $\Gamma$ does not have property \D3, every maximal join is not essentially $2$-reducible.

	We may assume that $v_a$ and $v_b$ are not adjacent and that $\Pal w$ is disconnected; otherwise, we could conclude by \cref{bigon-quadratic-combinatorial-path}. In particular, the vertices $v_a$ and $v_b$ do not belong to $\supp w$. If $\supp w$ were reducible, we would also be done by \cref{bigon-quadratic-join}, so we assume that $\supp w$ is irreducible.

	Since the subgraph $\supp w \ast (\Pal w \setminus \supp w)$ is reducible, it is contained in a maximal reducible subgraph $\Lambda \subgraph \Gamma$. Consider the decomposition $\Lambda = \Lambda_1 \ast \Lambda_2 \ast \dots \ast \Lambda_k$, where $k \geq 2$ and $\Lambda_i$ is irreducible for all $i \in \range k$.

	Since $\supp w \subgraph \Lambda$ is irreducible, it is contained in some $\Lambda_i$, say $\Lambda_1$.
	Note that $\Pal w = (\Pal w \cap \Lambda_1) \ast \Lambda_2 \ast \dots \ast \Lambda_k$. Since $\Pal w$ is disconnected, it follows that $k=2$ and $\Pal w \cap \Lambda_1 = \emptyset$. Thus, the subgraph $\Lambda_2$ has at least two vertices $v_a$ and $v_b$. Since $ \Lambda $ is not essentially $2$-reducible, we have $ \Lambda_1 = \set{v_c}$. In this case, we obtain $v_c \in \Pal w \cap \Lambda_1$, which is a contradiction.
\end{proof}

\subsubsection{Coloured commutators}

We now give an upper bound on the almost-flat area of the $h$-pushdown of a coloured commutator.

\begin{lemma}\label{commutator-common-colour}
	There exists a constant $C > 0$ such that the following holds. Assume that $\alpha \in \set{3,4}$ and that $\Gamma$ does not have $\D\alpha$. Let $\mathbf w = [\colword {w_1} a, \colword {w_2} b]$ be a coloured commutator. If $\Pal{w_1} \cap \Pal{w_2} \neq \emptyset$, then for all $h \in \Gamma$, we have
	\[
		\flatarea{\pushdown[h]{\boldword}} \leq C \cdot (\length{w_1} + \length{w_2} + \abs h)^{\alpha - 1}.
	\]
\end{lemma}
\begin{proof}
	Pick $v_c \in \Pal{w_1} \cap \Pal{w_2}$. A coloured diagram $\bolddiag$ for $\boldword$ consists of five bounded regions labeled by $\boldword_1, \boldword_2, \boldword_3, \boldword_4$, and $\boldword_5$. The first four coloured words are coloured bigons of the form $\boldword_1 = \colword{w_1} a \colword {w_1} c ^{-1}$, $\boldword_2 = \colword{w_2}b \colword{w_2}c^{-1}$, $\boldword_3 = \colword{w_1}a^{-1} \colword{w_1}c $, and $\boldword_4 = \colword{w_2}b^{-1} \colword{w_2}c$, respectively, while the last one is a null-homotopic monochromatic word $\boldword_5 = [\colword{w_1}c, \colword{w_2}c]$; see \cref{fig:common-colour}.

	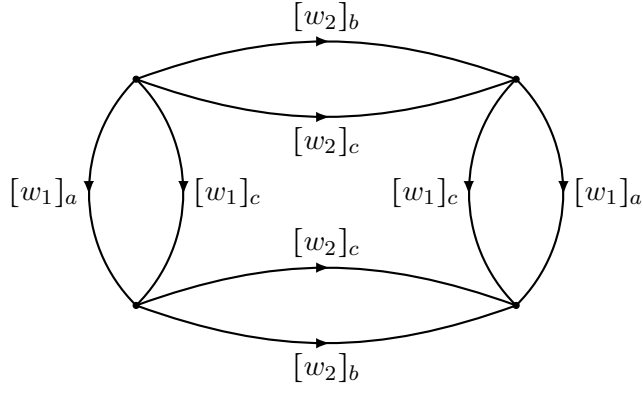
\begin{figure}
		\centering
		\begin{tikzpicture}
	\definecolor{happya}{RGB}{136,34,85}
	\definecolor{happyb}{RGB}{17,119,51}
	\tikzset{|->-|/.style={decoration={markings,
			mark=at position 0 with {\arrow[ultra thin]{|}},
			mark=at position #1 with {\arrow{latex}},
			mark=at position 1 with {\arrow[ultra thin]{|}},
		},postaction={decorate}}}
	\tikzset{nodes={fill=white,fill opacity=0,text opacity=1}}
	\tikzset{->-/.style={decoration={markings,mark=at position #1+.01 with {\arrow{latex}}},postaction={decorate}}}
	\tikzset{-<-/.style={decoration={markings,mark=at position #1-.01 with {\arrowreversed{latex}}},postaction={decorate}}}

	\coordinate (p0) at (0,0);
	\coordinate (p2) at (0,-3);
	\coordinate (ml) at ($(p0)!0.5!(q2)$);

	\coordinate (r0) at (5,0);
	\coordinate (r2) at (5,-3);
	\coordinate (mr) at ($(r0)!0.5!(s2)$);

	\draw[thick,->-=.5, happya](p0) to[out=-135, in=135]
	coordinate[pos=1/2](m1)
	(p2);
	\draw[thick,->-=.5](p0) to[out=-45, in=45]
	coordinate[pos=1/2](m2)
	(p2);

	\draw[thick,->-=.5](r0) to[out=-135, in=135]
	coordinate[pos=1/2](m3)
	(r2);
	\draw[thick,->-=.5, happya](r0) to[out=-45, in=45]
	coordinate[pos=1/2](m4)
	(r2);

	\draw[thick,->-=.5, happyb](p0) to[out=20, in=160]
	coordinate[pos=1/2](top1)
	(r0);
	\draw[thick,->-=.5](p0) to[out=-20, in=-160]
	coordinate[pos=1/2](top2)
	(r0);

	\draw[thick,->-=.5](p2) to[out=20, in=160]
	coordinate[pos=1/2](bottom1)
	(r2);
	\draw[thick,->-=.5, happyb](p2) to[out=-20, in=-160]
	coordinate[pos=1/2](bottom2)
	(r2);

	\node[left, happya] at (m1) {$\colword {w_1}a$};
	\node[right] at (m2) {$\colword {w_1}c$};

	\node[left] at (m3) {$\colword {w_1}c$};
	\node[right, happya] at (m4) {$\colword {w_1}a$};

	\node[above] at (bottom1) {$\colword {w_2}c$};
	\node[below, happyb] at (bottom2) {$\colword {w_2}b$};
	\node[above, happyb] at (top1) {$\colword {w_2}b$};
	\node[below] at (top2) {$\colword{w_2}c$};
	
	\draw[fill=black] (p0) circle (.04);
	\draw[fill=black] (p2) circle (.04);

	\draw[fill=black] (r0) circle (.04);
	\draw[fill=black] (r2) circle (.04);

\end{tikzpicture}
		\caption{Filling of a coloured diagram for the coloured commutator $\boldword = [\colword{w_1}a, \colword{w_2}b]$ in the case where $v_c \in \Pal{w_1} \cap \Pal{w_2}$.}
		\label{fig:common-colour}
	\end{figure}

	We now estimate the almost-flat area of the bounded regions in $\pushdown[h]{\bolddiag}$, which are labeled by the alternating word $u_i = \pushdown[h_i]{\boldword_i}$ for $i \in \range 5$, where $h_1 = h_5 = h, h_2 = h + \height(w_1), h_3 = h + \height(w_1 w_2)$, and $h_4 = h + \height(w_2)$.

	On the one hand, for all the coloured bigons, it follows from \cref{area-coloured-bigon} that there exists $C_1>0$ such that $\flatarea{u_i} \leq C_1 \cdot \length{u_i}^{\alpha - 1}$ for all $i \in \range 4$. On the other hand, it follows from \cref{area-monochromatic} that for the null-homotopic monochromatic word $\boldword_5$, we have $\flatarea{u_5} \leq 12 \cdot (\length{w_1} + \length{w_2})^{2}$.

	All together, we find a constant $C_3>0$ such that for all $i \in \range 5$, we have
	\[
		\flatarea{u_i} \leq C_3 \cdot \length{u_i}^{\alpha - 1}.
	\]

	We conclude by first applying \cref{sparse-pushdown} with $M=4$, obtaining a uniform bound $C_M$ on the density of $\pushdown[h]{\bolddiag}$, and then by \cref{lem:area-of-protodiagram}, we obtain
	\[
		\flatarea{\pushdown[h]{\boldword}} \leq C_3 \cdot C_M^2 \cdot \length{\pushdown[h]{\boldword}}^{\alpha-1} \leq C \cdot (\length{w_1} + \length{w_2} + \abs h)^{\alpha-1}
	\]
	for sufficiently large $C>0$.
\end{proof}

\begin{lemma}\label{commutator-3joins}
	There exists a constant $C > 0$ such that the following holds. Assume that $\alpha \in \set{3,4}$ and that $\Gamma$ does not have $\D\alpha$. Let $\mathbf w = [\colword {w_1} a, \colword {w_2} b]$ be a coloured commutator. If either $\supp{w_1}$ or $\supp{w_2}$ is reducible, then for all $h \in \Z$, we have
	\[
		\flatarea{ \pushdown[h]{\boldword}} \leq C \cdot \left( \length{w_1} + \length{w_2} + \abs{h} \right)^{\alpha-1}.
	\]
\end{lemma}
\begin{proof}
	Suppose without loss of generality that $\supp{w_1}$ decomposes as a join $\Lambda' * \Lambda''$. 
	Set $\Lambda = \supp{w_1}$. In this case, due to the decomposition $\raag\Lambda = \raag{\Lambda'} \times \raag{\Lambda''}$, the element $ g \in \raag\Gamma $ represented by $w$ can also be represented by a product of two words $w'$ and $w''$ whose supports lie respectively in $\Lambda'$ and $\Lambda''$, and such that $\length{w'} + \length{w''} \leq \length{w}$. Thus, a coloured diagram $\bolddiag$ for the alternating word $\pushdown[h]{ \boldword }$ can be constructed with four bounded regions whose labels are coloured words $\boldword_1, \boldword_2, \boldword_3$, and $\boldword_4$. Two of these are null-homotopic monochromatic words $\boldword_1 = \colword{w_1}a \colword{w''}a^{-1} \colword{w'}a^{-1}$ and $\boldword_3 = \colword{w'}a \colword{w''}a \colword{w_1}a^{-1}$, and the others are coloured commutators $\boldword_2 = [\colword{w''}a, \colword{w_2}b]$ and $\boldword_4 = [\colword{w'}a, \colword{w_2}b]$; see \cref{fig:square-peanut}.

	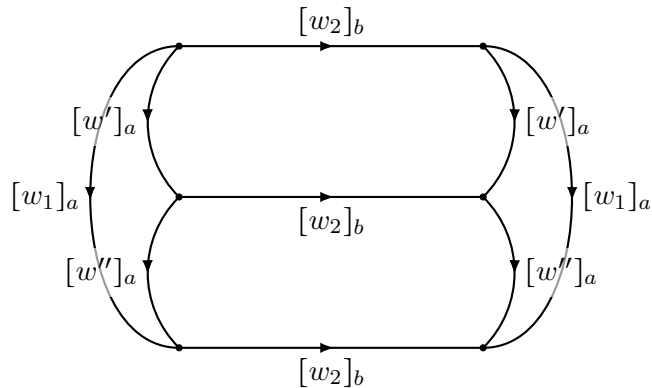
\begin{figure}
		\centering
		\begin{tikzpicture}
	\definecolor{happya}{RGB}{136,34,85}
	\definecolor{happyb}{RGB}{17,119,51}
	\tikzset{|->-|/.style={decoration={markings,
			mark=at position 0 with {\arrow[ultra thin]{|}},
			mark=at position #1 with {\arrow{latex}},
			mark=at position 1 with {\arrow[ultra thin]{|}},
		},postaction={decorate}}}
	\tikzset{->-/.style={decoration={markings,mark=at position #1+.01 with {\arrow{latex}}},postaction={decorate}}}
	\tikzset{-<-/.style={decoration={markings,mark=at position #1-.01 with {\arrowreversed{latex}}},postaction={decorate}}}

	\coordinate (p0) at (0,0);
	\coordinate (p2) at (0,-4);
	\coordinate (q0) at (p0);
	\coordinate (q2) at (p2);
	\coordinate (ml) at ($(p0)!0.5!(q2)$);
	
	\coordinate (r0) at (4,0);
	\coordinate (r2) at (4,-4);
	\coordinate (s0) at (r0);
	\coordinate (s2) at (r2);
	\coordinate (mr) at ($(r0)!0.5!(s2)$);

	\draw[thick,->-=.5, happya](p0) to[out=-180, in=+180]
	coordinate[pos=1/2](p1)
	(p2);
	\draw[thick,draw opacity=0,->-=.5, happya](p0) to[out=-135, in=135]
	coordinate[pos=1/2](ml11)
	(ml);
	\draw[thick,draw opacity=0,->-=.5, happya](ml) to[out=-135, in=135]
	coordinate[pos=1/2](ml21)
	(p2);
	
	\draw[thick,->-=.5, happya](r0) to[out=0, in=0]
	coordinate[pos=1/2](s1)
	(r2);
	\draw[thick,draw opacity=0,->-=.5, happya](s0) to[out=-45, in=45]
	coordinate[pos=1/2](mr12)
	(mr);
	\draw[thick,draw opacity=0,->-=.5, happya](mr) to[out=-45, in=45]
	coordinate[pos=1/2](mr22)
	(s2); 
	
	\draw[thick,->-=.5, happyb](p0) to%
	coordinate[pos=1/2](top1)
	(r0);
	
	\draw[thick,->-=.5, happyb](ml) to%
	coordinate[pos=1/2](mid2)
	(mr);
	
	\draw[thick,->-=.5, happyb](p2) to 
	coordinate[pos=1/2](bottom)
	(s2);
	
	\node[left, happya] at (p1) {$\colword {w_1}a$};
	\node[left, fill=white,fill opacity=.6,text opacity=1, text=happya] at (ml11) {$\colword {w'}a$};
	\draw[thick,->-=.5, happya](p0) to[out=-135, in=135](ml);
	\node[left, fill=white,fill opacity=.6,text opacity=1, text=happya] at (ml21) {$\colword {w''}a$};
	\draw[thick,->-=.5, happya](ml) to[out=-135, in=135](p2);
	
	\node[right, happya] at (s1) {$\colword {w_1}a$};
	\node[right, fill=white,fill opacity=.6,text opacity=1, text=happya] at (mr12) {$\colword {w'}a$};
	\draw[thick,->-=.5, happya](s0) to[out=-45, in=45](mr);
	\node[right, fill=white,fill opacity=.6,text opacity=1, text=happya] at (mr22) {$\colword {w''}a$};
	\draw[thick,->-=.5, happya](mr) to[out=-45, in=45](s2); 
	
	\node[below, happyb] at (bottom) {$\colword {w_2}b$};
	\node[below, happyb] at (mid2) {$\colword {w_2}b$};
	\node[above, happyb] at (top1) {$\colword {w_2}b$};
	
		\draw[fill=black] (p0) circle (.04);
	\draw[fill=black] (p2) circle (.04);
	\draw[fill=black] (ml) circle (.04);
	
	\draw[fill=black] (r0) circle (.04);
	\draw[fill=black] (r2) circle (.04);
	\draw[fill=black] (mr) circle (.04);
	
\end{tikzpicture}
		\caption{Filling of a coloured diagram for the coloured commutator $[\colword{w_1}a, \colword{w_2}b]$ in the case where $\supp{w_1}$ is reducible}
		\label{fig:square-peanut}
	\end{figure}

	For $i \in \range 4$, let $u_i = \pushdown[h_i]{\boldword_i}$ be the alternating word that labels the corresponding bounded region of $\pushdown[h]{\bolddiag}$, where $h_1=h_4 = h$, $h_2 = h + \height(w_1)$, and $h_3 = h + \height(w_2)$.
	By \cref{area-monochromatic}, for $i \in \set{1,3}$, we have $\flatarea{u_i} \leq 12 \cdot \length{w_1}^2$.

	Since $\Lambda = \supp{w_1}$ decomposes as a join $\Lambda' * \Lambda''$, the intersection  $\Pal{w''} \cap \Pal{w_2}$ contains $\Lambda'$ so it is non-empty. Therefore, it follows from \cref{commutator-common-colour} that
	\[
		\flatarea{u_2} \leq C_2 \cdot ( \length{w''} + \length{w_2} + \abs {h_2})^{\alpha-1}.
	\]
	for sufficiently large $C_2>0$. By symmetry, the same holds for $u_4$, namely, $\flatarea{u_4} \leq C_2 \cdot (\length{w'} + \length{w_2} + \abs{h_4})^{\alpha-1}$.

	All together, we have $\flatarea{u_i} \leq C_3 \cdot \length{u_i}^{\alpha-1}$ for sufficiently large $C_3>0$. A straightforward application of \cref{sparse-pushdown} with $M=3$ gives us a uniform bound $C_M$ on the density of the pushdown of the coloured diagram for $\boldword$. Thus, it follows from \cref{lem:area-of-protodiagram} that
	\[
		\flatarea{\pushdown[h]{\boldword}} \leq C_3 \cdot C_M^2 \cdot \length{\pushdown[h]{\boldword}}^{\alpha-1} \leq C \cdot (\length{w_1} + \length{w_2} + \abs h)^{\alpha-1}
	\]
	for sufficiently large $C>0$.
\end{proof}

\begin{lemma}\label{commutator-cubic}
	There exists a constant $C > 0$ such that the following holds. Let $\mathbf w = [\colword {w_1} a, \colword {w_2} b]$ be a coloured commutator.  Assume that $\Lambda_1$ and $\Lambda_2$ are subgraphs of $\Gamma$ spanning a join with $\supp {w_1} \cup \set{v_b} \subseteq \Lambda_1$ and $\supp{w_2} \cup \set{v_a} \subseteq \Lambda_2$, such that either $\Lambda_1$ or $\Lambda_2$ is connected. Then for all $h \in \Z$, we have
	\[
		\flatarea{\pushdown[h]\boldword} \leq C \cdot \left( \length{w_1} + \length{w_2} + \abs h \right)^3.
	\]
\end{lemma}

\begin{proof}

	Assume without loss of generality that $\Lambda_1$ is connected.
	Let $w_1 =s_1^{\epsilon_1} \dots s_k^{\epsilon_k}$ with  $\epsilon_i \in \{\pm 1\}$ for all $i \in \range k$. A coloured diagram for the coloured commutator $\boldword = [\colword {w_1} a, \colword {w_2} b]$ can be constructed using a total of $4k +1$ bounded regions; see \cref{fig:miniplunder}:
	\begin{itemize}
		\item  $k + 1$ bounded regions whose labels are coloured bigons of the form $\colword{w_2}b \colword{w_2}{s_1}^{-1}$ (corresponding to the top region in \cref{fig:miniplunder}), of the form $\colword{w_2}{s_i} \colword{w_2}{s_{i+1}}^{-1}$ for $i \in \range {k-1}$, and of the form $\colword{w_2}{s_k} \colword{w_2}b^{-1}$ (corresponding to the bottom region in \cref{fig:miniplunder}); 

		\item  $2k$ bounded regions whose labels are coloured bigons of the form $\colword{s_i}{a}^{\epsilon_i} \colword{s_i}{s_i}^{-\epsilon_i}$ and $\colword{s_i}{s_i}^{\epsilon_i} \colword{s_i}{a}^{-\epsilon_i}$ (corresponding to the regions on the left and right sides in \cref{fig:miniplunder}, respectively); 

		\item $k$ bounded regions whose labels are null-homotopic monochromatic words of the form $[\colword{s_i}{s_i}^{\epsilon_i}, \colword{w_2}{s_i}]$ for $i \in \range k$ (corresponding to the regions with four sides in \cref{fig:miniplunder}). 
	\end{itemize}

	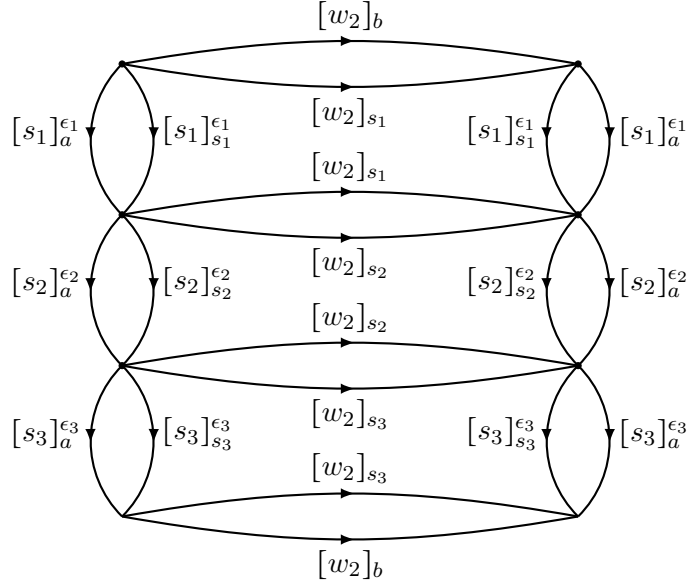
\begin{figure}
		\centering
		\begin{tikzpicture}
	
	\definecolor{happyaa}{RGB}{136,34,85}
	\definecolor{happybb}{RGB}{17,119,51}
	\definecolor{happyss1}{RGB}{254,97,0}
	\definecolor{happyss2}{RGB}{100,143,255}
	\definecolor{happyss3}{RGB}{255,176,0}
	
	\tikzset{|->-|/.style={decoration={markings,
			mark=at position 0 with {\arrow[ultra thin]{|}},
			mark=at position #1 with {\arrow{latex}},
			mark=at position 1 with {\arrow[ultra thin]{|}},
		},postaction={decorate}}}
	\tikzset{nodes={fill=white,fill opacity=0,text opacity=1}}
	\tikzset{->-/.style={decoration={markings,mark=at position #1+.01 with {\arrow{latex}}},postaction={decorate}}}
	\tikzset{-<-/.style={decoration={markings,mark=at position #1-.01 with {\arrowreversed{latex}}},postaction={decorate}}}

	\def\b{3};

	\pgfmathsetmacro{\bb}{\b+1}

    \foreach \i in {1,...,\b} {
    	\draw[thick,->-=.5, happyaa] (0,-2*\i+2) to[out=-135, in=135] (0,-2*\i);
    	\draw[thick,->-=.5] (0,-2*\i+2) to[out=-45, in=45] (0,-2*\i);
    	\draw[thick,->-=.5, happyaa] (6,-2*\i+2) to[out=-45, in=45] (6,-2*\i);
    	\draw[thick,->-=.5] (6,-2*\i+2) to[out=-135, in=135] (6,-2*\i);
    
    	\node[happyaa] at (-1,-2*\i+1.1){$\colword{s_{\i}}{a}^{\epsilon_\i}$};
    	\node at (1,-2*\i+1.1){$\colword{s_{\i}}{s_{\i}}^{\epsilon_\i}$};
	    	
	    \node[happyaa] at (7,-2*\i+1.1){$\colword{s_{\i}}{a}^{\epsilon_\i}$};
    	\node at (5,-2*\i+1.1){$\colword{s_{\i}}{s_{\i}}^{\epsilon_\i}$};
	}
	
	\foreach \i in {1,...,\bb} {
	    \def\shift{1}
        \pgfmathtruncatemacro{\ii}{\i-\shift}
        
        \ifthenelse{\i=1}{
        	\draw[thick,->-=.5, happybb] (0,-2*\i+2) to[out=10, in=170] node[midway,above]{$\colword{w_2}b$} (6,-2*\i+2);
			\draw[thick,->-=.5] (0,-2*\i+2) to[out=-10, in=-170] node[midway,below]{$\colword{w_2}{s_\i}$} (6,-2*\i+2);
		}{
	     	\ifthenelse{\i=\bb}{
				\draw[thick,->-=.5, happybb] (0,-2*\i+2) to[out=-10, in=-170] node[midway,below]{$\colword{w_2}b$} (6,-2*\i+2);
				\draw[thick,->-=.5] (0,-2*\i+2) to[out=10, in=170] node[midway,above]{$\colword{w_2}{s_\b}$} (6,-2*\i+2);
			}{
	    
	    		\draw[thick,->-=.5] (0,-2*\i+2) to[out=10, in=170] node[midway,above]{$\colword{w_2}{s_{\ii}}$} (6,-2*\i+2);
	    		\draw[thick,->-=.5] (0,-2*\i+2) to[out=-10, in=-170] node[midway,below]{$\colword{w_2}{s_{\i}}$} (6,-2*\i+2);
			}
		}
	    \draw[fill=black] (0,-2*\i+2) circle (.04);
    	\draw[fill=black] (6,-2*\i+2) circle (.04);
	}
\end{tikzpicture}
		\caption{Filling of the coloured diagram for the coloured commutator $[\colword{w_1}a, \colword{w_2}b]$, where $w_1 = s_1^{\epsilon_1} s_2^{\epsilon_2} \dots s_k^{\epsilon_k}$.}
		\label{fig:miniplunder}
	\end{figure}

	Let $h_{0} = h$ and $h_{i} = h + \height(s_1^{\epsilon_1} \cdots s_i^{\epsilon_i})$ for $i \in \range k$. The pushdown of each of these $4k+1$ pieces can be filled with at most quadratic almost-flat area:
	\begin{itemize}
		\item Since $\Lambda_1$ is connected, it follows from \cref{bigon-quadratic-combinatorial-path} that there exists $C_1>0$ such that the $h$-pushdown of the coloured bigon $\colword{w_2}b \colword{w_2}{s_1}^{-1}$ has almost-flat area at most $C_1 \cdot (\length{w_2} + \abs h)^2$, the $h_i$-pushdown of the coloured bigon $\colword{w_2}{s_i} \colword{w_2}{s_{i+1}}^{-1}$ has almost-flat area at most $C_1 \cdot (\length{w_2} + \abs {h_{i}})^2$ for $i \in \range{k-1}$, and the $h_k$-pushdown of the coloured bigon $\colword{w_2}{s_k} \colword{w_2}b^{-1}$ has almost-flat area at most $C_1 \cdot (\length{w_2} + \abs {h_{k}})^2$.

		\item For all $i \in \range k$, it follows from \cref{bigon-quadratic-combinatorial-path} that there exists $C_2>0$ such that the $h_i$-pushdown of the coloured bigon $\colword{s_i}{a}^{\epsilon_i}\colword{s_i}{s_i}^{-\epsilon_i}$ has almost-flat area at most $C_2 \cdot (1 + \abs{h_{i-1}} )^2$. Similarly, we have that the almost-flat area of the $(h_i + \height{(w_2)})$-pushdown of the coloured bigon $\colword{s_i}{s_i}^{\epsilon_i} \colword{s_i}{a}^{-\epsilon_i}$ is at most $C_2 \cdot (1 + \abs{h_{i-1} + \height(w_2)} )^2$ for all $i \in \range k$.

		\item Finally, \cref{area-monochromatic} implies that for all $i \in \range k$, the $h_i$-pushdown of the null-homotopic monochromatic word $[\colword{s_i}{s_i}^{\epsilon_i}, \colword{w_2}{s_i}]$ has almost-flat area at most $12 \cdot (1 + \length{w_2})^2$.
	\end{itemize}
	By setting $C_3 = \max\set{C_1, C_2, 16}$, we obtain that $\flatarea{\pushdown[h]{\boldword}}$ is bounded above by
	\[
		C_3 \left( \sum_{i=0}^{k} (\length{w_2} + \abs{h_i})^2 + \sum_{i=1}^k  (1 + \abs{h_{i-1}} )^2 + (1 + \abs{h_{i-1} + \height(w_2)} )^2 + (1 + \length{w_2} + \abs{h_i})^2 \right).
	\]
	Since for all $i \in \range[0]k$, we have $\abs{h_i}, \abs{h_i + \height(w_2)} \leq \length{w_1} + \length{w_2} + \abs{h}$, and $k = \length{w_1}$, we obtain
	\begin{align*}	\flatarea{\pushdown[h]\boldword} & \leq C_3  \left( \length{w_1} (\length{w_2} + \length{w_1} + \abs h)^2 + 3  \length{w_1} (1 + \length{w_1} + \length{w_2} + \abs h)^2	 \right) \\
                                                & \leq C  \left( \length{w_1} + \length{w_2} + \abs h \right)^3
	\end{align*}
	for sufficiently large $C>0$.
\end{proof}

We now conclude this section showing that coloured commutators satisfy \cref{upper-bounds}.

\begin{proposition}\label{area-coloured-commutator}
	There exists a constant $C > 0$ with the following property. Assume that $\alpha \in \set{3,4}$ and that $\Gamma$ does not have $\D\alpha$. Let $h \in \Z$. The almost-flat area of the $h$-pushdown of the coloured commutator $\boldword =[\colword {w_1} a, \colword {w_2} b]$ satisfies
	\[
		\flatarea{\pushdown[h]\boldword} \leq C \cdot \length{\pushdown[h]\boldword}^{\alpha-1}.
	\]
\end{proposition}
\begin{proof}
	Let $\boldword = [\colword {w_1} a, \colword {w_2} b]$ be a coloured commutator.
	Since $\supp{w_1} \cup \set{v_b} \subgraph \Pal{w_2}$ and $\supp{w_2} \cup \set{v_a} \subgraph \Pal{w_1}$, we may assume that
	\[(\supp{w_1} \cup \set{v_b}) \cap (\supp{w_2} \cup \set{v_a}) = \emptyset,\]
	otherwise, $\Pal{w_1} \cap \Pal{w_2} \neq \emptyset$, and the statement would follow from \cref{commutator-common-colour}.

	In particular, the subgraph induced by $\supp{w_1} \cup \supp{w_2} \cup \set{v_a, v_b}$ is reducible, and is contained in a maximal reducible subgraph $\Lambda \subgraph \Gamma$. Let $\Lambda = \Lambda_1 \ast \dots \ast \Lambda_k$ for $k \geq 2$, with each $\Lambda_i$ irreducible.

	Suppose that $\supp{w_1} \cup \set{v_b}$ is reducible. Then, either $\supp{w_1}$ is reducible or $v_b \in \Pal{w_1}$, and we conclude by \cref{commutator-3joins} or \cref{commutator-common-colour}, respectively. So, we can assume that $\supp{w_1} \cup \set{v_b}$ is irreducible and, symmetrically, that the same holds for $\supp{w_2} \cup \set{v_a}$.
	This implies that $\supp{w_1} \cup \set{v_a} \subgraph \Lambda_i$ and $\supp{w_2} \cup \set{v_a} \subgraph \Lambda_j$ for some $i,j \in \range k$. We can also assume that $k=2$ and that $\set{i,j}=\range 2$. If this were not the case, we would have $\Pal{w_1} \cap \Pal{w_2} \neq \emptyset$, and the statement would follow from \cref{commutator-common-colour}. Therefore, we have $\Lambda = \Lambda_1 * \Lambda_2$ with $\supp {w_1} \cup \set{v_b} \subgraph \Lambda_1$ and $\supp{w_2} \cup \set{v_a} \subgraph \Lambda_2$.

	If $\Gamma$ does not have \D4, then $\flag{\Lambda}$ is simply connected and either $\Lambda_1$ or $\Lambda_2$ is connected. In this case, it follows from \cref{commutator-cubic} that $\flatarea{\pushdown[h]\boldword} \leq C \cdot \length{\pushdown[h]\boldword}^3$ for some $C>0$.
	If $\Gamma$ does not have \D3, then $\Lambda$ is not essentially $2$-reducible, so either $\Lambda_1$ or $\Lambda_2$ is a single vertex. In this case, the desired almost-flat area estimate now follows from \cref{commutator-common-colour}.
\end{proof}

\subsection{Putting the pieces together}

Now that we have proven the upper bounds for the pushdown of the fundamental pieces, we combine them to obtain upper bounds for progressively more complicated words. We begin with the pushdown of null-homotopic $k$-coloured words whose underlying words admit van Kampen diagrams with well-behaved corridors. For now, we also allow constants to depend on the number of colours.

\NewDocumentCommand{\leftword}{s m}{\IfBooleanTF{#1}{\colword{u_{#2}^{<}}{a_{#2}}}{u_{#2}^{<}}}
\NewDocumentCommand{\Leftword}{s m}{\IfBooleanTF{#1}{\colword{u_{#2}^{<<}}{a_{#2}}}{u_{#2}^{<<}}}
\NewDocumentCommand{\Rightword}{s m}{\IfBooleanTF{#1}{\colword{u_{#2}^{>>}}{a_{#2}}}{u_{#2}^{>>}}}
\NewDocumentCommand{\rightword}{s m}{\IfBooleanTF{#1}{\colword{u_{#2}^{>}}{a_{#2}}}{u_{#2}^{>}}}

\begin{lemma}\label{subdivide-k-gon}
	For every $k \geq 2$, there exists a constant $C_k>0$ such that the following holds. Let $\alpha \in \set{3,4}$ and assume that $\Gamma$ does not have $\D\alpha$. Let $\boldword = \colwordproduct k$ be a null-homotopic $k$-coloured word with $a_k \neq a_1$, and $h \in \Z$. Assume that the underlying word $ w=w_1 \cdots w_k $ admits a minimal-area van Kampen diagram $D$ such that, for all $i \in \range k$, each corridor starting from an edge of $ w_i $ ends in an edge of either $w_{i-2}, w_{i-1}, w_i, w_{i+1}$, or $w_{i+2}$ (indices intended mod $k$).
	Then
	\[
		\flatarea{\pushdown[h]\boldword} \leq C_k \cdot \length{\pushdown[h]\boldword}^{\alpha-1}.
	\]
\end{lemma}

\begin{proof}
	We aim to construct a coloured diagram for $\boldword$ whose bounded regions are fundamental pieces, and whose density can be bounded in terms of the number of colours.

	For every $i \in \range k$, we choose a freely reduced coloured word $\colword{w_i'}{a_i}$ such that $w_i \groupid{\raag\Gamma} w_i'$, and we choose a van Kampen diagram $D'$ for $w_1' \cdots w_k'$ such that corridors starting from an edge in $w_i'$ end in an edge of $ w'_{i-2}$, $ w'_{i-1}$, $ w'_i$, $ w'_{i+1}$ or $w'_{i+2} $. We make these choices so that the area of $D'$ is the smallest possible. Note that it is always possible to make such choices, as one can always take $w'_i = w_i$ and $D' = D$.

	Two corridors in $D'$ that originate from the same ${w_i'}$ cannot intersect. Otherwise, consider an innermost pair of intersecting corridors starting from ${w_i'}$; these corridors must begin at adjacent letters. Exchanging those letters would yield a word $w_i''$ such that $w'_1 \cdots w''_i \cdots w'_k$ admits a diagram with smaller area, a contradiction. Similarly, a corridor cannot connect two edges of the same $w_i'$. Since corridors starting from the same $w'_i$ do not intersect, we can decompose each ${w_i'}$ as
	\[
		{w_i'} = \leftword i \Leftword i \Rightword i \rightword i
	\]
	in such a way that the corridors of $D'$ starting in $\rightword i$ end in $\leftword{i+1}$, and those starting in $\Rightword i$ end in $\Leftword{i+2}$. In particular, $\rightword i = (\leftword{i+1})^{-1}$ and $\Rightword{i}=(\Leftword{i+2})^{-1}$, and all letters of $\Rightword i $ commute with all letters of $ \Leftword{i+1}$.

	\begin{figure}
		\begin{tikzpicture}[scale=2, line width=0.20mm]
	\usetikzlibrary{calc}
	\usetikzlibrary{decorations.markings}
	\tikzset{nodes={fill=white,fill opacity=.6,text opacity=1}}
	\tikzset{->-/.style={decoration={markings,mark=at position #1 with {\arrow{latex}}},postaction={decorate}}}
	\tikzset{-<-/.style={decoration={markings,mark=at position #1-.05 with {\arrowreversed{latex}}},postaction={decorate}}}

	\def\n{5}

	\def\outr{4}

	\def\midr{2.8}

	\def\inr{2}

	\def\bigonang{30}

	\foreach \ii in {1,...,\n} {
			\def\ang{360/\n};
			\def\curang{360/\n*\ii+\ang/2+90};
			\draw[thick, ->-=.45] (\curang:\outr) -- (\curang+\ang:\outr);
			\draw[thick, ->-=.8] (\curang:\outr) to[out=\curang+180-\bigonang,in=\curang+\bigonang] (\curang:\midr);
			\draw[thick, -<-=.8] (\curang:\outr) to[out=\curang+180+\bigonang,in=\curang-\bigonang] (\curang:\midr);
			\draw[thick, ->-=.8] (\curang:\midr) -- (\curang+\ang/2:\inr);
			\draw[thick, -<-=.8] (\curang+\ang:\midr) -- (\curang+\ang/2:\inr);
			\draw[thick, ->-=.7] (\curang+0.5*\ang:\inr) to[out=\curang+\ang/2+180-\bigonang, in=\curang+\ang/2+\bigonang] (0,0);
			\draw[thick, -<-=.7] (\curang+0.5*\ang:\inr) to[out=\curang+\ang/2+180+\bigonang, in=\curang+\ang/2-\bigonang] (0,0);
		};
	\foreach \ii in {1,...,\n} {
			\def\ang{360/\n};
			\def\curang{360/\n*\ii+\ang/2+90};
			\node at ($(\curang:\outr)!0.5!(\curang+\ang:\outr)+(\curang+\ang/2:.2)$) {$\colword{w_\ii}{a_\ii}$};
			\node at ($(\curang:\outr)!0.5!(\curang:\midr)+(\curang+90:.3)$) {$\leftword*\ii$};
			\node at ($(\curang+\ang:\outr)!0.5!(\curang+\ang:\midr)+(\curang+\ang-90:.3)$) {$\rightword*\ii$};
			\node at ($(\curang:\midr)!0.5!(\curang+\ang/2:\inr)$) {$\Leftword*\ii$};
			\node at ($(\curang+\ang:\midr)!0.5!(\curang+\ang/2:\inr)$) {$\Rightword*\ii$};
			\node at ($(0,0)!0.5!(\curang-\ang/2:\inr)+(\curang-\ang/2+90:.33)$) {$\Leftword*\ii$};
			\node at ($(0,0)!0.5!(\curang+1.5*\ang:\inr)+(\curang+1.5*\ang-90:.33)$) {$\Rightword*\ii$};
		};
\end{tikzpicture}
		\caption{Filling of a coloured $k$-gon with coloured bigons and coloured commutators.}
		\label{fig:fill-k-coloured}
	\end{figure}
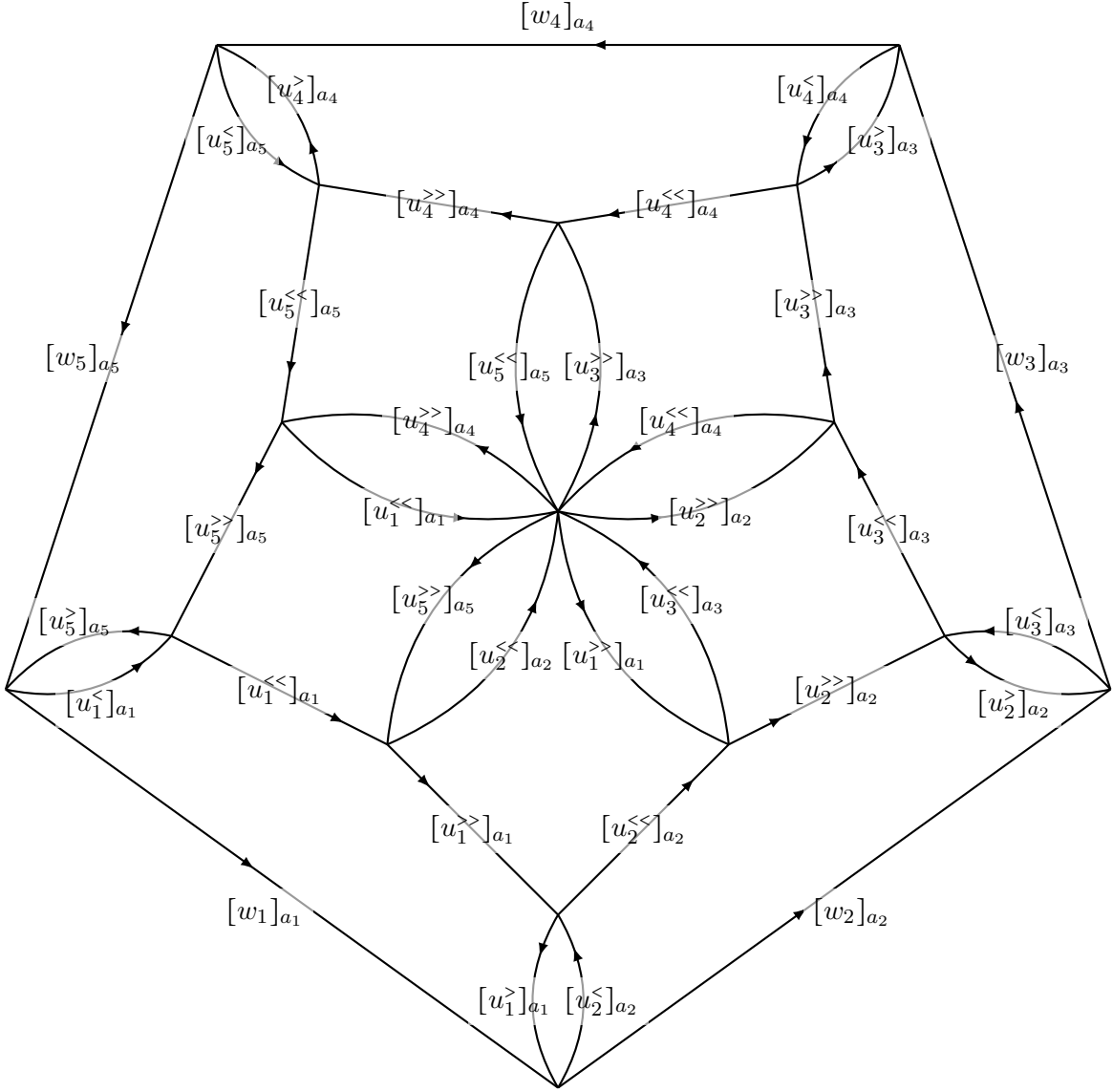

	A coloured diagram $\bolddiag$ for $\boldword$ can be constructed, as shown in \cref{fig:fill-k-coloured}, using the following fundamental pieces:
	\begin{itemize}
		\item $k$ monochromatic regions labelled by
		      \[
			      \colword{w_i}{a_i} \colword{w_i'}{a_i}^{-1} = \colword{w_i}{a_i} \rightword* i ^{-1} \Rightword*i^{-1} \Leftword* i ^{-1} \leftword* i ^{-1};
		      \]
		\item $k$ bigons labelled by $ \leftword* i \rightword* {i-1}^{-1} $;
		\item $k$ bigons labelled by $ \Leftword*{i+1} \Rightword*{i-1} $;
		\item $k$ coloured commutators $ [\Rightword* i, \Leftword*{i+1}] $.
	\end{itemize}

	In this way, we obtain a coloured diagram $\bolddiag$ which has $3k+1$ polychromatic vertices, including the $k$ $\partial$-polychromatic vertices that are the vertices of the $k$-gon. By applying \cref{sparse-pushdown} with $M=3k+1$, the density of $\pushdown[h] \bolddiag$ is bounded above by some constant $C_k'>0$.

	Since every word $w$ labelling the bounded regions of $\pushdown[h]{\mathbf D}$ has almost-flat area bounded above by $C \cdot \length w^{\alpha-1}$ by \cref{area-coloured-bigon,area-coloured-commutator}, for some constant $C>0$, we conclude by applying \cref{lem:area-of-protodiagram}.
\end{proof}

By an induction argument, we can remove the hypothesis on the corridors in \cref{subdivide-k-gon}.

\begin{proposition}\label{almost-flat area-k-coloured}
	Let $\alpha \in \set{3,4}$ and assume that $\Gamma$ does not have $\D\alpha$. For every $k \geq 1$, there exists a constant $ C_k $ such that for every null-homotopic $k$-coloured word $\boldword$ and every $ h \in \Z $, we have
	\[
		\flatarea {\pushdown[h]{\boldword}}  \leq C_k \cdot \length{\pushdown[h]\boldword}^{\alpha-1}.
	\]
\end{proposition}

\begin{proof}

	We proceed by induction on $k$. The case $ k=1 $ was handled in \cref{area-monochromatic}. Suppose that the statement holds for $\ell<k$.

	By a cyclic conjugation, we may assume $\boldword=\colwordproduct k$ with $a_k \neq a_1$. Let $D$ be a van Kampen diagram for the underlying word $w=w_1 \cdots w_k$. Assume that there is a corridor in $D$ connecting a letter of $w_i$ with a letter of $w_j $, with $3 \leq \abs{i-j} \leq \abs{n-3}$; otherwise, we are under the hypotheses of \cref{subdivide-k-gon} and can conclude directly. 

	By cutting along the boundary of the corridor, we obtain a coloured diagram $\bolddiag$ with only two regions, labelled by a $( \abs{i-j} + 2 )$-coloured word and a $( n-\abs{i-j} + 2 )$-coloured word respectively; both have fewer than $k$ colours. We have $\density\bolddiag \leq 2$, and by \cref{sparse-pushdown}, there is a constant $C>0$ independent of $k$ such that $\density {\pushdown[h]\bolddiag} \leq C$.

	By combining the inductive hypothesis with \cref{lem:area-of-protodiagram}, we obtain that the statement holds with $ C_k \coloneq \max\set{C_\ell : \ell < k} \cdot C^{\alpha-1} $.
\end{proof}

Finally, we remove the dependence of the multiplicative constant $C_k$ on the number of colours. The key to this is the following lemma.

\begin{lemma}\label{cut-triangle-15-coloured}
	There exists a constant $C>0$ such that the following holds. Let $ \mathbf w_1$, $\mathbf w_2$, and $\mathbf w_3$ be efficient coloured words such that $ \mathbf w = \mathbf w_1 \mathbf w_2 \mathbf w_3$ is null-homotopic. Then there exists a coloured diagram $ \mathbf D $ for $ \mathbf w $ whose bounded regions are at most $ 15 $-coloured, and such that $\density {\pushdown[h]{\mathbf D}} \leq C$. 
\end{lemma}

\begin{proof}
	We argue using \cref{cut-efficient-diagram} to construct a coloured diagram for the null-homotopic coloured word $\mathbf w_1 \mathbf w'^{-1}$, where $\boldword'^{-1} = \boldword_2 \boldword_3$. This yields a coloured diagram $\bolddiag'$ such that $\density {\pushdown[h]{\bolddiag'}} \leq C'$ for some constant $ C'>0$, and whose bounded regions are labelled by words of the form $\mathbf u'_1 \mathbf w'_2 \mathbf w'_3$, where $\mathbf u'_1$  is at most $5$-coloured, and $\mathbf w'_2$ and $ \mathbf w'_3$ are (possibly empty) subwords of $\mathbf w_2$ and $\mathbf w_3$, respectively; they are, in particular, efficient.

	By \cref{lem:fill-sparse-with-sparse}, it suffices to show that we can find a coloured diagram for the labels of the bounded regions of $\bolddiag'$, whose regions are at most $15$-coloured and such that the density of the $h$-pushdown of $\bolddiag'$ is at most $C$ for some constant $C>0$ independent of $h$.

	To do so, we perform two more iterations of \cref{cut-efficient-diagram}. First, we apply it to each bounded region of $\bolddiag'$ to obtain coloured diagrams $\bolddiag''$ for the words $\mathbf w'_2 \cdot (\mathbf w'_3 \mathbf u'_1)$, such that $\density {\pushdown[h']{\bolddiag''}} \leq C'$ for every $h' \in \Z$, and whose regions are labelled by words of the form $\mathbf u''_2 \mathbf w''_3 \mathbf u''_1$, where $ \mathbf u''_2$ is at most $5$-coloured, $\mathbf w''_3$ is a subword of $ \mathbf w'_3 $, and $ \mathbf u''_1$ is a subword of $\mathbf u'_1$ and therefore is at most $5$-coloured. We replace each bounded region of $\bolddiag'$ with the corresponding coloured diagram $\bolddiag''$.

	Second, we apply \cref{cut-efficient-diagram} once more to each bounded region of the coloured diagram constructed above, obtaining coloured diagrams $\bolddiag'''$ for the words $\boldword_3'' \cdot (\mathbf{u}_1'' \mathbf{u_2}'')$ such that $\density{\pushdown[h'']{\bolddiag'''}} \leq C''$ for every $h'' \in \Z$, and whose regions are labelled by words of the form $\mathbf{u}_3''' \mathbf{u}_1''' \mathbf{u}_2'''$, where $\mathbf{u}_3'''$ is at most $5$-coloured, $\mathbf{u}_1'''$ is a subword of $\mathbf{u}_1''$, and $\mathbf{u}_2'''$ is a subword of $\mathbf{u}_2''$. We again replace each bounded region with the corresponding coloured diagram $\bolddiag'''$.

	In this way, we obtain a coloured diagram for $\mathbf w$ whose regions are labelled by words that are at most $15$-coloured and whose pushdown has density at most $C'^3$ by \cref{lem:fill-sparse-with-sparse}.
\end{proof}

We conclude the section by establishing the upper bound for the Dehn function of $\bbg \Gamma$.

\begin{proof}[Proof of \cref{upper-bounds}]
	For every $g \in \bbg\Gamma $, choose a word $u_g$ representing $g$ that is the $0$-pushdown of an efficient coloured word $ \mathbf w_g $. By \cref{coloured-norm-vs-bbg-norm}, the length of $ u_{g} $ is bounded linearly in terms of $ \flatnorm{g} $, so we may apply \cref{triangle-lemma}. Now, it suffices to estimate the area of the word $ u=u_{g_1}u_{g_2}u_{g_3} $, where $g_1,g_2,g_3\in\bbg\Gamma$ such that $ g_1g_2g_3 = 1 $.

	To do so, consider the coloured diagram $\bolddiag$ produced by \cref{cut-triangle-15-coloured} for the null-homotopic word $ \mathbf w = \mathbf w_{g_1} \mathbf w_{g_2} \mathbf w_{g_3} $. Its $0$-pushdown is an alternating diagram for the null-homotopic word
	\[
		\pushdown{\boldword} = \pushdown{\boldword_{g_1}} \pushdown{\boldword_{g_2}} \pushdown{\boldword_{g_3}} = u;
	\]
	here, we use the fact $ \height(\boldword_{g_i})=0 $ and apply \cref{pushdown-properties}. The density $\density{\pushdown\bolddiag}$ is bounded above by a constant $ C $, and the regions of $\bolddiag$ are labelled by alternating words $u_i$, which are the $h_i$-pushdown of coloured words that are at most $15$-coloured. By \cref{almost-flat area-k-coloured}, each $u_i$ admits an almost-flat diagram of area at most $ C_{15} \cdot \length{u_i}^{\alpha-1} $. We conclude by applying \cref{lem:area-of-protodiagram}.
\end{proof}

\section{Lower bounds}\label{sec:lower-bounds}

We now turn our attention to the lower bounds. The aim of this section is to prove the following.

\begin{theorem}\label{lower-bounds}
	Let $\Gamma$ be a finite simplicial graph such that the associated flag complex $\flag\Gamma$ is simply connected, and let $\alpha \in \set{3,4}$. If $\Gamma$ has property $\D \alpha$, then $\Dehn{\bbg\Gamma}(n) \succcurlyeq n^\alpha$.
\end{theorem}

To establish cubic and quartic lower bounds for the Dehn functions in \cref{lower-bounds}, we need to produce a family of null-homotopic alternating words $ w_n $ such that we can find a cubic, respectively quartic, lower bound for the area of almost-flat van Kampen diagrams for $w_n$. To obtain these lower bounds, we analyse the behaviour of corridors and annuli inside such diagrams.

Recall that both corridors and annuli inherit a natural orientation: our convention is that the positive orientation of an $a$-corridor or $a$-annulus is the one such that the $a$-edges transverse to it are oriented from left to right.

Let $D$ be an almost-flat van Kampen diagram over the standard presentation $\raagpres$ of $\raag \Gamma$. Removing the interior of a corridor or annulus subdivides $D$ into two connected components. We say that a corridor $C$ \emph{separates} two vertices $ p $ and $q$ if they belong to different connected components after removing the interior of $C$, and that an annulus $A$ \emph{encloses} a vertex $p$ if $p$ is contained in the connected component that does not contain the boundary of $D$.

Let $p_0$ denote the base point of $D$. For every vertex $ p \in D $, we define
\def\cheight{\kappa}
\def\annheight{\alpha}
\begin{align*}
	\cheight_+(p)   & \coloneq \cardinality{\set{C : C \ \text{is a corridor separating $p_0$ and $p$, with $p$ on the right of $C$} }} \\
	\cheight_-(p)   & \coloneq \cardinality{\set{C  : C \ \text{is a corridor separating $p_0$ and $p$, with $p$ on the left of $C$}} } \\
	\annheight_+(p) & \coloneq \cardinality{\set{A: A \ \text{is an annulus oriented clockwise enclosing $p$}} }                        \\
	\annheight_-(p) & \coloneq \cardinality{\set{A: A \ \text{is an annulus oriented counterclockwise enclosing $p$}} }
\end{align*}
and set
\[
	\annheight(p) \coloneq \annheight_+(p) - \annheight_-(p) \quad
	\text{and} \quad \cheight(p) = \cheight_+(p) - \cheight_-(p).
\]

\begin{lemma}\label{heights}
	For every vertex $p \in D$, we have $ \height(p) = \annheight(p) + \cheight(p)$.
\end{lemma}

\begin{proof}
	Consider an arbitrary combinatorial path connecting $p_0$ and $p$. Every edge in the path crosses exactly one corridor or annulus. The height increases by one when crossing a corridor from left to right, when entering an annulus oriented clockwise, or when exiting an annulus oriented counterclockwise. Conversely, the height decreases by one when crossing a corridor from right to left, when exiting an annulus oriented clockwise, or when entering an annulus oriented counterclockwise. From this, we get the statement.
\end{proof}

\begin{corollary}\label{many-annuli}
	For every vertex $p \in D$, we have $\abs{\annheight(p)} \geq \abs{\cheight(p)} - 2$.
\end{corollary}
\begin{proof}
	The result follows from \cref{heights} and the fact that every vertex $p$ in an almost-flat van Kampen diagram satisfies $ \height(p) \in \set{0,1,2} $.
\end{proof}

Thus, \cref{many-annuli} tells us that in an almost-flat van Kampen diagram, a vertex $p$ with large $\abs{\cheight(p)}$ must be enclosed by many annuli. We show that, under certain assumptions, these annuli produce a lot of crossings, which in turn yield a large lower bound for the area of $D$.

\newcommand{\enclosing}[1]{\Theta(#1)}
\newcommand{\notenclosing}[1]{\Omega(#1)}
\newcommand{\crossings}[1]{x(#1)}
\begin{definition}
	Let $p$ be a vertex in $D$. We denote by $ \crossings p $ the set of pairs $ (A,A') $, where $A$ is an annulus enclosing $p$, $A'$ is an annulus \emph{not} enclosing $p$, and $A$ crosses $A'$.
\end{definition}

\begin{lemma}\label{crossing-annuli-lower-bound}
	Let $p$ and $q$ be vertices of $D$, and let $ \gamma $ be a (combinatorial) path connecting them. Assume that no annulus encloses both $p$ and $q$, and that every vertex in $\gamma$ is enclosed by at least $k$ annuli of $D$. Then \[
		\cardinality{\crossings p} \geq \frac{k(k+1)}{2}.
	\]
\end{lemma}
\begin{proof}
	Assume that $\gamma \colon [0,\ell] \to D$ is parametrised by length, with $ \gamma(0)=p$ and $\gamma(\ell)=q $. For a vertex $p'$ of $D$, denote by $\enclosing{p'}$ the set of annuli enclosing $p'$, and by $\notenclosing {p'}$ the set of annuli not enclosing $p'$.

	For $A \in \enclosing p$, we call the \emph{exit time} of $A$ the number
	\[
		\epsilon(A) := \min\set{t \in \N \cap [0,\ell] \colon A \in \notenclosing{\gamma(s)} \text{ for } s \in \N \cap [t, \ell]}.
	\]
	That is, the number $\epsilon(A)$ is the time after which $ \gamma $ last crosses $A$. Note that this number is well-defined since $ A \in \notenclosing q$. Moreover, an annulus $A \in \enclosing p$ crosses, by definition, the edge between $ \gamma(\epsilon(A)-1) $ and $ \gamma(\epsilon(A)) $, so different annuli have distinct exit times.

	Choose the ordering $ \enclosing p = \set{A_1, \dots, A_n} $ according to the exit times so that $ \epsilon(A_i) < \epsilon(A_j) $ whenever $ i<j $. For example, if the annuli are pairwise disjoint, then $ A_1 $ is the innermost annulus, and $ A_n $ the outermost. For $ i \in \range n $, let $p_i\coloneq \gamma(\epsilon(A_i)) $ denote the exit point of $A_i$.

	By the chosen ordering, we have $A_j \in \notenclosing{p_i}$ for $i\geq j$. In particular, $ \cardinality{\notenclosing{p_i} \cap \enclosing{p}} \geq {i-1} $, and therefore $ \cardinality{\enclosing{p_i} \cap \enclosing{p}} \leq n-(i-1) $.
	By the assumption $ \cardinality{\enclosing{p_i}} \geq k $, it follows that $ \cardinality{\enclosing{p_i} \cap \notenclosing{p}} \geq k - n + i-1 $.
	Every annulus $ A' \in \enclosing{p_i} \cap \notenclosing{p} $ must cross $ A_i $, since $p$ is enclosed by $ A_i $ but not by $ A' $, $ p_i $ is enclosed by $ A' $ but not by $ A_i $, and the vertex after $p_i$ (moving towards $p$) is enclosed by both $A'$ and $A_i$. Therefore, $ (A_i, A') \in \crossings p $.

	Since $n \geq k$, we conclude that $\crossings p$ contains at least $ k+(k-1)+\dots + 1 = \frac{k(k+1)}2 $ pairs of crossing annuli.
\end{proof}

We have the following consequence.

\begin{corollary}\label{estimate-area-by-family-of-points}
	Assume that there exists a family $P = \set{p_1, \dots, p_m}$ of vertices in $D$ such that no annulus encloses two or more vertices of $P$. Furthermore, suppose that for every $i \in \range m$, there is some integer $ k_i \geq 0 $ and a vertex $q_i$ in $D$ such that:
	\begin{itemize}
		\item no annulus encloses both $ p_i $ and $q_i$;
		\item there exists a path connecting $p_i$ and $q_i$ such that every vertex of the path is enclosed by at least $k_i$ annuli.
	\end{itemize}
	Then
	\[
		\area D \geq \frac12 \sum_{i=1}^m{k_i^2}.
	\]
\end{corollary}

\begin{proof}
	The area of $D$ coincides with the total number of crossings between annuli and/or corridors. In particular, it is bounded below by the cardinality of
	\[
		X \coloneq \set{(A, A'): A \neq A' \text{ are annuli and } A \text{ crosses } A'}.
	\]
	Note that we do not divide by two to take the symmetry into account, since if two annuli cross, they must do so at least twice.

	Denote by $ L_1, \dots, L_m $ the families of annuli enclosing $ p_1, \dots, p_m $, respectively. By assumption, the sets $L_i$ are pairwise disjoint, so $ x(p_1), \dots, x(p_m) \subseteq X $ are also pairwise disjoint. By \cref{crossing-annuli-lower-bound}, we have $ \cardinality{x(p_i)} \geq \frac{k_i^2}2 $ for every $ i \in \range m $. Therefore, we obtain $ \cardinality X \geq \frac{1}2 \sum_{i=1}^m k_i^2 $, as desired.
\end{proof}

\subsection{Cubic lower bound}\label{sec:cubic-lower-bound}

Throughout this subsection, we assume that $ \Gamma $ has property \ref{D3}; that is, it admits a maximal reducible subgraph $\Lambda$ that is essentially $2$-reducible. In particular, the subgraph $\Lambda$ decomposes as a join $\Lambda' \ast \Lambda''$, where $ \Lambda'$ and $\Lambda''$ are irreducible and each has at least two vertices. We aim to show that, in this case, the Dehn function of $ \bbg \Gamma $ is bounded below by a cubic polynomial.

\begin{remark}
	The results in this subsection also apply when $\Gamma$ has \D4, since \D4 implies \D3. In fact, the proof of the quartic lower bound in \cref{sec:quartic-lower-bound} builds on the work done for the cubic lower bound in \cref{sec:cubic-lower-bound}, using the additional hypothesis given by property \ref{D4} to promote the cubic lower bound into a quartic one.
\end{remark}

Denote by $\mathcal A$ and $\mathcal B$ the sets of generators corresponding to the vertices of $ \Lambda' $ and $ \Lambda'' $, respectively, and by $ \mathcal C \coloneq \raaggens \setminus  (\mathcal A \cup \mathcal B) $ the generators  corresponding to the vertices outside the join.

We begin by noting the following fact.

\begin{lemma}\label{star-of-graph-is-other-graph}
	Let $s \in \raaggens$ be a generator. If $s$ commutes with all $ a \in \mathcal A $, then $s \in \mathcal B$. Similarly, if $s$ commutes with all $b \in \mathcal B $, then $s \in \mathcal A$.
\end{lemma}

\begin{proof}
	We prove the first assertion; the other follows by symmetry. Suppose that $s\in\raaggens$ commutes with all $a\in\mathcal{A}$, so the vertex $v_s$ belongs to $\cap_{a\in\mathcal{A}}\Star{v_a}$. If $v_s=v_a$ for some $a\in\mathcal{A}$, then $\Lambda'$ would be a cone graph, contradicting the fact that $\Lambda'$ is irreducible. If $v_s\neq v_a$ for all $a\in\mathcal{A}$, then $v_s$ must belong to $\Lambda''$; otherwise, the graph $\Lambda'\ast(v_s\cup\Lambda'')$ would be a reducible graph containing $\Lambda$, contradicting the maximality of $\Lambda$. Thus, we have $v_s\in\Lambda''$, and therefore $s\in\mathcal{B}$.
\end{proof}

\cref{star-of-graph-is-other-graph} has an immediate corollary.

\begin{corollary}\label{no-cone-point}
	There is no generator that commutes with all $ s \in \mathcal A \cup \mathcal B $.
\end{corollary}

\begin{proof}
	Such a generator should belong to both $ \mathcal A $ and $ \mathcal B $, but $ \mathcal A \cap \mathcal B = \emptyset $.
\end{proof}

\begin{lemma}\label{non-commuting-sequence}
	There exist $ k \in \N$ and a sequence $ a_0, a_1, \dots, a_k=a_0 $ of non-necessarily distinct elements of $\mathcal A$ such that $ \set{a_1, \dots, a_k} = \mathcal A $, and $ [a_{i-1},a_{i}] \neq 1 $ for all $ i \in \range k $. Similarly, there exist $ \ell \in \N$ and a sequence $ b_0, b_1, \dots, b_\ell=b_0 $ of non-necessarily distinct elements of $\mathcal B$ such that $ \set{b_1, \dots, b_\ell} = \mathcal B $, and $ [b_{j-1},b_{j}] \neq 1 $ for all $ j \in \range \ell $.
\end{lemma}

\begin{proof}
	We prove the first statement; the other follows from a similar argument. Since $ \Lambda' $ is irreducible, its complement $ \compl{(\Lambda')} $ is connected. Thus, the graph $ \compl{(\Lambda')} $ admits a combinatorial closed path $ \gamma \colon [0, k] \to \compl{(\Lambda')}$ (that is, $\gamma(k)=\gamma(0)$) that visits all the vertices at least once. Notice that we do not require the path $\gamma$ to be injective, for instance, it may backtrack. Let $ a_0, a_1, \dots, a_k \in \mathcal A $ be the sequence of generators corresponding to the vertices of $\gamma$, so that $ \gamma(i) = v_{a_i} $. Since consecutive vertices on $\gamma$ are not adjacent in $\Lambda'$, the corresponding generators do not commute.
\end{proof}

Let $ a_1, \dots, a_k = a_0 \in \mathcal A$ and $ b_1, \dots, b_\ell = b_0 \in \mathcal B $ be any choice of elements satisfying the assumptions of \cref{non-commuting-sequence}. For $ n \in \N $, consider the alternating words
\begin{align*}
	w_n'  & \coloneq (a_1 b_1^{-1} \cdots a_k b_1^{-1})^n (b_1 a_1^{-1} \cdots b_1 a_k^{-1})^n,          \\
	w''_n & \coloneq (b_1 a_1^{-1} \cdots b_{\ell} a_1^{-1})^n(a_1 b_1^{-1} \cdots a_1 b_{\ell}^{-1})^n.
\end{align*}
Notice that $ w'_n $ and $w''_n$ represent elements in $\bbg{\Lambda'}$ and $\bbg{\Lambda''}$, respectively. Therefore, these two words commute; that is, the commutator $[w_n', w''_n]$ is a null-homotopic alternating word.

Let $D_n$ be an almost-flat van Kampen diagram for the null-homotopic alternating word
\[
	w_n \coloneq [w_n', w''_n].
\]
The diagram $D_n$ is pictured as a square, where the top and bottom sides are labelled by $ w'_n $, and the left and right sides are labelled by $w''_n$; see \cref{fig:grid}. In this subsection, we prove that the area of $D_n$ is bounded below by a cubic polynomial.

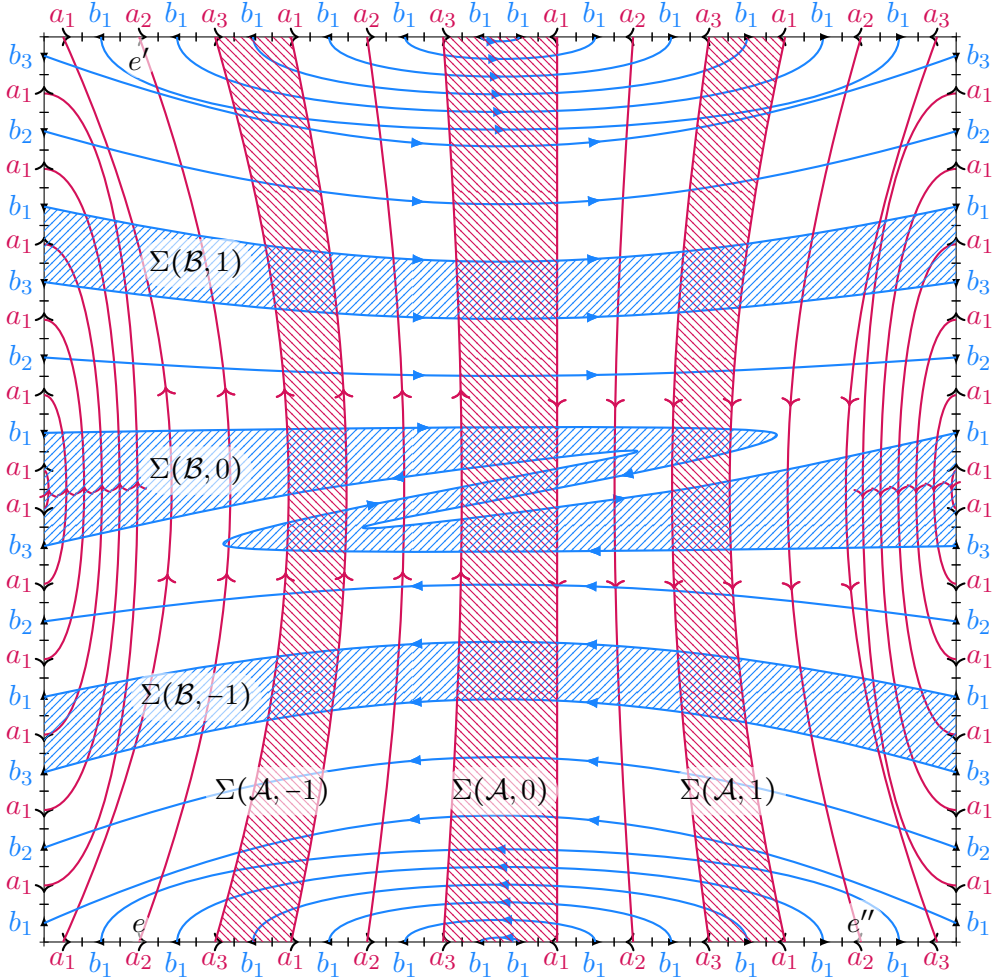
\begin{figure}
	\centering
	\begin{tikzpicture}[scale=.5]
	\usetikzlibrary{decorations.markings}
	\definecolor{happyblue}{RGB}{26,133,255}
	\definecolor{happypink}{RGB}{212,17,89}
	\tikzset{|->-|/.style={decoration={markings,
			mark=at position 0 with {\arrow[ultra thin]{|}},
			mark=at position #1 with {\arrow[thick]{>}},
			mark=at position 1 with {\arrow[ultra thin]{|}},
		},postaction={decorate}}}
	\tikzset{->-/.style={decoration={markings,mark=at position #1+.01 with {\arrow{>}}},postaction={decorate}}}
	\tikzset{-<-/.style={decoration={markings,mark=at position #1-.01 with {\arrowreversed{>}}},postaction={decorate}}}
	\tikzset{|-|>-|/.style={decoration={markings,
			mark=at position 0 with {\arrow[ultra thin]{|}},
			mark=at position #1 with {\arrow{latex}},
			mark=at position 1 with {\arrow[ultra thin]{|}},
		},postaction={decorate}}}
	\tikzset{-|>-/.style={decoration={markings,mark=at position #1+.01 with {\arrow{latex}}},postaction={decorate}}}
	\tikzset{-<|-/.style={decoration={markings,mark=at position #1-.01 with {\arrowreversed{latex}}},postaction={decorate}}}

	\def\n{2}
	\def\k{3}
	\def\squarelength{4*\n*\k}
	\coordinate (O) at (\squarelength/2,\squarelength/2);

	\foreach \ii in {1, ..., \n} {
			\foreach \jj in {1, ..., \k} {
					\foreach \ll in {0, 1} {
							\def\t{\ll*2*\k*\n + 2*\ii*\k + 2*\jj - 2 - 2*\k};

							\draw[|->-|=.7] (0, \t+2-2*\ll) -- (0, \t+1);
							\draw[|-|>-|=.7] (0, \t+2*\ll) -- (0, \t+1);
							\draw[|->-|=.7] (\squarelength, \t + 2 - 2*\ll) -- (\squarelength, \t+1);
							\draw[|-|>-|=.7] (\squarelength, \t + 2*\ll) -- (\squarelength, \t+1);
							\draw[|-|>-|=.7] (\t+2-2*\ll, 0) -- (\t+1, 0);
							\draw[|->-|=.7] (\t+2*\ll, 0) -- (\t+1, 0);
							\draw[|-|>-|=.7] (\t+2-2*\ll, \squarelength) -- (\t+1, \squarelength);
							\draw[|->-|=.7] (\t+2*\ll, \squarelength) -- (\t+1, \squarelength);

							\node[left,happyblue] at (0, \t + .5 + \ll) {$b_\jj$};
							\node[left,happypink] at (0, \t + 1.5 - \ll) {$a_1$};
							\node[right,happyblue] at (\squarelength, \t + .5 + \ll) {$b_\jj$};
							\node[right,happypink] at (\squarelength, \t + 1.5 - \ll) {$a_1$};
							\node[below,happypink] at (\t + .5 + \ll,0) {$a_\jj$};
							\node[below,happyblue] at (\t + 1.5 - \ll, 0) {$b_1$};
							\node[above,happypink] at (\t + .5 + \ll,\squarelength) {$a_\jj$};
							\node[above,happyblue] at (\t + 1.5 - \ll, \squarelength) {$b_1$};

							\def\inclinationangle{(\t-(\squarelength/2))*2}
							\ifthenelse{\ll=0}{
								\draw[happypink, ->-=.4, ->-=.6, thick, name path global=happypink-\ii-\jj-\ll-path] (\t+.5+\ll, 0) to[in={270-(\inclinationangle)},out=90+(\inclinationangle)] (\t+.5+\ll, \squarelength);
							}{
								\draw[happypink, -<-=.4, -<-=.6, thick, name path global=happypink-\ii-\jj-\ll-path] (\t+.5+\ll, 0) to[in={270-(\inclinationangle)},out=90+(\inclinationangle)] (\t+.5+\ll, \squarelength);
							}
							\ifthenelse{\jj=\k \AND \ii=\n \AND \ll=0}{
								\draw[happyblue, -<|-=.21, -<|-=.5, -<|-=.8, thick, name path global=happyblue-\ii-\jj-\ll-path] plot[smooth] coordinates {
										(0*\squarelength, \t+.5+\ll)
										($(.3*\squarelength, \t+.5+\ll)+(0,1.5)$)
										($(.65*\squarelength, \t+.5+\ll)+(0,2.5)$)
										($(.2*\squarelength, \t+.5+\ll)+(0,0)$)
										(1*\squarelength, \t+.5+\ll)
									};
							}{
								\ifthenelse{\jj=1 \AND \ii=1 \AND \ll=1}{
									\draw[happyblue, -|>-=.21, -|>-=.5,-|>-=.8, thick, name path global=happyblue-\ii-\jj-\ll-path] plot[smooth] coordinates {
											(0*\squarelength, \t+.5+\ll)
											($(.8*\squarelength, \t+.5+\ll)+(0,0)$)
											($(.35*\squarelength, \t+.5+\ll)+(0,-2.5)$)
											($(.7*\squarelength, \t+.5+\ll)+(0,-1.5)$)
											(1*\squarelength, \t+.5+\ll)
										};

								}{
									\ifthenelse{\ll=1}{
										\draw[happyblue, -|>-=.41, -|>-=.6, thick, name path global=happyblue-\ii-\jj-\ll-path] (0, \t+.5+\ll) to[in={180+(\inclinationangle)},out=-(\inclinationangle)] (\squarelength, \t+.5+\ll);
									}{
										\draw[happyblue, -<|-=.41, -<|-=.6, thick, name path global=happyblue-\ii-\jj-\ll-path] (0, \t+.5+\ll) to[in={180+(\inclinationangle)},out=-(\inclinationangle)] (\squarelength, \t+.5+\ll);
									}
								}
							}

							\ifthenelse{\ll=0}{
								\draw[happypink, thick, ->-=.5] (0,\t+1.5) to[out=0, in=0, looseness=0.4] (0,{\squarelength-(\t)-1.5});
								\draw[happyblue, thick, -<|-=.5] ({\t+1.5},0) to[out=90,in=90,looseness=0.4] ({{\squarelength-(\t)-1.5}},0);
								\draw[happypink, thick, -<-=.5] (\squarelength,\t+1.5) to[out=180, in=180, looseness=0.4] (\squarelength,{\squarelength-(\t)-1.5});
								\draw[happyblue, thick, -|>-=.5] ({\t+1.5},\squarelength) to[out=270,in=270,looseness=0.4] ({{\squarelength-(\t)-1.5}},\squarelength);
							}
						}
				}
			\ifthenelse{\ii=1}{}{
				\foreach \ll in {0,1} {
						\tikzfillbetween[of=happypink-\ii-1-\ll-path and happypink-\the\numexpr\ii-1\relax-\k-\ll-path]{pattern=north west lines, pattern color=happypink};
						\tikzfillbetween[of=happyblue-\ii-1-\ll-path and happyblue-\the\numexpr\ii-1\relax-\k-\ll-path]{pattern=north east lines, pattern color=happyblue};
					}
			}
		}
	\tikzfillbetween[of=happypink-\n-\k-0-path and happypink-1-1-1-path]{pattern=north west lines, pattern color=happypink};
	\tikzfillbetween[of=happyblue-\n-\k-0-path and happyblue-1-1-1-path]{pattern=north east lines, pattern color=happyblue};

	\tikzset{nodes={fill=white,fill opacity=.6,text opacity=1,rounded corners=10}}
	\node[below] at (2.5,\squarelength) {$e'$};
	\node[above] at (2.5,0) {$e$};
	\node[above] at (\squarelength - 2.5,0) {$e''$};
	\node at (6,4) {$\Sigma(\mathcal A, -1)$};
	\node at (12,4) {$\Sigma(\mathcal A, 0)$};
	\node at (18,4) {$\Sigma(\mathcal A, 1)$};
	\node at (4,6.5) {$\Sigma(\mathcal B, -1)$};
	\node at (4,12.5) {$\Sigma(\mathcal B, 0)$};
	\node at (4,18) {$\Sigma(\mathcal B, 1)$};
\end{tikzpicture}
	\caption{A schematic depiction of a van Kampen diagram for $w_n$, where $ k=\ell=3 $ and $n=2$. The coloured lines represent corridor cores. We call \emph{bridges} the corridors that connect opposite sides of the square, and \emph{half-annuli} the corridors that start and end on the same side of the square. The regions between an $ a_k $-bridge and an $ a_1 $-bridge are called $ \mathcal A $-gaps, and the regions between a $ b_\ell $-bridge and a $ b_1 $-bridge are called $ \mathcal B $-gaps.}
	\label{fig:grid}
\end{figure}

\begin{remark}
	Following the notation in \cref{sec:coloured-words}, we have
	\begin{align*}
		w'_n  & = \pushdown{\colword{(a_1 \cdots a_k)^n(a_1^{-1} \cdots a_k^{-1})}{b_1}}        \\
		w''_n & = \pushdown{\colword{(b_1 \cdots b_\ell)^n(b_1^{-1} \cdots b_\ell^{-1})}{a_1}}.
	\end{align*}
	So, the word $w_n=[w'_n,w''_n]$ is the pushdown of the coloured commutator $[\colword{w'_n}{b_1}, \colword{w''_n}{a_1}]$ (recall \cref{coloured-commutators}). While this fact is not used in the following arguments, it is still worthy of note: we have seen in \cref{sec:upper-bounds} that upper bounds on the almost-flat area of the pushdowns of fundamental pieces (monochromatic words, coloured bigons, and coloured commutators) yield an upper bound on the Dehn function of $\bbg \Gamma$. This suggests that the Dehn functions of $\bbg\Gamma$ should be witnessed by a family of such words.
\end{remark}

\begin{definition}
	We call a corridor of $D_n$ an \emph{$ \mathcal A $-corridor}, \emph{$ \mathcal B $-corridor} or \emph{$ \mathcal C $-corridor}, if it is an $ s $-corridor for some $ s \in \mathcal A, \mathcal B, \mathcal C $ respectively.
	Define \emph{$ \mathcal A $-annulus}, \emph{$ \mathcal B $-annulus} and \emph{$ \mathcal C $-annulus} similarly.
\end{definition}

\begin{remark}
	Since $ w_n $ is a word in $ \mathcal A \cup \mathcal B $, there are no $ \mathcal C $-corridors in $D_n$; however, $ \mathcal C $-annuli may be present.
\end{remark}

The properties described in \cref{non-commuting-sequence} place strong restrictions on how $ w_n $ can be transformed into the trivial word by applying relations. This allows us to describe precisely the behaviour of corridors inside $D_n$.

\begin{definition}
	A corridor in $D_n$ is called a \emph{bridge} if it connects two $i$th letters on the opposite sides of the square (counting from left to right and from bottom to top).
	It is called a \emph{half-annulus} if it connects the $i$th letter with the $(4n-i)$th letter on the same side of the square.
\end{definition}

\begin{lemma}\label{corridors-in-square}
	Every $ \mathcal A $-corridor in $D_n$ is either a half-annulus that starts and ends on the left or right side of the square, or a bridge connecting the top and bottom sides. Analogously, every $\mathcal B$-corridor in $D_n$ is either a half-annulus that starts and ends on the top or bottom side of the square, or a bridge connecting the left and right sides; see \cref{fig:grid}.
\end{lemma}

\begin{proof}
	We prove the first statement; the second statement follows from a similar argument.
	Consider the $j$th leftmost $a_i$-edge $e$ on the bottom side of $D_n$, with $2 \leq i \leq n$. The $ \mathcal A $-corridor starting at $e$ must end at another $a_i$-edge with opposite orientation. This corridor divides the other edges of $D_n$ into two sets, one on each side of the corridor. Since two $a_i$-corridors cannot cross, the total number of boundary $a_i$-edges on each side of the corridor, counted with sign, must be zero. To satisfy this condition, there are only two possibilities: the corridor must end either at the $j$th leftmost $a_i$-edge $e'$ on the top side, or at the $j$th rightmost $a_i$-edge $e''$ on the bottom side.

	Suppose, by contradiction, that there is a corridor $C$ connecting $e$ to $e''$. Then there are $(n-j)$ many $a_{i-1}$-edges below $C$ oriented to the right and $(n-j+1)$ many $a_{i-1}$-edges below $C$ oriented to the left. Since $a_{i-1}$ and $a_i$ do not commute, every corridor that starts at a letter $a_{i-1}$ below $C$ cannot intersect $C$. Thus, there is no way to pair all the $a_{i-1}$ edges below $C$. Therefore, a corridor that starts at $e$ must end at $e'$, that is, the corridor $C$ is a bridge.

	For $ a_1 $-edges, the reasoning is similar: let $e$ be an $a_1$-edge on the bottom side that is not the leftmost one. There is only one other edge with the same label that can be reached by a corridor $C$ without intersecting any $ a_2 $-corridors or $a_k$-corridors, namely the $a_1$-edge directly above on the top side. Thus, the corridor $C$ is a bridge.

	The leftmost $a_1$-edges on the top and bottom sides, together with all the $a_1$-edges on the left and right sides, are oriented in such a way that the only possible $a_1$-corridors that do not pairwise intersect and do not intersect any $a_2$-corridors or $a_k$-corridors are bridges and half-annuli, respectively, as shown in \cref{fig:grid}.
\end{proof}

\cref{corridors-in-square} shows that bridges in $D_n$ look almost like a grid pattern, in the sense that $ \mathcal A $-bridges are parallel and do not intersect, and the same holds for $ \mathcal B $-bridges. Every $\mathcal A$-bridge crosses every $\mathcal B$-bridge at least once, and possibly multiple times, as shown in \cref{fig:grid}.
Moreover, the half-annuli must stay ``close'' to the boundary, as they cannot cross a bridge: for example, the bottom half-annuli labelled by $b_1$ would need to cross a $b_1$-bridge, and the top ones would have to cross a $b_k$-bridge, and both are not allowed.
This gives a complete description of the corridors inside $D_n$.

We now proceed to investigate the behaviour of the annuli. To this end, we introduce the following definition.

\begin{definition}
	For $ -(n-1) \leq i \leq n-1 $, define the \emph{$\mathcal A$-gap}, denoted by $\Sigma(\mathcal A, i)$, to be the region between the $(n+i)$th $a_k$-bridge and the $(n+i+1)$th $a_1$-bridge, counted from left to right. Similarly, define the \emph{$\mathcal B$-gap}, denoted by $\Sigma(\mathcal B, j)$, to be the region between the $(j+n)$th $b_{\ell}$-bridge and the $(j+n+1)$th $b_1$-bridge, counted from bottom to top.
\end{definition}

\begin{lemma}\label{corridor-touching-two-gaps}
	Let $A$ be an annulus that intersects two distinct $ \mathcal A $-gaps. Then $A$ is a $ \mathcal B $-annulus. Similarly, if it intersects two distinct $ \mathcal B $-gaps, then it is an $ \mathcal A $-annulus.
\end{lemma}
\begin{proof}
	Let $s \in \raaggens$. If $A$ is an $s$-annulus intersecting two distinct $\mathcal A$-gaps, it must also cross every $\mathcal A$-bridge in between. Then $s$ commutes with all the generators of $\mathcal A$, so it belongs to $\mathcal B$ by \cref{star-of-graph-is-other-graph}. The argument for the second statement is similar.
\end{proof}

For $i,j \in \Z$ with $\abs i, \abs j \leq n-1 $, denote by $R_{i,j}$ the intersection $\Sigma(\mathcal A,i) \cap \Sigma(\mathcal B, j)$, which is non-empty as it contains at least intersections of the sides of the bridges. However, it may be disconnected. For example, in \cref{fig:grid}, the region $ R_{0,0} $ has three connected components.

\begin{lemma}\label{corridor-height-of-region}
	For every vertex $p \in R_{i,j}$, we have $ \cheight(p) = k (n-\abs{i}) +\ell (n-\abs{j})$
\end{lemma}

\begin{proof}
	The bottom and top half-annuli, which are the $b_1$-corridors, do not intersect $\mathcal B$-gaps, as to do so they would need to intersect a $b_i$-bridge for every $i \in \range {\ell}$. Similarly, the left and right half-annuli do not intersect $\mathcal A$-gaps. In particular, they cannot separate $p$ from the base point, so only the bridges contribute to $\cheight(p)$. The statement now follows from the definition of $\cheight$.
\end{proof}

For $\abs i\leq n-1$, choose a point $p_i$ inside $R_{i,i}$. We verify in the following two lemmas that the family $P=\set{p_i : \abs i\leq n-1}$ satisfies the two conditions in the hypotheses of \cref{estimate-area-by-family-of-points}.

\begin{lemma}\label{distinct-annuli-cubic-case}
	No annulus can enclose both $p_i$ and $p_j$ for $i \neq j$.
\end{lemma}

\begin{proof}
	If such an annulus exists, then it would intersect two distinct $\mathcal A$-gaps and two distinct $ \mathcal B $-gaps. By \cref{corridor-touching-two-gaps}, it would be both an $ \mathcal A $-annulus and a $ \mathcal B $-annulus, which is a contradiction.
\end{proof}

\begin{lemma}\label{path-inside-grid}
	Let $ p \in R_{i,j} $ be a vertex, with $ \abs{i}, \abs{j} \leq n-2$. There exists $q \in R_{i', j'}$, for some $ i', j' $ satisfying $ \abs{i-i'}=\abs{j-j'}=1 $, and a (combinatorial) path connecting $p$ and $q$ that intersects only  the four gaps $ \Sigma(\mathcal A, i)$, $\Sigma(\mathcal A, i')$, $\Sigma(\mathcal B, j)$, and $\Sigma(\mathcal B, j')$. In particular, for every point $x$ on the path between $p$ and $q$, we have $ \abs{\cheight(x)} \geq k (n-\abs {i} - 1) + \ell(n-\abs j - 1)$.
\end{lemma}

\begin{proof}
	Consider a path, starting from $p$ and ending on the boundary of $D_n$, that is entirely contained in $ \Sigma(\mathcal A,i) $. This path must intersect a $ \mathcal B $-gap different from $ \Sigma(\mathcal B, j) $. Let $p'$ be the first vertex inside such a gap $ \Sigma(\mathcal B, j') $. Clearly, $ \abs{j-j'}=1 $.

	Now, repeat the same argument by taking a path inside $ \Sigma(\mathcal B, j') $ starting at $p'$ and ending on the boundary of $D_n$; let $q$ be the first vertex on this path that lies in a different gap $\Sigma(\mathcal A, i') $. Then, $ \abs{i-i'}=1 $. The union of the two paths, from $ p $ to $p'$ and from $ p' $ to $q$, intersects only the four gaps in the statement. Since this path lies entirely within these four gaps, the estimate for $ \cheight(x) $ follows analogously to the proof of \cref{corridor-height-of-region}.
\end{proof}

Now, we can apply \cref{estimate-area-by-family-of-points} to obtain the cubic lower bound.

\begin{proposition}
	If $\Gamma$ has property \ref{D3}, then $\Dehn{\bbg \Gamma}(n) \succcurlyeq n^3$.
\end{proposition}

\begin{proof}
	We prove that there exists a constant $ C>0 $, independent on the choice of the almost-flat diagram $D_n$ for $w_n$, such that the area satisfies
	\[
		\area{D_n} \geq \frac 1C n^3 - C.
	\]

	Consider the family $ P=\set{p_i\in R_{i,i} : \abs i \leq n-1 } $. \cref{distinct-annuli-cubic-case} says that there are no annuli enclose $p_i$ and $p_j$ for $i\neq j$. By \cref{path-inside-grid}, for each $i$, we get a point $q_i\in R_{i \pm 1, i \pm 1}$ and a path from $p_i $ to $q_i$, and every point $x$ on this path is enclosed by at least
	\[
		\annheight_+(x) + \annheight_-(x) \geq \abs{\annheight(x)} \geq \abs{\cheight(x)} - 2 \geq (k+\ell)(n-\abs{i}-1)-2
	\]
	many annuli. Thus, the family $P$ satisfies the hypotheses of \cref{estimate-area-by-family-of-points}, and applying the conclusion yields
	\begin{align*}
		\area{D_n} & \geq  \frac12 \sum_{i=-(n-1)}^{n-1} \left((k+\ell)(n-\abs{i} -1)-2\right)^2 \\
		           & \geq \frac1C n^3 -C
	\end{align*}
	for some  constant $ C>0 $ sufficiently large and independent of $n$.
\end{proof}

\subsection{Quartic lower bound}\label{sec:quartic-lower-bound}

In this section, we assume that $ \Gamma $ has property \ref{D4}; that is, it has a maximal reducible subgraph $\Lambda \subgraph \Gamma$ such that $\flag\Lambda$ is not simply connected. Thus, the subgraph $\Lambda$ decomposes as a join $\Lambda' * \Lambda''$, where $\Lambda'$ and $\Lambda''$ are both disconnected. In particular, the subgraph $\Lambda$ is essentially $2$-reducible.

In this case, we may choose the generators $a_1, \dots, a_k$ and $ b_1, \dots, b_\ell $ that satisfy the hypotheses of \cref{non-commuting-sequence}, and with the additional constraint that the vertices $ v_{a_1}$ and $v_{a_k}$ in $\Lambda'$, corresponding to $ a_1 $ and $ a_k $, lie in different connected components of $ \Lambda' $, and the same holds for $v_{b_1}$ and $v_{b_\ell}$ in $\Lambda''$.

With this additional assumption, we show that the area of an almost-flat van Kampen diagram $D_n$ for the word $w_n$, defined in \cref{sec:cubic-lower-bound}, has a quartic lower bound. We begin with the following standard topological result.

\begin{lemma}\label{vertical-or-horizontal-line}
	Let $C$ be a $2$-dimensional contractible cell complex, and let $A \sqcup B$ be a partition of its vertices. Assume that for every $2$-dimensional cell $\sigma$ of $C$, the full subcomplexes of $ \sigma $ whose vertices belong to $A$ and to $B$ are connected or empty.
	Let $\gamma \colon S^1 \to C $ be a combinatorial loop (not necessarily simple) in $C$, obtained by concatenating four combinatorial paths $\gamma_1$, $\gamma_2$, $\gamma_3$, and $\gamma_4$.
	Then, either there exists a combinatorial path connecting the images of $\gamma_1$ and $\gamma_3$, with all vertices in $A$, or there exists a combinatorial path connecting the images of $\gamma_2$ and $\gamma_4$, with all vertices in $B$.
\end{lemma}

\begin{remark}
	Note that \cref{vertical-or-horizontal-line} does not assume that the vertices of $\gamma_1$, $ \gamma_3$ belong to $A$, nor that the vertices of $\gamma_2$, $ \gamma_4$ belong to $B$. However, a counterexample cannot be constructed by requiring $\gamma_1$, $\gamma_3$ to be proper subsets of $B$ and $\gamma_2$, $\gamma_4$ to be proper subsets of $A$, since $(\gamma_1 \cup \gamma_3) \cap (\gamma_2 \cup \gamma_4) \neq \emptyset$ (it contains the four points where the paths are concatenated).
\end{remark}

\begin{proof}
	Suppose that this is not the case. We claim that there exists $f \colon C \to S^1$ such that $f \circ \gamma$ is homotopic to a homeomorphism. This leads to a contradiction, as $C$ being contractible implies that the composition $f \circ \gamma$ must be null-homotopic. For the following, we view $ S^1 $ as the unit circle in $ \mathbb C$.

	To define $f$, we first note that a vertex $v \in A$ cannot be connected inside $A$ to both $\gamma_1$ and $\gamma_3$. For $v \in A$, define $f(v) = 1$ if it is connected to $\gamma_1$, and $f(v) = -1$ if it is connected to $\gamma_3$ or to neither. Similarly, for $ v \in B $, define $f(v)=i$ if $v$ is connected to $\gamma_2$, and $f(v) = -i$ otherwise.

	If two vertices $v$ and $v'$ share an edge, then $f(v) \neq - f(v')$. So there is a canonical way to extend $f$ to the $1$-skeleton by choosing the shortest path in $S^1$ connecting the images. Then, if $\sigma$ is a $2$-cell, by hypothesis the vertices $ \partial \sigma \cap A$ are all in the same connected component, so $\restrict f{\partial \sigma}$ is not surjective. Therefore, the map $f$ can be extended to $ \sigma $ continuously.

	It remains to show that $ f \circ \gamma  $ is homotopic to a homeomorphism. This follows since $f$ never assumes the value $-1$ on $\gamma_1$, $-i$ on $\gamma_2$, $1$ on $ \gamma_3 $, or $ i $ on $\gamma_4$, so $f$ is a degree-one map.
\end{proof}

We now use \cref{vertical-or-horizontal-line} to obtain information about the structure of the $\mathcal A$-gaps and $\mathcal B$-gaps.

\begin{lemma}\label{path-outside-annuli}
	In every $\mathcal A$-gap (respectively, $ \mathcal B $-gap) of $D_n$, there is a combinatorial path connecting the two opposite boundary components that does not cross any $\mathcal A$-annuli (respectively, $ \mathcal{B} $-annuli).
\end{lemma}

\begin{proof}
	\newcommand{\Comp}{\operatorname{Comp}}
	Let $p$ be a vertex in $D_n$ that is enclosed in an $ \mathcal A $-annulus. Consider the set of outermost $ \mathcal A $-annuli that enclose $p$, that is, those not enclosed by any larger $ \mathcal A$-annulus. Since they pairwise cross, all the corresponding generators must commute, so their associated vertices all belong to the same connected component of $ \Lambda ' $. Denote this connected component by $ \Comp(p) $. If $p$ and $p'$ are adjacent vertices in $D_n$ and each is enclosed by at least one $ \mathcal A $-annulus, then there is an $ \mathcal A $-annulus that encloses both, and therefore $ \Comp(p)=\Comp(p') $.

	We now apply \cref{vertical-or-horizontal-line} to the $ \mathcal A $-gap, where the partition of the vertices is given by those enclosed by some $ \mathcal A $-annulus and those that are not. Suppose by contradiction that there is no path that connects the boundary components of an $\mathcal{A}$-gap and lies outside every $ \mathcal A $-annulus. Then, there is a path $L$ connecting the $ a_k $-corridor to the $ a_1 $-corridor such that every vertex of $L$ lies in at least one $ \mathcal A $-annulus. Therefore, repeating the argument in the previous paragraph on the vertices of $L$ gives $ \Comp(p)=\Comp(q) $, where $p$ and $q$ are the endpoints of $L$.

	However, all $ \mathcal A$-annuli enclosing $p$ must cross the $ a_k $-corridor, so $ \Comp(p) $ is the connected component containing $ v_{a_k} $. Similarly, the connected component $ \Comp(q) $ contains $ v_{a_1} $. Since we assumed that $ v_{a_1} $ and $ v_{a_k} $ belong to different connected components of $ \Lambda' $, we obtain a contradiction.
\end{proof}

We apply \cref{path-outside-annuli} to the gaps $\Sigma( \mathcal A, i)$ and $ \Sigma(\mathcal B, j) $ and deduce the existence of two paths $\gamma \subseteq \Sigma(\mathcal A, i)$ and $\gamma' \subseteq \Sigma(\mathcal B,j)$, whose vertices are not enclosed by any $\mathcal A$-annuli and any $\mathcal B$-annuli, respectively. By intersecting $\gamma$ and $\gamma'$, we obtain a vertex $ p_{i,j} \in R_{i,j}=\Sigma(\mathcal A, i) \cap \Sigma(\mathcal B, j)$ that is not enclosed by any $\mathcal A$-annulus or any $\mathcal B$-annulus.

We are now ready to apply \cref{estimate-area-by-family-of-points} to establish the quartic lower bound.

\begin{proposition}
	If $\Gamma$ has property \D4, then $\Dehn{\bbg \Gamma}(n) \succcurlyeq n^4$.
\end{proposition}

\begin{proof}
	We claim that the area of the almost-flat van Kampen diagram $D_n$ for $w_n$ is bounded below by a quartic polynomial.	First, note that there is no annulus enclosing two distinct points $ p_{i,j} $ and $ p_{i',j'} $. Suppose, for instance, that $ i \neq i' $ (the case $j\neq j'$ is similar). Then an annulus enclosing both $ p_{i,j} $ and $ p_{i',j'} $ would have to intersect different $ \mathcal A $-gaps. By \cref{corridor-touching-two-gaps}, it must be a $ \mathcal B $-annulus, but this is ruled out by the definition of $ p_{i,j} $.

	By \cref{path-inside-grid}, we obtain a path connecting $ p_{i,j} $ to some other point $q_{i,j} \in R_{i\pm1, j\pm1}$, with every point along the path contained in at least $ k(n - \abs i - 1) + \ell(n - \abs j - 1) - 2$ many annuli. By the argument above, no annulus can enclose both $ p_{i,j} $ and $ q_{i,j} $. Therefore, we can apply \cref{estimate-area-by-family-of-points} to the family $ p_{i,j} $ and obtain
	\begin{align*}
		\area{D_n} & \geq  \frac12 \sum_{i=-(n-1)}^{n-1} \sum_{j=-(n-1)}^{n-1} \left(k(n-\abs{i} -1)+\ell(n-\abs j -1)-2\right)^2 \\
		           & \geq \frac1C n^4 -C
	\end{align*}
	for some constant $ C>0 $ sufficiently large and independent of $n$.
\end{proof}

\bibliographystyle{alpha}

\bibliography{bib}

\end{document}